\documentclass{article}
\usepackage[a4paper]{geometry}
\usepackage[english]{babel}
\usepackage[utf8]{inputenc}
\usepackage{amsmath}
\usepackage[toc,page]{appendix}
\usepackage{amsthm}
\usepackage[small]{titlesec}
\usepackage{mathtools}
\usepackage{graphicx}
\usepackage{mathrsfs}
\usepackage{amsfonts} 
\usepackage{amssymb}
\usepackage{xcolor}
\usepackage{esint}
\usepackage{relsize}
\usepackage{comment}
\usepackage{bigints}
\usepackage{dsfont}
\usepackage{geometry}
\usepackage{float}
\usepackage{bbm}

\numberwithin{equation}{section}
\theoremstyle{plain}
\newtheorem{thm}{Theorem}[section]

\newtheorem{lem}[thm]{Lemma}
\newtheorem{prop}[thm]{Proposition}

\newtheorem{mainresult}[thm]{Main Result}
\theoremstyle{definition}
\newtheorem{defn}{Definition}
\newtheorem{remark}{Remark}

\newenvironment{proofof}[1][Proof]{\noindent \textit{#1.} }{\ \qed}

\newcommand{\R}{\mathbb{R}^3}

\newcommand{\Rc}{\mathbb{R}^3}
\newcommand{\M}{\mathscr{M}_{+}\left(\Rc \right)}

\newcommand{\vp}{\varphi}
\newcommand{\Rl}{\mathbb{R}}

\newcommand{\adresse}{
  \begin{description}
   \item[A. Nota:] {Gran Sasso Science Institute,\\ Viale Francesco Crispi 7, 67100 L’Aquila, Italy  \\
E-mail: \texttt{alessia.nota@gssi.it}}

\item[J.~J.~L. Vel\'azquez:] { Institute for Applied Mathematics, University of Bonn, \\ Endenicher Allee 60, D-53115 Bonn, Germany\\
E-mail: \texttt{velazquez@iam.uni-bonn.de}}
  \end{description}
}

\theoremstyle{remark}

\title{Non Maxwellian long-time asymptotics  
for Homoenergetic Boltzmann flows under strong shear}

\author{Alessia Nota, Juan J. L. Vel\'azquez}
\begin{document}

\maketitle

\begin{abstract}
We compute, using  
matched asymptotic expansions, the long-time asymptotics of homoenergetic solutions to the nonlinear Boltzmann equation, in presence of a shear term, in the hyperbolic dominated regime, for homogeneous collision kernels for which there are infinitely many collisions as $\tau \to \infty$. The mass of the resulting solutions is concentrated in an increasing family of dyadic scales, something that makes the behaviour of these solutions very different from classical Maxwellian distributions. 
  
\end{abstract}

\bigskip

\tableofcontents

\section{Introduction}

We consider the nonlinear Boltzmann equation for homoenergetic flows. 
Specifically, let $g=g(t,x,v)$, $g: \R \times \R \times \Rl_+ \rightarrow  \mathbb{R}_+$, denote a solution of the spatially inhomogeneous Boltzmann equation 
\begin{equation}
  \partial_t g+v\cdot \partial_x g = Q(g,g), \qquad g=g(t,x,v) 
  \end{equation}
where
\begin{equation}\label{eq:collB_bis}
  Q(g,g)(v)=   \int_{\R}\int_{S^2} B(n\cdot \omega,|v-v_*|)(f'_*f'-f_*f)\, dv_* d \omega\,. 
\end{equation} 
Here $n=\frac{v-v_*}{|v-v_*|}$ is  the unit vector and $\omega\in S^2$ the scattering vector. We use the conventional notation in kinetic theory  $g'_*=g(v'_*),  \, g'=g(v'),\,g_*=g(v_*),\,g=g(v)$, 
and the post-collisional velocities $(v',v'_*)$ are related to the pre-collisional velocites $(v,v_*)$ by the classical collision rule
\begin{align}\label{eq:CollRule}
 & v'= v- \left((v-v_*)\cdot \omega \right)\omega \notag \\&
 v'_*=v_* +\left( (v-v_* )\cdot \omega \right) \omega\,.
\end{align} 

We focus on solutions $g=g(t,x,v)$ having the following equidispersive form 
\begin{equation}\label{eq:homansatz}
    g(t, x, v) = f (t, w), \quad w = v -\xi (t, x), \quad \text{with}\  \ \xi(t, x) = L(t)x= A(I + tA)^{-1} x
\end{equation} 
where $A\in M_{3,3,}(\mathbb{R})$ is a $3\times 3$ real matrix, e.g., see \cite{CercArchive,JNV1}.  
The characteristic feature of these solutions is the fact that the dispersion of velocities is the same at any point $x\in\R$ for any given time but the mean value of the velocity depends on the position and changes also with time. Moreover, 
every solution of with the form  yields only time-dependent internal energy and density $\mathcal{E} (t, x) = \mathcal{E} (t),$ $\rho (t, x) = \rho(t)$ which motivates the name ``homoenergetic''. We further notice that this class of solutions (i.e., those of the form \eqref{eq:homansatz}) is motivated by an invariant manifold of the classical molecular dynamics equations which are invariant under a suitable symmetry group 
(e.g., see  \cite{DJ1, DJ2, James}), which is inherited  by the Boltzmann equation.  
The nonlinear Boltzmann equation for homoenergetic flows, which describes the evolution of the distribution function $f=f\left(w,t\right)$ in the velocity phase space $\R$, namely $f: \R \times \Rl_+ \rightarrow  \mathbb{R}_+$, then reduces to (e.g., see \cite{JNV1} and references therein)
\begin{equation}\label{eq:homBol}
    \partial_t f-L(t)w \cdot \partial_w f = Q(f,f), \qquad \, f=f (t, w) 
\end{equation}
where the collision operator $Q(f,f)$ is given by 
\begin{equation}\label{eq:collB}
  Q(f,f)(w)=   \int_{\R}\int_{S^2} B(n\cdot \omega,|w-w_*|)(f'_*f'-f_*f)\, dw_* d \omega\,,
\end{equation} 
and the relation between the post-collisional velocities $(w',w'_*)$ and the pre-collisional ones $(w,w_*)$ is given by the collision rule \eqref{eq:CollRule}. 

In the following we will denote the term  $L(t)w \cdot \partial_w f$ as hyperbolic term.   
We are interested in the long-time asymptotics of the 
function $\xi\left(  x,t \right)  =L\left(  t\right)  x = A\left(  I+tA\right)^{-1} x$ where 
the matrix $L(t)$ describes the mechanical deformation acting on the gas (cf.~Theorem 3.1 in \cite{JNV1}). Here we will focus on the case of \emph{simple shear deformation}, i.e. when the matrix $L(t)$ has the form
\begin{equation}\label{eq:ShearMat}
L = \left( \begin{array}{ccc} 0 & -K & 0 \\
	0 & 0 & 0  \\
    0 & 0 & 0 
\end{array} \right), \quad K>0.
 \end{equation}
 
Concerning the collision kernel $B$, we assume the collision kernels can be decomposed 
as 
\begin{equation}\label{eq:decB}
    B(n \cdot \omega,|w-w_\ast|)=\kappa (n\cdot \omega)|w-w_\ast|^{\gamma}, \quad \kappa \in L^{\infty}([-1,1]),\, \kappa   \geq 0, \, \gamma \in \mathbb{R}.
\end{equation}  
Actually, in this paper we consider the simplest case in which $B$ has  the form \eqref{eq:decB} and satisfies
\begin{equation}\label{eq:assB}
   B(n \cdot \omega, |w-w_\ast|)=|w-w_\ast|^{\gamma} \quad \text{with}\quad \gamma \in (-1, 0)\, . 
\end{equation}  
Notice that, for simplicity,  we have assumed the angular part $\kappa(n\cdot \omega)$ of the collision kernel to be constant and, without loss of generality, we choose the constant equal to $1$, i.e. $\kappa(n\cdot \omega)=1$.  It seems possible to adapt the argument of this paper to the more general case of $\kappa( n \cdot \omega)$ non-constant. 

Here we are only considering only a specific range of negative values for $\gamma$ and, to simplify the notation, we will set $a:=|\gamma|$.  Hence \eqref{eq:assB} becomes $B(n \cdot \omega, |w-w_\ast|)=|w-w_\ast|^{-a}$.   
\smallskip

\paragraph{State of the art on the theory of homoenergetic solutions.} 
In recent years, there has been much progress in the analysis of homoenergetic solutions of the Boltzmann equation.  These are a particular type of non-equilibrium solutions of the Boltzmann equation which are  useful to describe the dynamics of Boltzmann gases under
shear, expansion or compression. They were introduced by Galkin \cite{Galkin1} and Truesdell \cite{T} in the 1960s, and have later gained significant attention in both physics \cite{garzo, TM} and mathematical literature, starting with the works of Truesdell and Muncaster (\cite{TM}) and later Cercignani (\cite{CercArchive}). Additional and more recent references  will be described more in detail below.  

Due to the fact that these solutions
describe far-from-equilibrium phenomena their long-time asymptotics
 is not always described by Maxwellian distributions.
More precisely, these solutions describe dilute gases under boundary conditions at infinity, induced by  mechanical deformations, effectively behaving as an open system. Consequently, they generally lack a detailed balance condition and do not thermalize to a standard Maxwellian equilibrium. As a  result, the  characterization of the long-time asymptotics  is highly non-trivial and, in some regimes, poses some   mathematical challenges.

We observe that, in general, the homoenergetic solutions to the Boltzmann equation, for long times, are solutions with a non zero entropy production due to the fact that they are non Maxwellians. Moreover, they are associated to fluxes of energy and momentum coming or going towards infinity. 
\smallskip

The well-posedness of homoenergetic solutions was first established in \cite{CercArchive}, \cite{Cerc2000}, and later in \cite{JNV1}, the long-time behavior has been the subject of extensive recent study \cite{BNV, DL1, DL2, JNV1, JNV2, JNV3, JQW, K1, K2, MT, NV}, including extensions to inelastic interactions \cite{CCDL}. We also refer to \cite{MNV25a,MNV25b} for the analysis of homoenergetic solutions to a linear version of the Boltzmann equation, specifically the Rayleigh-Boltzmann equation.

 It has been originally noticed in \cite{JNV1} (see \cite{JNV2,JNV3}) that according to the value of the homogeneity parameter $\gamma$ and the type of mechanical deformation we can distinguish between  the collision-dominated case, i.e., when the collision dynamics is the dominant one for large times; the case in which the collision term and the drift term in \eqref{eq:homBol} balance and the hyperbolic-dominated case, when the drift term in \eqref{eq:homBol} is the dominant. 
 In the collision-dominate case it has been possible to prove that the asymptotic behaviour of  the solutions  to \eqref{eq:homBol} is described by Maxwellian distribution with time-dependent temperature (cf. \cite{JNV2, K1}), where the change of temperature is due to the flux of energy   induced by the mechanical deformation (e.g., simple shear). On the other hand in the case of balance we have obtained existence, uniqueness and global stability of non-equilibrium self-similar solutions (i.e. non Maxwellian distributions) under smallness assumption of the mechanical deformation (cf.~\cite{BNV, JNV1}). Also in this case it is possible to associate a macroscopic temperature to the distribution (in terms of variance of the velocity) which changes in time due to the effect of the deformation.   The case in which the hyperbolic term is dominant yields asymptotic distribution of particles velocities that are very different from Maxwellian distribution. 
This is also the case in which less information about the long-time behaviour of the solutions is available. This is due to the fact that the main contribution to the particles velocities distribution is due to the deformation term (e.g., shear) that prevents the thermalization effect of collisions (cf.~\cite{JNV3}).   
 In this case, it is possible to have two possible situations: either the collisions are completely negligible for long times (\emph{frozen collisions}) or there are infinitely many collisions that, however, take place in increasingly longer times. These collisions have a huge effect on the distribution of velocities of the system. This regime is the one considered in this paper. The fact that these increasingly rare collisions modify in a very significant manner the distribution of velocities has been noticed for the first time in \cite{JNV3} for a toy version of \eqref{eq:homBol} and later in \cite{MNV25b} for the Rayleigh-Boltzmann equation.

\paragraph{Main result of the paper. } As mentioned above, in this paper we are interested in studying the long-time asymptotics of the solutions to 
\begin{equation}\label{eq:BNLshear} 
	\partial_tf+K w_2 \partial_{w_1}f=\int_{\mathbb{R}^{3}}dw_{\ast}\int_{S^{2}%
	} \,  |w-w_*|^{-a}
	\left(  f^{\prime}f^{\prime}_{\ast}-ff_{\ast}\right) d\omega \, ,
\end{equation} 
in the hyperbolic-dominated regime, in the \emph{non frozen collision} case, which corresponds to cutoff soft-potential interactions, i.e. to collision kernels $B$ satisfying \eqref{eq:assB} with $0<a< 1$.  

We notice that it is possible to rescale time and velocities so that the shear parameter $K=1$; hence, without loss of generality, we assume $K=1$  in \eqref{eq:BNLshear} throughout the manuscript.

We remark, that the solutions constructed in this paper, can be thought as flux solutions due to the fact  that there is a flux of momentum from $w_3=-\infty$ to $w_3=+\infty$. 

In order to study the asymptotic behavior of the solutions to \eqref{eq:BNLshear} we will first reformulate the problem in an appropriate set of self-similar variables. We point out that we call this set of variables ``self-similar" because in the rescaled  equations  the different terms associated to the shear deformation become comparable. Notice however that the solutions obtained in this paper are far from being self-similar in the usual sense, namely they do not have the same shape, up to a rescaling, for different times. 

\smallskip

\noindent\textbf{Ansatz:} We look for solutions $f$ to \eqref{eq:BNLshear}  with the following form
   \begin{equation}\label{def:scaling}
    f(w,t)= \frac{1}{t^{\frac{3}{a}-2}} F\left(\xi_1, \xi_2, \xi_3, \tau \right), \;\;\xi_1=\frac{w_1}{t^{\frac{1}{a}}}, \ \ \xi_i=\frac{w_i}{t^{\frac{1}{a}-1}} , \ i=2,3, \quad \tau= \log t \ .  
\end{equation}
\smallskip

In this paper we will carry out a formal analysis of the long-time asymptotics of the velocity distribution, specifically we will derive an explicit expression for the asymptotic profile.   
We now state the main result  which provides a   characterization of 
the long-time behaviour of the solutions to \eqref{eq:BNLshear}. 

\begin{mainresult}[Asymptotics]\label{thm:ConvergenceF} 
    Let $f \in C([0,+\infty),\M)$ be 
    a solution of \eqref{eq:BNLshear} 
    and let $a=\vert \gamma\vert \in (0,1)$.  
    Let be $F$ the rescaled solution $f$ given as in \eqref{def:scaling}, written in the set of self-similar variables. 
Then we can construct, using matched asymptotic expansions, a solution with the following asymptotic behavior
\begin{equation}\notag
F(\xi_1,\xi_2,\xi_3,\tau) \sim 
 \frac 1 \pi \frac{\lambda(\tau )}{\xi_2^2+\xi_3^2 } \sigma\left(\frac{\xi_2}{\varepsilon(\tau)},\tau\right)\delta(\xi_1-\xi_2)  %
\ \ \text{as} \ \ \tau\to + \infty 
\end{equation}
where $\sigma=\sigma(s, \tau)$,  with $s={\xi_2}/{\varepsilon(\tau)}$, converges to zero as a power-law as $s\to 0^+$ and also when $s\gg e^\tau$ while $\lambda(\tau)\sim \frac{C_0}{\tau}$  with $C_0 =\frac 1 {2a} \frac 1 {\log \left(\frac{1}{1-a}\right)}.$   
In other words we can   interpret the function $\sigma(s, \tau)$ as a cutoff function $\chi_{\{s \in [1,{e^\tau}]\}}$. 
The characteristic scale  $\varepsilon(\tau)$ is defined as 
\begin{equation}\label{defMR:esplamtau0}
    \varepsilon(\tau) =\left( a\lambda((1-a)\tau)\right)^{\frac{1}{a}}\,. 
\end{equation}
\end{mainresult}

    We notice that the asymptotic form of the solution obtained in the Main Result encodes several relevant properties. Firstly, the solution is supported in the range of values of $\xi\in \mathbb{R}^3$ for which $\varepsilon(\tau) \lesssim \vert \xi \vert \lesssim \varepsilon(\tau)e^{\tau}$. On the other hand, most of the collisions take place for $\vert \xi\vert \approx \varepsilon(\tau)$. In addition, most of the particles are localized on the plane $\{\xi_1=\xi_2\}$. This is due to the effect of the shear acting on the system.

\paragraph{Plan of the paper.}
The paper is organized as follows.  In Section 2, we introduce the characteristic scales of the system and reformulate the problem in a convenient set of self-similar variables, detailing both its weak and strong formulations. Section 3 focuses on reducing the system to a two-dimensional problem, where we derive the evolution equation for a reduced distribution $\displaystyle G=\int_{\mathbb{R}} F \, d\xi_3$  
and reformulate it as a suitable boundary value problem. In Section 4, we discuss the rationale underlying the asymptotic form of the solution derived in this work. Section 5 establishes the asymptotics for the solution $G$ of the reduced two dimensional model, while Section 6 extends these results to compute the asymptotics for the full distribution $F$ as $\tau\rightarrow\infty$, alongside a test of their consistency. Section 7 provides a rigorous analysis of the removal mechanism of particles in the collision region, i.e. when $|\xi|\approx\varepsilon(\tau)$, including the proof of existence and the asymptotic behavior of the associated stationary problem. Finally, Appendix A contains the justification of the gain and loss terms for the reduced two dimensional problem, as well as the conservation of mass properties for both the full distribution $F$ and the reduced distribution $G$. These computations provide a direct and simple test of the validity of the evolution equations for $F$ and $G$. 
\smallskip

\section{Characteristic scales of the system and reformulation of the problem \eqref{eq:BNLshear} in the self-similar variables}

We are interested in the long-time behaviour of solutions of \eqref{eq:BNLshear} and from now on we  will focus on this.  
We start by introducing the notion of solutions to \eqref{eq:BNLshear} we will consider. 
We remark that we will use indistinctly $f(w)dw$ and $f(dw)$ to denote elements of these measure
spaces. The notation $f(dw)$ will be preferred when performing integrations or when
we want to emphasize that the measure might not be absolutely continuous with
respect to the Lebesgue measure. 

We now provide the precise notion of weak solutions of \eqref{eq:BNLshear} that we will use. 
\begin{defn}\label{def:weakSol}
 	Let $T>0$. A function $f \in C\left( [0,+\infty),\M\right)$ is a \emph{weak solution} of \eqref{eq:BNLshear} if for every function $\psi \in C^1\left( [0,+\infty),C^{\infty}_c(\R) \right) $ one has
 \begin{align}
     \label{eq:weakf}
		& \partial_t\left(\int_{\R}\psi(w,t)f(d w,t)\right)\notag \\ & = 
      \int_{\R} \partial_t \psi (w,t)\, f(dw,t) + \int_{\R}w_2\partial_{w_1}\psi(w,t) \, f(dw,t)\notag \\ & 
     \quad + \int_{S^2} d \omega\int_{\R} \int_{\R} B(n\cdot \omega,|w-w_*|)f(d w,t)f(d w_*,t)\Big[\psi(w'_*)+\psi(w')-\psi(w_*)-\psi(w)\Big].
	\end{align} 
\end{defn}

In what follows we will consider asymptotic expansions of solutions as $t \to \infty$. In order to make rigorous these expansions the issue of the global well-posedness
should be addressed. Global existence of solutions of \eqref{eq:BNLshear} with a suitable initial datum is expected for all reasonable
collision kernels $B$ with arbitrary homogeneity $\gamma$. Although we will not make them rigorous
in this paper, we believe that a typical global well-posedness result that can be proved rigorously
by adapting the current available methods for the homogeneous Boltzmann equation. 
In particular, adapting the proof of Theorem 4.2 in \cite{JNV1}, using
standard arguments from the theory of homogeneous Boltzmann equations as described in
\cite{CIP, De, He, V02}.
\bigskip

\noindent\textbf{Change to self-similar variables.} We perform the following change of variables 
\begin{equation}\label{eq:selfsim_var}
    {w_1}={t^{\frac{1}{a}}}\xi_1, \ \ w_i= {t^{\frac{1}{a}-1}}\xi_i , \ i=2,3, \quad t= e^\tau  
\end{equation}
and use, in the weak formulation \eqref{eq:weakf}, test functions with the form 
$$ \psi(w,t)=\varphi(\xi_1, \xi_2, \xi_3, \tau), , \;\;\xi_1=\frac{w_1}{t^{\frac{1}{a}}}, \ \ \xi_i=\frac{w_i}{t^{\frac{1}{a}-1}} , \ i=2,3, \quad \tau= \log t\, . $$ 
We then consider solutions $f$ with the following scaling 
\begin{equation}\label{def:scaling}
    f(w,t)= \frac{1}{t^{\frac{3}{a}-2}} F\left(\xi_1, \xi_2, \xi_3, \tau \right), \;\;\xi_1=\frac{w_1}{t^{\frac{1}{a}}}, \ \ \xi_i=\frac{w_i}{t^{\frac{1}{a}-1}} , \ i=2,3, \quad \tau= \log t \, .
\end{equation}

\subsection{Weak formulation of the problem in self-similar variables
}\label{ssec:weakF}

In order to obtain the weak formulation of the problem in the self-similar variables \eqref{eq:selfsim_var} we start considering the first two contributions in the r.h.s.~of \eqref{eq:weakf}, namely the ones corresponding to the time  derivative and the drift term respectively,  and treat separately the collision term. We have
\begin{align}\label{eq:weakF_tr}
&\partial_{\tau}\left(\int_{\R}\varphi(\xi,\tau)F(d \xi,\tau)\right)\notag \\ & =
    \int_{\R} F(d\xi,\tau)\left[\partial_\tau \varphi -\frac{1}{a}\xi_1\partial_{\xi_1}\varphi -\left(\frac{1}{a}-1\right)\xi_2\partial_{\xi_2}\varphi-\left(\frac{1}{a}-1\right)\xi_3\partial_{\xi_3}\varphi
    +\xi_2 \partial_{\xi_1}\varphi
    \right] (\xi,\tau)\notag\\&
    +\int_{\R} \varphi(\xi,\tau)Q(F,F)(d\xi,\tau)
\end{align}
    We now have to write explicitly the weak form of the collision operator, namely $\int_{\R} \varphi(\xi,\tau)Q(F,F)(d\xi,\tau)$ 
in terms of the rescaled variables $\xi_i$ and, thus, in terms of the test function $\varphi$.  
We observe that, due to the anisotropy of the
change of variables, the collision rule and the collision kernel are also affected by the scaling \eqref{def:scaling}.  
In order to determine the  collision rule for the variable $\xi$ and thus rewrite the weak form of the collision operator 
in terms of  $\varphi(\xi,\tau)$ in place of $\psi(w,t)$, we introduce the matrix
\begin{equation*}
\label{eq:defmatrixS}
    S(t)=\frac{1}{t^{\frac{1}{a}}}\left( \begin{array}{ccc}
        1 & 0 & 0  \\
        0 & t & 0  \\
        0 & 0 & t
    \end{array}\right)
\end{equation*}
so that we can write $\xi=S(t)w$ as well as $\xi_{\ast}=S(t)w_{\ast}$. Conversely, denoting by $T(t)$ the inverse matrix of $S(t)$ i.e. $T(t)=(S(t))^{-1}$, we have $w=T(t)\xi=\left(t^{\frac{1}{a}} \xi_1, t^{\frac{1}{a}-1}\xi_2, t^{\frac{1}{a}-1}\xi_3\right)$ and $w_{\ast}=T(t)\xi_{\ast}=\left(t^{\frac{1}{a}} \xi_{1,\ast}, t^{\frac{1}{a}-1}\xi_{2,\ast}, t^{\frac{1}{a}-1}\xi_{3,\ast}\right)$.  
Thus, we have 
\begin{align}  \label{eq:psiw'} \psi(w',t)=\varphi(S(t)w',t)=\varphi(S(t)\left[w-\left((w-w_{\ast})\cdot\omega\right) \omega\right],t)
\end{align} 
We notice that 
\begin{equation}\label{eq:CollRule_NewVar}
    w'= w-\left((w-w_{\ast})\cdot\omega\right) \omega=T(t)\xi-(T(t)(\xi-\xi_{\ast})\cdot \omega)\omega
\end{equation} and then, applying $S(t)$ to both sides of the above equation yields
\begin{equation}\label{eq:approxCollXi}
    \xi'=\xi-(T(t)(\xi-\xi_{\ast}) \cdot \omega)S(t)\omega.
\end{equation}
from \eqref{eq:psiw'}  we arrive at 
\begin{align}  \label{eq:psiw'_2} 
\psi(w',t)=\varphi(\xi- T(t)(\xi-\xi_{\ast})\cdot \omega)S(t)\omega,t)
\end{align} \smallskip

\noindent\textbf{Approximation:}  
Since we expect that in the regime considered here, namely when the shear deformation term is dominant, $[w_1]\gg [w_2],[w_3]$, we can assume $[w] \approx [w_1]$. Hence 
$$w-w_{\ast} =  T(t)\left(\xi-\xi_{\ast}\right)\approx t^{\frac{1}{a}} \left(\xi_{1}-\xi_{1,\ast}\right)$$ 
and 
$$(w-w_{\ast}) \cdot \omega = (T(t)\left(\xi-\xi_{\ast}\right))\cdot \omega \approx t^{\frac{1}{a}} \left(\xi_{1}-\xi_{1,\ast}\right)  \omega_1\, .$$ 
which implies 
\begin{align}
(T(t)(\xi-\xi_{\ast}) \cdot \omega)S(t)\omega  & \approx 
\left(t^{\frac{1}{a}} \left(\xi_{1}-\xi_{1,\ast}\right)\omega_1\right) t^{-\frac{1}{a}} \left( \begin{array}{ccc}
     1  & 0 & 0  \\
     0 & t & 0 \\
     0 & 0 & t
\end{array}\right) \omega  
\notag \\& 
=
\left( \left(\xi_{1}-\xi_{1,\ast}\right)\omega_1\right) \, t \left(I-e_1 \otimes e_1\right)(\omega)\end{align}
We then obtain the following approximation for \eqref{eq:psiw'_2}:
\begin{align} \label{eq:ffi_app} 
\psi(w',t)&=
\varphi(\xi- T(t)(\xi-\xi_{\ast})\cdot \omega)S(t)\omega,t)\\&=
\varphi(\xi- t \left(\xi_{1}-\xi_{1,\ast}\right)\omega_1 \left(I-e_1 \otimes e_1\right)(\omega),t) 
\end{align}
  Neglecting also the term $\xi$ we obtain the approximation 
\begin{align} \label{eq:ffi_app2} 
\psi(w',t)&\approx  
\varphi(t \left(\xi_{1,\ast}-\xi_{1}\right)\omega_1 \left(I-e_1 \otimes e_1\right)(\omega),t) 
\end{align} 
and, arguing similarly for $\psi(w'_{\ast},t)$, we obtain
\begin{align} \label{eq:ffi_app2} 
\psi(w'_{\ast},t)&\approx  
\varphi(t \left(\xi_{1}-\xi_{1,\ast}\right)\omega_1 \left(I-e_1 \otimes e_1\right)(\omega),t) 
\end{align}
Therefore,  the sum $\left[ 
\psi(w'_\ast)+\psi(w'_\ast)-\psi(w_\ast)-\psi(w) \right]$ in the integrand of \eqref{eq:weakf} becomes 
\begin{align} 
\Big[  
\varphi(t \left(\xi_{1,\ast}-\xi_{1}\right)\omega_1 \left(I-e_1 \otimes e_1\right)(\omega))+ 
\varphi(t \left(\xi_{1}-\xi_{1,\ast}\right)\omega_1 \left(I-e_1 \otimes e_1\right)(\omega)) 
-\varphi(\xi_\ast)-\varphi(\xi) \Big]
\end{align}
and then, 
using that $\frac{t}{|w-w_{\ast}|^{a}}\sim\frac{1}{|\xi_1-\xi_{1,\ast}|^{a}}$ as $t \to \infty$, the weak form of the collision operator becomes 
\begin{align}\label{eq:weakF_coll}
\int_{S^2}\int_{\R}\int_{\R} &\frac{1}{|\xi_1-\xi_{1,\ast}|^{a}} F(d\xi, \tau)F(d\xi_{\ast},t)   \notag \\& \Big[  
\varphi(t \left(\xi_{1,\ast}-\xi_{1}\right)\omega_1 \left(I-e_1 \otimes e_1\right)(\omega))+ 
\varphi(t \left(\xi_{1}-\xi_{1,\ast}\right)\omega_1 \left(I-e_1 \otimes e_1\right)(\omega)) 
-\varphi(\xi_\ast)-\varphi(\xi) \Big] 
 d\omega \, .
\end{align}

Collecting \eqref{eq:weakF_tr} and \eqref{eq:weakF_coll} we thus arrive at  the following weak formulation for the approximated problem  
\begin{align}\label{eq:weakF}
&\partial_{\tau}\left(\int_{\R}\varphi(\xi,\tau)F(d \xi,\tau)\right)\notag \\ & =
    \int_{\R} F(d\xi,\tau)\left[\partial_\tau \varphi -\frac{1}{a}\xi_1\partial_{\xi_1}\varphi -\left(\frac{1}{a}-1\right)\xi_2\partial_{\xi_2}\varphi-\left(\frac{1}{a}-1\right)\xi_3\partial_{\xi_3}\varphi
    +\xi_2 \partial_{\xi_1}\varphi\right] (\xi,\tau) \notag\\&
    \quad + \int_{S^2}\int_{\R}\int_{\R} \frac{1}{|\xi_1-\xi_{1,\ast}|^{a}} F(d\xi, \tau)F(d\xi_{\ast},t)   \notag \\& 
    \quad \times \Big[  
\varphi(t \left(\xi_{1,\ast}-\xi_{1}\right)\omega_1 \left(I-e_1 \otimes e_1\right)(\omega))+ 
\varphi(t \left(\xi_{1}-\xi_{1,\ast}\right)\omega_1 \left(I-e_1 \otimes e_1\right)(\omega)) 
-\varphi(\xi_\ast)-\varphi(\xi) \Big] 
 d\omega \,.
\end{align}
Obtaining the strong formulation of the equation satisfied by the rescaled measure $F$ introduced in \eqref{def:scaling},  requires some care, 
due to the complicated structure of the collision operator. 
The derivation of the strong formulation of the equation \eqref{eq:weakF} and, specifically, of the collision operator $\mathcal{Q} (F,F)$ will be performed in the next Section \ref{ssec:strongF}.

\subsection{Strong formulation  of the problem in self-similar variables 
}\label{ssec:strongF}

In this Section we derive the  strong formulation  of the problem arising from \eqref{eq:weakF}. More precisely, we will arrive at the following equation for $F=F(\xi,\tau)$
\begin{equation}
\label{eq:ss_strong2} 
    \partial_{\tau}F-\frac{1}{a}\partial_{\xi_1}(\xi_1F)-\left(\frac{1}{a}-1\right)\partial_{\xi_2}(\xi_2F)-\left(\frac{1}{a}-1\right)\partial_{\xi_3}(\xi_3F)
   +  \partial_{\xi_1}(\xi_2 F )={Q}[F,F] (\xi)
\end{equation}
where ${Q}[F,F](\xi):=\delta(\xi_1) 
     {Q}^+[F,F](\tilde{\xi})-{Q}^-[F,F](\xi)$ with  $\xi=(\xi_1,\tilde \xi)\in \mathbb{R}^3$, $\tilde \xi=(\xi_2,\xi_3)\in \mathbb{R}^2$. More precisely, the gain operator reads as
\begin{align}\label{eq:ss_strong2_gain}
 & {Q}^+[F,F](\tilde{\xi})=  \frac {32}{2^a} t^{a-1}\int_{\R} dx \int_{\mathbb{R}^2}d\tilde{y}\, F\left( x_1+  |\tilde \xi|\left(t \sin(2\theta)\right)^{-1}, x_2+y_2,x_3+y_3, \tau\right) \times \notag \\& \qquad \qquad \qquad \qquad \qquad \times F\left(x_1-|\tilde \xi|\left(t \sin(2\theta)\right)^{-1},x_2-y_2,x_3-y_3,\tau \right) 
 \frac{  \sin(\theta) |\sin(2\theta)|^{a-1}}{|\tilde\xi|^{a+1}}  
 \end{align}
 where $\tilde{y}=(y_2,y_3)\in \mathbb{R}^2$ and the loss term reads as
 \begin{equation}
  {Q}^-[F,F](\xi)=2F(\xi,\tau) \int_{S^2} d\omega\int_{\R}d\xi_{\ast} \,|\xi_1-\xi_{1,\ast}|^{-a} F(\xi_{\ast},\tau)\, .\label{eq:ss_strong2_loss}
\end{equation}
\medskip

Notice that, as expected from the weak formulation, this equation satisfies the mass conservation property. This will be checked by means of a direct computation in Appendix \ref{ssec:massF}. \smallskip 

\noindent\textbf{Derivation of the strong formulation \eqref{eq:ss_strong2}-\eqref{eq:ss_strong2_loss} from the weak formulation \eqref{eq:weakF}.} 
Starting from the weak formulation \eqref{eq:weakF}, i.e.
\begin{align*}
&\partial_{\tau}\left(\int_{\R}\varphi(\xi,\tau)F(d \xi,\tau)\right)\notag \\ & =
    \int_{\R} F(d\xi,\tau)\left[\partial_\tau \varphi -\frac{1}{a}\xi_1\partial_{\xi_1}\psi -\left(\frac{1}{a}-1\right)\xi_2\partial_{\xi_2}\psi-\left(\frac{1}{a}-1\right)\xi_3\partial_{\xi_3}\psi 
    +\xi_2 \partial_{\xi_1}\psi \right] (\xi,\tau) \notag\\&
    \quad + \int_{S^2}\int_{\R}\int_{\R} \frac{1}{|\xi_1-\xi_{1,\ast}|^{a}} F(d\xi, \tau)F(d\xi_{\ast},\tau)   \notag \\& \quad \times \Big[  
\varphi(t \left(\xi_{1,\ast}-\xi_{1}\right)\omega_1 \left(I-e_1 \otimes e_1\right)(\omega))+ 
\varphi(t \left(\xi_{1}-\xi_{1,\ast}\right)\omega_1 \left(I-e_1 \otimes e_1\right)(\omega)) 
-\varphi(\xi_\ast)-\varphi(\xi) \Big] 
 d\omega \,
\end{align*} 
we can obtain the strong formulation of the equation choosing in \eqref{eq:weakF}  test functions with the form $\varphi(\xi,\tau)=\delta(\xi-\xi_0)$ with $\xi_0\in\R$ arbitrary. With this choice we obtain $\int_{\R}\varphi(\xi,\tau)F(d \xi,\tau)=F(\xi_0,\tau)$ and thus  $\partial_{\tau}\left(\int_{\R}\varphi(\xi,\tau)F(d \xi,\tau)\right)=\partial_{\tau} F(\xi_0,\tau)$.  
We now integrate by parts in the drift terms (cf. the first integral on the right hand side of \eqref{eq:weakF}). It then  follows 
\begin{align*}
\int_{\R} \varphi(\xi)\Bigg[\partial_\tau F(d\xi,\tau) -\frac{1}{a}&  \partial_{\xi_1}(\xi_1 F(d\xi,\tau)) -\left(\frac{1}{a}-1\right)\partial_{\xi_2}(\xi_2F(d\xi,\tau))\notag \\& 
-\left(\frac{1}{a}-1\right)\partial_{\xi_1}(\xi_1 F(d\xi,\tau)) 
    + \partial_{\xi_1}(\xi_2 F(d\xi,\tau)) \Bigg]   \notag 
\end{align*}
Choosing $\varphi(\xi,\tau)=\delta(\xi-\xi_0)$ and  using that the derivative of $\vp$ with respect to the the time variable $\tau$ is zero,  we arrive at
$$
 \frac{1}{a}\partial_{\xi_1}(\xi_1 F(d\xi,\tau)) + \left(\frac{1}{a}-1\right)\partial_{\xi_2}(\xi_2
F(d\xi,\tau))+ \left(\frac{1}{a}-1\right)\partial_{\xi_1}(\xi_1 F(d\xi,\tau)) 
    - \partial_{\xi_1}(\xi_2 F(d\xi,\tau)) \Bigg|_{\xi=\xi_0}\, .
$$
We are now left with the collision term. We first consider the contribution coming from the loss term. Using that $\varphi(\xi,\tau)=\delta(\xi-\xi_0)$ we obtain
\begin{align*}
-\int_{S^2}\int_{\R}\int_{\R} \frac{1}{|\xi_1-\xi_{1,\ast}|^{a}} F(d\xi, \tau)F(d\xi_{\ast},\tau) 
\Big[  \varphi(\xi_\ast)+\varphi(\xi) \Big] 
 d\omega \, = -2 F(\xi_0, \tau) \int_{S^2} \int_{\R} \frac{1}{|\xi_1-\xi_{1,\ast}|^{a}} F(d\xi_{\ast},t)d\omega
\end{align*} 
where we used the symmetry of the collision operator with respect to $\xi\to\xi_{\ast}$. 
We now compute the  gain term. We first notice that 
$$\left(I-e_1 \otimes e_1\right)(\omega)=\left(0, P_2(\omega),P_3(\omega) \right)^T $$
where $P_i$, $i=2,3$ denote the components of the orthogonal projection in the vectors $e_2,\, e_3$. Then, 
$$\varphi(t \left(\xi_{1,\ast}-\xi_{1}\right)\omega_1 \left(I-e_1 \otimes e_1\right)(\omega))=\varphi(0,t \left(\xi_{1,\ast}-\xi_{1}\right)\omega_1 P_2(\omega), t\left(\xi_{1,\ast}-\xi_{1}\right)\omega_1 P_3(\omega))\, .$$
Arguing similarly to the previous case, using the symmetry of the collision operator as well as the factorization properties of the Dirac delta, i.e. $\delta^{(3)}(x-x_0)=\delta(x_1-x_{1,0})\delta(x_2-x_{2,0})\delta(x_3-x_{3,0})$ with $x\in \R$, $x_0=(x_{0,1},x_{0,2},x_{0,3})\in\R$, we arrive at 
\begin{align}
&\int_{S^2}\int_{\R}\int_{\R} \frac{1}{|\xi_1-\xi_{1,\ast}|^{a}} F(d\xi, \tau)F(d\xi_{\ast},\tau) \times \nonumber \\&
 \qquad \qquad \quad \times \Big[  
\varphi(t \left(\xi_{1,\ast}-\xi_{1}\right)\omega_1 \left(I-e_1 \otimes e_1\right)(\omega))+ \varphi(t \left(\xi_{1}-\xi_{1,\ast}\right)\omega_1 \left(I-e_1 \otimes e_1\right)(\omega)) 
\Big]d\omega \notag \\& 
=2\int_{S^2}\int_{\R}\int_{\R} \frac{1}{|\xi_1-\xi_{1,\ast}|^{a}} F(d\xi, \tau)F(d\xi_{\ast},\tau)    
\varphi(t \left(\xi_{1,\ast}-\xi_{1}\right)\omega_1 \left(I-e_1 \otimes e_1\right)(\omega))\, d\omega \notag \\&
=2\int_{S^2}\int_{\R}\int_{\R} \frac{1}{|\xi_1-\xi_{1,\ast}|^{a}} F(d\xi, \tau)F(d\xi_{\ast},\tau)   
\, \delta^{(3)}(t \left(\xi_{1,\ast}-\xi_{1}\right)\omega_1 \left(I-e_1 \otimes e_1\right)(\omega)-\xi_0)\, d\omega \notag \\& 
=2\delta(\xi_{0,1})\int_{S^2}\int_{\R}\int_{\R} \frac{1}{|\xi_1-\xi_{1,\ast}|^{a}} F(d\xi, \tau)F(d\xi_{\ast},\tau)  
\, \times \nonumber \\& 
\qquad \qquad \quad \times \delta(t \left(\xi_{1,\ast}-\xi_{1}\right)\omega_1 P_2(\omega)-\xi_{0,2}) 
\, \delta(t \left(\xi_{1,\ast}-\xi_{1}\right)\omega_1 P_3(\omega)-\xi_{0,3})\, d\omega \, . \label{eq:strongFcoll}
\end{align}
We now change variables, working in the center of mass system coordinate frame, and set $X=\frac{\xi+\xi_{\ast}}{2}$, $Y=\frac{\xi-\xi_{\ast}}{2}$ so that $\xi=X+Y$, $\xi_{\ast}=X-Y$ and $d\xi d\xi_{\ast}=8\,dXdY$. The integral on the left hand side of \eqref{eq:strongFcoll} then becomes 
\begin{align}\label{eq:strongFcoll1}
& \frac{16}{2^a}\delta(\xi_{0,1})\int_{S^2}d\omega \int_{\R}dX \int_{\R}  dY\, \frac{ F(X+Y, \tau)F(X-Y,\tau) }{|Y_{1}|^{a}}
\, \delta\left(- {2}t \,\omega_1 Y_1 P_2(\omega)-\xi_{0,2}\right) 
\, \delta\left(- {2}t \,\omega_1 Y_1   P_3(\omega)-\xi_{0,3}\right)\notag \\& 
=\frac{16}{2^a}\delta(\xi_{0,1})\int_{S^2}d\omega \int_{\R}dX \int_{\R}  dY \frac{ F(X+Y, \tau)F(X-Y,\tau) }{|Y_{1}|^{a}} 
\, \delta\left({2}t \,\omega_1 Y_1 P_2(\omega)-\xi_{0,2}\right) 
\, \delta\left( {2}t \,\omega_1 Y_1   P_3(\omega)-\xi_{0,3}\right)  \, 
\end{align}
where we changed $Y\to -Y$ in  the last step. Following now the same strategy used in the linear setting for the case of the Rayleigh-Boltzmann equation in \cite{MNV25b} (see Section 4.2.1) we consider first the integral w.r.t. $\omega\in S^2$. We will denote as $\tilde{\xi}_0=(\xi_{0,2},\xi_{0,3})$ and $\tilde \xi_0= |\tilde \xi_0|(\cos(\psi_0),\sin(\psi_0))$. It has been proved in \cite{MNV25b} that 
\begin{align}\label{eq:Diracidentities}
&\int_{S^2} \, d\omega\, \delta\left(t \, \xi_1 \, \omega_1 P_2(\omega)-\xi_{0,2}\right) 
\, \delta\left( t \,\xi_1 \, \omega_1   P_3(\omega)-\xi_{0,3}\right) \nonumber \\&=
\int_0^{2\pi} d\vp \int_0^{\pi} d\theta  \frac{4}{\vert \xi_1\vert t^2\sin^2(2\theta)} \left(\delta(\vp-\psi_0)\delta\left(\xi_1-\frac{2|\tilde \xi_0|}{t\sin(2\theta)}\right)+\delta(\vp-\psi_0+\pi)\delta\left(\xi_1+\frac{2|\tilde \xi_0|}{t\sin(2\theta)}\right)\right)\,.
\end{align}
Applying \eqref{eq:Diracidentities}  we obtain
\begin{align} \label{eq:strongFcoll1_omega}
 & \int_{S^2}  \delta\left({2}t \, Y_1 \,\omega_1 P_2(\omega)-\xi_{0,2}\right) 
\, \delta\left({2}t \,Y_1\,\omega_1   P_3(\omega)-\xi_{0,3}\right)\, d\omega \nonumber \\& 
= 
\int_0^{2\pi} d\vp \int_0^{\pi} d\theta \frac{4 \sin\theta}{2|Y_1| t^2\sin^2(2\theta)}\Bigg[ 
\left(\delta(\vp-\psi_0)\delta\left({2Y_1}-\frac{2|\tilde \xi_0|}{t\sin(2\theta)}\right)+\delta(\vp-\psi_0+\pi)\delta\left({2Y_1}+\frac{2|\tilde \xi_0|}{t\sin(2\theta)}\right)\right)
\Bigg]\nonumber\\& 
= 
\int_0^{2\pi} d\vp \int_0^{\pi} d\theta \frac{ \sin\theta}{|Y_1| t^2\sin^2(2\theta)}\Bigg[ 
\left(\delta(\vp-\psi_0)\delta\left(Y_1-\frac{|\tilde \xi_0|}{t\sin(2\theta)}\right)+\delta(\vp-\psi_0+\pi)\delta\left(Y_1+\frac{|\tilde \xi_0|}{t\sin(2\theta)}\right)\right)
\Bigg]
\end{align}
Substituting  \eqref{eq:strongFcoll1_omega} into \eqref{eq:strongFcoll1} we arrive at 
\begin{align}\label{eq:strongFcoll2}
 \frac{16}{2^a}\delta(\xi_{0,1})\int_{\R}dX \int_{\R}  dY & \, \frac{ F(X+Y, \tau)F(X-Y,\tau) }{|Y_{1}|^{a}} 
\, \int_0^{2\pi} d\vp \int_0^{\pi} d\theta \frac{ \sin\theta}{|Y_1| t^2\sin^2(2\theta)}\notag \\& 
\Bigg[ 
\delta(\vp-\psi_0)\delta\left(Y_1-\frac{|\tilde \xi_0|}{t\sin(2\theta)}\right)  +\delta(\vp-\psi_0+\pi)\delta\left(Y_1+\frac{|\tilde \xi_0|}{t\sin(2\theta)}\right) \Bigg]\, .
\end{align}
We now compute the full integral above. Noticing that $ \int_0^{2\pi} d\vp\, \delta(\vp-\psi_0)= \int_0^{2\pi} d\vp\,\delta(\vp-\psi_0+\pi)=1$ and using the symmetry of the integral, it follows that 
\begin{align}\label{eq:strongFcoll3}
& \frac{32}{2^a}\delta(\xi_{0,1})\int_{\R}dX \int_{\mathbb{R}^2}  dY_2\,dY_3\,\int_0^{\pi} d\theta  \Bigg[ F\left(X_1+\frac{|\tilde \xi_0|}{t\sin(2\theta)}  ,X_2+Y_2,X_3+Y_3,\tau\right) \times \notag\\& \qquad   F\left(X_1-\frac{|\tilde \xi_0|}{t\sin(2\theta)},X_2-Y_2,X_3-Y_3,\tau\right) 
  \frac{ \sin\theta} {t^2\sin^2(2\theta)} \left(\frac{ t\, \vert \sin(2\theta)\vert }{|\tilde \xi_0| }\right)^{a+1} \Bigg]\notag \\& 
  =\frac{32}{2^a}\, t^{a-1}\,\delta(\xi_{0,1})\int_{\R}dX \int_{\mathbb{R}^2}  dY_2\,dY_3\,\int_0^{\pi} d\theta  \Bigg[ F\left(X_1+\frac{|\tilde \xi_0|}{t\sin(2\theta)}  ,X_2+Y_2,X_3+Y_3,\tau\right) \times \notag\\& \qquad   F\left(X_1-\frac{|\tilde \xi_0|}{t\sin(2\theta)},X_2-Y_2,X_3-Y_3,\tau\right) 
  \frac{ \sin\theta \, \vert \sin(2\theta)\vert^{a-1} }{  |\tilde \xi_0| ^{a+1}} \Bigg]\notag
\end{align}
where $\tilde \xi=(\xi_2, \xi_3)$ and this gives the strong form of the gain term in \eqref{eq:ss_strong2_gain}.

\section{Reduction to a two-dimensional problem} 
\label{sec:redupb}

To analyze the long-time behaviour of the solution $F=F(\xi_1,\xi_2,\xi_3,\tau)$ to \eqref{eq:ss_strong2}, we will first consider a reduced two-dimensional problem, integrating out the $\xi_3$ variable.  Afterwards, we will determine the asymptotic behavior of $F(\xi_1, \xi_2, \xi_3, \tau)$ (see Section \ref{sec:asymF} below) exploiting the asymptotics of the solution $G = G(\xi_1, \xi_2, \tau)$ of the reduced problem, obtained herein. We note that the function 
$G$ itself satisfies a kinetic equation. More specifically, we define
\begin{equation}\label{def:G} 
    G(\xi_1,\xi_2, \tau) =\int_{\Rl} F(\xi_1,\xi_2,\xi_3,\tau)\, d \xi_3\, .
\end{equation}
Notice that, with a slight abuse of notation we will still denote by $\xi$ the two dimensional variable, i.e. $\xi=(\xi_1,\xi_2)\in \mathbb{R}^2$. The precise meaning of $\xi$ will be made clear by the context in each specific case. 

\subsection{Evolution equation for the reduced distribution $G$} 
\label{sse:eqG}
Assuming that $F$ decays sufficiently rapidly at infinity, the reduced distribution $G$ defined as in \eqref{def:G} then satisfies
\begin{equation}
 \label{eq:G} 
    \partial_{\tau}G-\frac{1}{a}\partial_{\xi_1}(\xi_1 G)-\left(\frac{1}{a}-1\right)\partial_{\xi_2}(\xi_2 G)
   +  \partial_{\xi_1}(\xi_2 G )=  {K}^+[G,G](\xi)- {K}^-[G,G](\xi), \ \, G=G(\xi,\tau), \ \xi\in \mathbb{R}^2
\end{equation}
where the gain term is 
\begin{align}\label{eq:G_gainbis}
  {K}^+[G,G](\xi)& 
 =\frac {8 \sqrt{2}}{2^a \, t}\delta(\xi_1)  \int_{\mathbb{R}} dx_1 \int_{\mathbb{R}} d\eta_2 \int_{\mathbb{R}}d\zeta_2 \int_{\frac{ |\xi_2|}{t}}^{+\infty} dz \, G\left( x_1+ z, \eta_2, \tau\right) G\left(x_1- z,\zeta_2,\tau \right) \frac{1}{z^{a+1}}\Phi\left(\frac{ |\xi_2|}{tz}\right)  
  \end{align}
with 
\begin{equation}\label{def:auxPhi}
     \Phi\left(s\right):= \int_{ s}^{1}   \frac{1}{\sqrt{1-\rho^2}} 
      \frac{  \sqrt{1 - \sqrt{1 - \rho^2} } + \sqrt{1 + \sqrt{1 - \rho^2}}} { \sqrt{\rho^2-s^2}} d\rho  , \quad 0<s<1  \,
\end{equation} 
and the loss term  reads 
\begin{align}\label{eq:G_loss}
  {K}^-[G,G](\xi)& 
  = G(\xi,\tau)\Lambda(\xi_1,\tau) \, 
\end{align}
with 
\begin{equation}
\label{def:LAMBDA}
\Lambda(\xi_1,\tau):=8\pi \int_{\mathbb{R}^2} d\eta \,|\xi_1-\eta_{1}|^{-a} G(\eta,\tau), \quad \xi=(\xi_1,\xi_2)\in \mathbb{R}^2 \, .
\end{equation}

The justification of \eqref{eq:G_loss}-\eqref{eq:G_gainbis} starting from \eqref{eq:ss_strong2} can be found in Appendix \ref{ss:justificationGeq}. We further notice that this equation satisfies the mass conservation property, as it will be shown in Appendix \ref{ssec:massG}.
\medskip

 \subsection{Reformulation of \eqref{eq:G} as a boundary value problem} \label{ssec:bdvalue}

 We consider the equation satisfied by the reduced distribution $G$, i.e. \eqref{eq:G}, and observe that it can be rewritten as a suitable boundary value problem. 
 More precisely, the boundary value we study is the following:
 \begin{align} 
 \label{eq:evG}
    \partial_{\tau}G&-\frac{1}{a}\partial_{\xi_1}(\xi_1G)-\left(\frac{1}{a}-1\right)\partial_{\xi_2}(\xi_2G) +  \partial_{\xi_1}(\xi_2 G )=  - {K}^-[G,G](\xi),  
    \\&
    G(0^+,\xi_2,\tau)= e^{-2\tau} H\left(\frac {\xi_2}{e^{\tau} }, \tau\right), \ \ \tau=\log(t)   
\label{eq:bdVG} \end{align}
where, denoting by $u:=\frac {\xi_2}{e^\tau}$, the function $H$ is given by 
\begin{equation}
\label{def:H}  
H(u,\tau):= \frac {8 \cdot \sqrt{2} }{2^a \, u} \! \!  \int_{\mathbb{R}} \! \! dx_1 \! \!\int_{\mathbb{R}}\! \! d\eta_2 \! \!\int_{\mathbb{R}} \! \!d\zeta_2  \! \!\int_{ |u|}^{+\infty} \! \! dz \, G\left( x_1+ z, \eta_2, \tau\right) G\left(x_1- z,\zeta_2,\tau \right) \frac{1}{z^{a+1}}\Phi\left(\frac{|u|}{z}\right)
\end{equation}
with $\Phi\left(s\right)$ defined as in \eqref{def:auxPhi} 
and where ${K}^-[G,G](\xi)$ is the loss term given as in \eqref{eq:G_loss}. 
\smallskip 

 To gain insight into the onset of  \eqref{eq:evG}, \eqref{eq:bdVG}, note that in the structure of the gain term $K^+[G,G]$ in \eqref{eq:G_gainbis}, due to the presence of $\delta(\xi_1)$, results in a discontinuity, or jump, at $\xi_1=0$. In fact, to obtain \eqref{eq:evG}, \eqref{eq:bdVG} we start from \eqref{eq:G} and, 
 assuming that $G(0^+, \xi_2, \tau)=0$ we compute  $G(0^+, \xi_2, \tau)$ and obtain 
\begin{equation}\label{eq:Gbdval} 
G(0^+\! \!, \xi_2, \tau)=
\frac{1}{\xi_2}\frac {8 \cdot \sqrt{2} }{2^a \, t} \! \!  \int_{\mathbb{R}} \! \! dx_1 \! \!\int_{\mathbb{R}}\! \! d\eta_2 \! \!\int_{\mathbb{R}} \! \!d\zeta_2  \! \!\int_{\frac{ |\xi_2|}{t}}^{+\infty} \! \! dz \, G\left( x_1+ z, \eta_2, \tau\right) G\left(x_1- z,\zeta_2,\tau \right) \frac{1}{z^{a+1}}\Phi\left(\frac{|\xi_2|}{tz}\right), \ \   \xi_2>0
\end{equation}
Notice that, due to the symmetry of the problem, we can restrict to the region where $\xi_2$ take positive values, i.e. $\xi_2>0$. We further observe that this formula has a self-similar structure since it can be rewritten as 
$$ G(0^+,\xi_2,\tau)  = \frac 1 {e^{2\tau}} H\left(\frac {\xi_2}{e^{\tau}}, \tau\right), \ \  {\tau}={\log(t)}\, ,$$
where $H$ is defined as in \eqref{def:H}. 
\smallskip

We observe that, as a matter of fact, the function $H\left(u, \tau\right)$ introduced in \eqref{eq:bdVG} and appearing in the equation above, must be modified by means of some logarithmic corrections. More precisely, the function $H\left(u, \tau\right)$ will have the approximated form 
$$H\left(u, \tau\right)\sim \lambda( \tau) \overline{H}\left(\frac {u}{\varepsilon(\tau)}\right)$$
where 
$$\lambda( \tau) \approx \frac 1 \tau \qquad \text{and}\qquad \varepsilon(\tau)\approx\frac 1 {\tau^{\frac 1 a}}\,.$$
The reason for the onset of the function $\varepsilon(\tau)$ is that most of the collisions between particles take place for particles with sizes of order $|\xi_1| \approx \varepsilon(\tau)$ and, as a result of these collision,  they acquire a size of order $|\xi| \approx \varepsilon(\tau)e^\tau$. On the other hand, the function $\lambda(\tau)$ will be determined using the mass conservation for the reduced distribution $G$ as well as the fact that $G$ behaves for large values of $|\xi|$ as $G(\xi,\tau)\approx \frac 1 {|\xi_1|}\delta(\xi_1-\xi_2)$ and that the mass of $G$ is concentrated in a region of size $\varepsilon(\tau)\lesssim |\xi|\lesssim \varepsilon(\tau)e^\tau$. As a consequence of this, a more accurate description of the self-similarity properties of $G(0^+,\xi_2,\tau)$ is given by $$G(0^+,\xi_2,\tau)\sim \frac 1 {e^{2\tau}(\varepsilon(\tau))^2}H\left(\frac {\xi_2}{\varepsilon(\tau)e^{\tau}},\tau\right)\,.$$


\subsection{The system of characteristics associated with  \eqref{eq:evG}, \eqref{eq:bdVG}}\label{ssec:char}

We consider the system of characteristics associated to \eqref{eq:evG}, \eqref{eq:bdVG}. It reads as
\begin{align}\label{eq:systemchar}
    \frac{d \xi_1}{d\tau} & = \xi_2  -\frac{1}{a} \xi_1; \notag \\
    \frac{d \xi_2}{d\tau} & = -\left(\frac{1}{a}-1\right) \xi_2; \notag \\
    \frac{d G}{d\tau} & =  \left(\frac{2}{a}-1\right)G- K^-[G,G]
\end{align}
with initial data, at time  $\tau_0$,  
$\xi_{1}(\tau_0)=0, \,\xi_{2}(\tau_0)=\xi_{2,0}$. The system above  
can be solved for $0 \leq \tau_0 < \tau$. We first take into account the first two equations in \eqref{eq:systemchar} and, by integrating, we obtain 
\begin{align}
    \xi_1(\tau) &=\xi_{2,0}  e^{-\left(\frac{1}{a}-1\right)(\tau-\tau_0)}- \xi_{2,0} e^{-\frac{1}{a}(\tau-\tau_0)}= \xi_{2,0} e^{-\frac{1}{a}(\tau-\tau_0)} \left(e^{(\tau-\tau_0)}-1 \right)\label{eq:xi1}; \\
    \xi_2(\tau) & = \xi_{2,0} e^{-\left(\frac{1}{a}-1\right)(\tau-\tau_0)}\label{eq:xi2}
\end{align}
which yields
\begin{equation} \label{eq:xi1xi2}
 \frac{\xi_1}{\xi_2}=1- e^{- (\tau-\tau_0)}  \text{ or, equivalently, } e^{(\tau-\tau_0)}= \frac{\xi_2}{\xi_2-\xi_1}\,
\end{equation}
where we are not writing explicitly the dependence of $\xi_1, \, \xi_2$ on $\tau$. 
We further observe that equations  \eqref{eq:xi1}, \eqref{eq:xi2} retain the same structure in both linear and nonlinear frameworks, and thus coincide with the corresponding equations derived in \cite{MNV25b}  for the linear Boltzmann-Rayleigh equation.  

Returning now to the third equation in \eqref{eq:systemchar}, i.e., the equation for $G$, we recall that $K^-[G,G]=\Lambda(\xi,\tau)\, G(\xi,\tau)$ 
with $\Lambda$ as in \eqref{def:LAMBDA} 
 and thus 
\begin{equation}
G(\xi_1,\xi_2,\tau) = G(0^+,\xi_{2,0},\tau_0) \exp\left(\left(\frac{2}{a}-1\right)(\tau-\tau_0) 
- \int_{\tau_0}^{\tau} \Lambda(\xi_1(s),\tau)
ds\right)\, .\label{eq:solChar}
\end{equation}
Inserting \eqref{eq:bdVG} with $\tau=\tau_0$ in \eqref{eq:solChar} we arrive at   
\begin{align}
G(\xi_1,\xi_2,\tau) = &  {e^{-2\tau_0}} H\left(\frac {\xi_{2,0}}{e^{\tau_0}}, \tau_0\right) \,\exp\left(\left(\frac{2}{a}-1\right)(\tau-\tau_0) 
- \int_{\tau_0}^{\tau} \Lambda(\xi_1(s),\tau) ds\right)\,  \nonumber \\&
=  {e^{-2\tau_0}} H\left(\frac {\xi_{2,0}}{e^{\tau_0}}, \tau_0\right) \,  \left(\frac{\xi_2}{\xi_2-\xi_1}\right)^{\left(\frac{2}{a}-1\right)} 
\exp\left(
-  \int_{\tau_0}^{\tau} \Lambda(\xi_1(s),\tau) ds\right)\,
\label{eq:solChar2}
\end{align}
where, in the second identity, we used \eqref{eq:xi1xi2}.

\section{Rationale behind the asymptotic form of the solution derived in this paper}\label{sec:heuristic}  

The aim of this description is to provide, heuristically, the general idea behind the asymptotic behavior of $G(\xi,\tau)$, $\xi=(\xi_1,\xi_2)\in \mathbb{R}^2$,  solution to the reduced model \eqref{eq:evG}, \eqref{eq:bdVG} that will be obtained in this paper. The explicit asymptotics for $G$ will then allow to obtain the asymptotic behavior of the original solution $F(\xi, \tau)$, $\xi=(\xi_1, \xi_2,\xi_3)\in\mathbb{R}^3$, using the behavior of $G$. In this section we just summarize the main ideas behing the asymptotics that we obtain, the details will be presented in the subsequent sections. 
We start recalling the equation \eqref{eq:evG} satisfied by $G$, namely 
\begin{equation*}
    \partial_{\tau}G - \frac{1}{a}\partial_{\xi_{1}}(\xi_{1}G) - (\frac{1}{a}-1)\partial_{\xi_{2}}(\xi_{2}G) + \partial_{\xi_{1}}(\xi_{2}G) = K^{+}[G,G](\xi) - K^{-}[G,G](\xi) 
\end{equation*}
where the gain and loss term (cf.\eqref{eq:G_gainbis}-\eqref{eq:G_loss}) respectively read
\begin{equation*}
    K^{+}[G,G](\xi) = \frac{8\sqrt{2}}{2^{a}t}\delta(\xi_{1})\int_{\mathbb{R}}dx_{1}\int_{\mathbb{R}}d\eta_{2}\int_{\mathbb{R}}d\zeta_{2}\int_{\frac{ |\xi_{2}|}{t}}^{+\infty}dz G(x_{1}+z,\eta_{2},\tau)G(x_{1}-z,\zeta_{2},\tau)\frac{1}{z^{a+1}}\Phi\left(\frac{ |\xi_{2}|}{tz}\right)  \, ,
\end{equation*}
\begin{equation*}
    K^{-}[G,G](\xi) = 8\pi G(\xi, \tau) \int_{\mathbb{R}^{2}}d\eta |\xi_{1}-\eta_{1}|^{-a}G(\eta, \tau), \quad \xi=(\xi_{1}, \xi_{2}) \in \mathbb{R}^{2} \, ,
\end{equation*}
with the auxiliary function $\Phi(s)$ given as follows:
$$
    \Phi(s) = \int_{s}^{1}\frac{1}{\sqrt{1-\rho^{2}}}\frac{\sqrt{1-\sqrt{1-\rho^{2}}}+\sqrt{1+\sqrt{1-\rho^{2}}}}{\sqrt{\rho^{2}-s^{2}}}d\rho, \quad 0 < s < 1\, .
$$

We notice that the gain term \eqref{eq:G_gainbis} can be removed from the equation and replaced by a boundary condition at $\xi_{1}=0^+$, $\xi_2>0$.  We now introduce the following decomposition of $\mathbb{R}^2$ that will be useful later: 
\begin{align}
& \mathcal{R}_1:=\{\xi_{1}>0, \, \xi_{2}>0\}\,,\qquad    \mathcal{R}_2:=\{\xi_{1}<0, \, \xi_{2}<0 \}\nonumber\\&
\mathcal{R}_3:=\{\xi_1>0,\, \xi_2<0\}\, , \qquad \mathcal{R}_4:=\{ \xi_{1}<0, \, \xi_{2}>0\}\,,
\end{align}
and also 
\begin{equation}
  \mathcal{R}_{1,2}:=\{0 < \xi_{2}<\xi_{1}\}\, , \qquad \mathcal{R}_{2,2}:=\{\xi_{1}< \xi_{2}<0\}\,.
\end{equation}
  
We then observe  that  we can restrict ourselves to examine the behavior of the solution in the region $\mathcal{R}_1$. Indeed, the function $G(\xi_1,\xi_2,\tau)$ is negligible in the regions 
 $\mathcal{R}_3\,\, \mathcal{R}_4 $
 since the characteristic curves associated with \eqref{eq:evG} in $\{\xi_1=0\}$ do not enter those regions. 
 The values of $G$ in $\mathcal{R}_2$ can be obtained from the values $\mathcal{R}_1$ due to the symmetry condition $G(-\xi,\tau)=G(\xi,\tau)$, $\xi\in \mathbb{R}^2$. 
Moreover, the characteristic curves of \eqref{eq:evG} starting at $\xi_1=0$ do not enter the regions $\mathcal{R}_{1,2}$ and $\mathcal{R}_{2,2}$. Therefore, $G$ will be negligible also in $\mathcal{R}_{1,2}\,,\, \mathcal{R}_{2,2}$. 

The solution $G$ is then constructed by means of the evolution by characteristics combined with jumps describing the collisions. These collisions are relevant for $|\xi| \approx \varepsilon(\tau)$ where $\varepsilon(\tau)$ will be determined later.  

The mechanism of collisions sends particles from the region $|\xi| \approx \varepsilon(\tau)$ to the line $\{\xi_{1}=0\}$. This will provide a relation between the values of $G(\xi,\tau)$ for $|\xi| \approx \varepsilon(\tau)$ and $G(0^+,\xi_2,\tau)$ with $|\xi_2|$ of order $\varepsilon(\tau)e^\tau$, given that, due to the form of the gain operator $K^{+}[G,G]$ (cf. \eqref{eq:G_gain})
the collisions between two particles with $|\xi| \approx \varepsilon(\tau)$ yield two particles with $|\xi| $ of order $ \varepsilon(\tau)e^\tau=\varepsilon(\tau) t$, one of them with $\xi_2>0$ and the other with $\xi_2<0$. It is relevant to notice that most of the contribution to the 
$L^1$ norm of the function $G(0^+,\xi_{2},\tau)$ comes from the region $|\xi_{2}| \lesssim \varepsilon(\tau)e^{\tau}$.
Additionally, the form of $K^{+}[G,G]$ implies that $G(0^+, \xi_{2}, \tau)$ has a self-similar form:
\begin{equation}\label{eq:Gss}
    G(0^+, \xi_{2}, \tau) = \frac{\lambda(\tau)}{(\varepsilon(\tau)e^{\tau})^{2}} \overline{H}\left(\frac{\xi_{2}}{\varepsilon(\tau)e^{\tau}}\right)  
\end{equation}
where $\overline{H}$ and $\lambda(\tau)$ will be determined later.

It is possible to compute $G(\xi,\tau)$ in the region $\mathcal{R}_{1,1}\cup \mathcal{R}_{2,1}:=\{0<\xi_{1}< \xi_{2}\}\cup \{\xi_{2}< \xi_{1}<0\}$ integrating by characteristics the equation 
\begin{equation*}
    \label{eq:reducedbdG}
    \partial_{\tau}G - \frac{1}{a}\partial_{\xi_{1}}(\xi_{1}G) - (\frac{1}{a}-1)\partial_{\xi_{2}}(\xi_{2}G) + \partial_{\xi_{1}}(\xi_{2}G) =- K^{-}[G,G](\xi) \,.
\end{equation*}

We observe that the loss term $K^-[G,G]$ yields a negligible contribution for $|\xi| \gg \varepsilon(\tau)$. Actually, the characteristic scale $\varepsilon(\tau)$ will be defined as the scale such that for values $|\xi| \approx \varepsilon(\tau)$ the term $K^-[G,G]$ yields a contribution to $G$ of order one. 
As it might be expected the values of $G(\xi,\tau)$ with 
$|\xi| \approx \varepsilon(\tau)$ yield also the main contribution to $K^+[G,G]$. This contribution would be relevant in the region where $|\xi_2|\lesssim \varepsilon(\tau)e^\tau$. 

Integrating by characteristics \eqref{eq:reducedbdG} we obtain an approximation for $G(\xi,\tau)$ in $\{0<|\xi_{1}|< |\xi_{2}|\}$. The characteristics associated with \eqref{eq:reducedbdG}, starting at $\xi_1=0$, $\xi_2=\xi_{2,0}$ at time $\tau = \tau_{0}$ (see \eqref{eq:xi1xi2}), converge towards the diagonal $\{\xi_{1}=\xi_{2}\}$ for $\tau - \tau_{0}\gg 1$.

This allows the approximation:
\begin{equation}\label{eq:apprG2}
G(\xi,\tau)\sim a\,\lambda(\tau_{0}(\tau, \xi)) \frac{\delta(\xi_{1} - \xi_{2})}{|\xi_{2}|} \sigma\left(\frac{\xi_{2}}{\varepsilon(\tau)}, \tau\right) \quad \text{as}\quad \tau\to \infty\,,
\end{equation}
where $\lambda(\cdot)$ is as in \eqref{eq:Gss} and $\tau_{0}(\tau, \xi)$ is the time at which a characteristic curve that approaches to the set $\{\xi_{1}=\xi_{2}\}$ at time $\tau$ starts its motion at the line $\xi_1=0$, for some $\xi_2=\xi_{2,0}$. Notice that $\tau_{0}(\tau, \xi)<\tau$. In addition, for the values of $\xi$ for which $G(\xi,\tau)$ yields a relevant contribution we have that $\tau_{0}(\tau, \xi) \to \infty$ as $\tau\to \infty$. 

The function $\sigma\left(s, \tau\right)$ with $s=\frac{\xi_{2}}{\varepsilon(\tau)}$ appearing in \eqref{eq:apprG2} is approximately equal to one if $\varepsilon(\tau) \ll|\xi_2| \ll \varepsilon(\tau)e^\tau $. 
This function provides a cutoff for $G$ for the values $ |\xi_2| \approx \varepsilon(\tau)$ and $ |\xi_2| \gtrsim \varepsilon(\tau)e^\tau$. The cutoff for  $ |\xi_2| \gtrsim \varepsilon(\tau)e^\tau$ is due to the fact that $G(0,\xi_{2,0},\tau)$ is negligible if $|\xi_{2,0}| \gtrsim \varepsilon(\tau)e^\tau$ and, as a consequence, integrating by characteristics, it follows that $G(\xi,\tau)$ is negligible for $|\xi| \gtrsim \varepsilon(\tau)e^\tau$. Furthermore, the function $\sigma\left(s, \tau\right)$ with $s=\frac{\xi_{2}}{\varepsilon(\tau)}$, provides also a cutoff for $G(\xi,\tau)$ in the region where $|\xi_2| \lesssim \varepsilon(\tau)$. This cutoff is due to the fact that collisions between particles with sizes $|\xi| \approx \varepsilon(\tau)$ become relevant. It is possible to show that the function $\sigma\left(s, \tau\right)$ can be approximated by means of the solution ${\sigma}\left(s \right)$ of an integro-differential equation that will be analysed in Section \ref{ss:cutoff}. 
It turns out that for values $|\xi|\approx \varepsilon(\tau)$, the function $\tau_{0}(\tau, \xi)$ in \eqref{eq:apprG2} can be approximated as  
$$\tau_{0}(\tau, \xi) \sim (1-a)\tau \quad \text{as} \quad \tau \to \infty\,.$$
Then, for $|\xi|\approx \varepsilon(\tau)$, \eqref{eq:apprG2} becomes 
\begin{equation}\label{eq:apprG3}
G(\xi,\tau)\sim a \,\lambda((1-a)\tau) \frac{\delta(\xi_{1} - \xi_{2})}{|\xi_{2}|} {\sigma}\left(\frac{\xi_{2}}{\varepsilon(\tau)}\right) \quad \text{as}\quad \tau\to \infty\,,
\end{equation}

Using \eqref{eq:apprG3} and the form of the gain term $K^+[G,G]$\eqref{eq:G_gainbis}  we can thus derive \eqref{eq:Gss}. More precisely, we obtain 
\begin{equation}
   G(0^+, \xi_{2}, \tau) \sim a^2\,\frac{(1-a)}{a}\frac{(\lambda((1-a)\tau)^{2}}{(\varepsilon(\tau))^{a}} 
{\frac{e^{-2\tau}}{(\varepsilon(\tau))^2}}
   \overline{H}\left(\frac{\xi_{2}}{\varepsilon(\tau)e^{\tau}}\right) \label{eq:apprG4}
\end{equation}

Combining \eqref{eq:Gss} with \eqref{eq:apprG4} we finally obtain
\begin{equation}\label{eq:lambdaeps}
\lambda(\tau) \sim a \, (1-a)\,\frac{(\lambda((1-a)\tau)^{2}}{(\varepsilon(\tau))^a} \quad \text{as} \quad \tau \to \infty\, .
\end{equation}
This equation will allow to obtain $\varepsilon(\tau)$ once $\lambda(\tau)$ has been obtained. 

On the other hand, we can derive the asymptotic behaviour of $\lambda(\tau)$ as $\tau \to \infty$ using the mass conservation property of $G$ given as in \eqref{eq:apprG2}. More precisely, using the assumption $\int_{\mathbb{R}^2} G(\xi,\tau) d\xi=1$, combined with \eqref{eq:apprG2} it follows
\begin{equation}\label{eq:asymGmass}
 a   \int_{\mathbb{R}^2} \lambda(\tau_{0}(\tau, \xi)) \frac{\delta(\xi_{1} - \xi_{2})}{|\xi_{2}|} \sigma\left(\frac{\xi_{2}}{\varepsilon(\tau)}, \tau\right) d\xi \sim 1 \quad \text{as}\quad \tau\to \infty\,
\end{equation}
where the factor $2$ is due to the symmetry $\xi \to -\xi$. Moreover, using the fact that $\sigma$ yields a cutoff for values of $\xi$ such that $|\xi|\lesssim \varepsilon(\tau)$, $|\xi|\gtrsim \varepsilon(\tau)e^\tau$
we obtain from \eqref{eq:asymGmass} that 
\begin{equation}\label{eq:asymGmass2}
   2a \int_{\varepsilon(\tau)}^{\varepsilon(\tau)e^\tau} \frac{\lambda(\tau_{0}(\tau, \xi)) }{|\xi_{2}|}  d\xi_2 \sim 1 \quad \text{as}\quad \tau\to \infty\,
\end{equation}
We use the ansatz 
\begin{equation}
\label{eq:ansatzlambda}
\lambda(\tau)\sim \frac{C_0}{\tau}\ \text{as} \ \tau\to \infty 
\end{equation}
where $C_0>0$ is a numerical constant to be determined later.  
The relation between the time $\tau$ and the time $\tau_0$, is given as in \eqref{eq:xi1xi2}, namely 
\[  \tau_0 = \tau-\log\left(\frac{\xi_2}{\xi_2-\xi_1}\right)=  \tau-\log(\xi_2) - \log(\xi_2 - \xi_1) \]
Moreover, we notice that, using the thickness of the Dirac delta in \eqref{eq:apprG2}, it is possible to show (see \eqref{eq:apprGdelta1} below)  that  
\begin{equation}\label{eq:lambdathick}
 \lambda(\tau_0(\tau,\xi)) = \lambda\left(  \tau-\log(\xi_2) - \log(\xi_2 - \xi_1) \right)=\lambda\left((1-a)\tau+a \log(\xi_2)\right).
\end{equation}

In order to determine $C_0$  we consider  \eqref{eq:asymGmass2} and, using \eqref{eq:ansatzlambda} combined with \eqref{eq:lambdathick}, the l.h.s. of \eqref{eq:asymGmass2} becomes 
\begin{align}
\label{eq:asymGmass3}   
& 2a \int_{\varepsilon(\tau)}^{\varepsilon(\tau)e^\tau} \frac{\lambda(\tau_{0}(\tau, \xi)) }{|\xi_{2}|}  d\xi_2 \sim 
2a \,C_0\int_{\varepsilon(\tau)}^{\varepsilon(\tau)e^\tau} \frac{1 }{|\xi_{2}|\left( (1-a)\tau+a \log(\xi_2)\right)}  d\xi_2 \nonumber \\& 
\sim 2a \,C_0\int_{1}^{e^\tau} \frac{1 }{|\zeta |\left( (1-a)\tau+a \log(\varepsilon(\tau)+a \log (\zeta))\right)}d\zeta 
\end{align}
where in the second line we used the change of variables $\xi_2=\varepsilon(\tau)\zeta$. We observe that the term $a \log(\varepsilon(\tau)$ in the denominator of \eqref{eq:asymGmass3} is negligible as $\tau \to \infty$ and thus  \eqref{eq:asymGmass3} reduces to 
\begin{align}
\label{eq:asymGmass4}   
&  2a \, C_0\int_{1}^{e^\tau} \frac{1 }{|\zeta |\left( (1-a)\tau+a \log (\zeta))\right)}d\zeta =\frac {2a \,C_0}{\tau} \int_{1}^{e^\tau} \frac{1 }{|\zeta |\left( (1-a)+\frac{a \log (\zeta))}{\tau}\right)}d\zeta \nonumber \\& 
=\frac {2a \,C_0}{\tau} \int_{0}^{\tau} \frac{e^\theta}{e^\theta \left( (1-a)+\frac{a \theta}{\tau})\right)}d\theta  
= \frac {2a \,C_0}{\tau}\int_{0}^{1} \frac{\tau}{\left( (1-a)+{a \beta}\right)}d\beta \nonumber \\&  = 2a \, C_0 \int_{0}^{1} \frac{1}{\left( (1-a)+{a \beta}\right)}d\beta  = 2a \, C_0 \log \left(\frac{(1-a)+a}{1-a}\right)=  2a \, C_0 \log \left(\frac{1}{1-a}\right) 
\end{align}
where in the second identity we set $\zeta=e^\theta$ with $d\zeta=e^\theta d\theta$
and in the third one we set $\theta/\tau=\beta$ with $d\theta=\frac 1 \tau d\beta$. 
Combining \eqref{eq:asymGmass4} with \eqref{eq:asymGmass2} we then obtain
$ 2a \, C_0 \log \left(\frac{1}{1-a}\right) \sim 1$ which then yields 
\begin{equation}\label{def:Co}
 C_0 =\frac 1 {2a \,\log \left(\frac{1}{1-a}\right)}.
\end{equation}
We observe that all the dyadic intervals, or scales, for $|\xi_2|\in[\varepsilon(\tau),\varepsilon(\tau)e^\tau]$ contribute
to the mass of $G(\xi,\tau)$. On the other hand, $\lambda(\tau_{0}(\tau, \xi))$ has a nontrivial dependence on $\xi_2$ in these scales. In particular, we remark that we cannot use neither the approximation $\lambda(\tau_{0})=\frac {C_0}{\tau}$ or $\lambda(\tau_{0})=\frac {C_0}{(1-a)\tau}$ in the whole integration domain, we  can use instead that $\lambda(\tau)=\frac {C_0} {\tau}$ but the contribution of $\lambda(\tau_0(\tau,\xi))$ depends on $\xi_2$: if $\xi_2\sim e^{\theta \tau}$ then $\lambda(\tau_0(\tau,\xi))$  depends on $\theta$. 
On the contrary, in the gain term $K^+[G,G]]$ we can use the approximation $\tau_0\sim (1-a)\tau$, $\lambda(\tau_0) \sim \lambda((1-a)\tau)$, in order to obtain the boundary value $G(0^+,\xi_2,\tau)$.

\section{Asymptotics for the solution $G$ of the reduced 2d model} \label{sec:asymG}
 
We consider the ``reduced'' evolution equation satisfied by $G=G(\xi,\tau), \ \xi\in \mathbb{R}^2$ that reads
\begin{equation}
 \label{eq:G_a} 
    \partial_{\tau}G-\frac{1}{a}\partial_{\xi_1}(\xi_1 G)-\left(\frac{1}{a}-1\right)\partial_{\xi_2}(\xi_2 G)
   +  \partial_{\xi_1}(\xi_2 G )=  {K}^+[G,G](\xi)- {K}^-[G,G](\xi), 
\end{equation}
where the gain term is 
\begin{align}\label{eq:G_gainbis_a}
  {K}^+[G,G](\xi)& 
 =\frac {8\sqrt{2}} {2^a \, t}\delta(\xi_1)  \int_{\mathbb{R}} dx_1 \int_{\mathbb{R}} d\eta_2 \int_{\mathbb{R}}d\zeta_2 \int_{\frac{ |\xi_2|}{t}}^{+\infty} dz \, G\left( x_1+ z, \eta_2, \tau\right) G\left(x_1- z,\zeta_2,\tau \right) \frac{1}{z^{a+1}}\Phi\left(\frac{|\xi_2|}{tz}\right)  
  \end{align}
with 
\begin{equation}\label{def:auxPhi_bis}
     \Phi\left(s\right):= \int_{ s}^{1}   \frac{1}{\sqrt{1-\rho^2}} 
      \frac{  \sqrt{1 - \sqrt{1 - \rho^2} } + \sqrt{1 + \sqrt{1 - \rho^2}}} { \sqrt{\rho^2-s^2}} d\rho  , \quad 0<s<1  \,
\end{equation} 
and the loss term  reads 
\begin{align}\label{eq:G_loss_bis}
{K}^-[G,G](\xi)&  =8\pi G(\xi,\tau)\int_{\mathbb{R}^2} d\eta \,|\xi_1-\eta_{1}|^{-a} G(\eta,\tau), \quad \xi=(\xi_1,\xi_2)\in \mathbb{R}^2 \, .
\end{align}
\smallskip

\noindent \textbf{Ansatz 1:} we assume that in the gain term the main contribution is due to a region where the measure $G$ can be approximated as the product of a Dirac delta measure and a measure $Q$ which depends only on the second component of the velocity variable, specifically  
\begin{equation}\label{eq:ansatzG2}
    G(\xi_1,\xi_2,\tau)=\delta(\xi_1-\xi_2)Q(\xi_2,\tau)
\end{equation} 
where $Q$ is symmetric in $\xi_2$ and it approaches to zero in the cutoff region where $|\xi_2|\approx \varepsilon (\tau)$ that will be computed in detail later. More precisely, the heuristic argument in Section \ref{sec:heuristic} would suggest to look for functions $Q$ with the following functional dependence
\begin{equation}\label{eq:Q}
    Q(\xi_2,\tau)= \overline{Q}\left(r, \tau\right), \quad r= \frac{\xi_2}{\varepsilon(\tau)}
\end{equation} 
and such that $\overline{Q}(r, \tau)\to 0$ as $r\to 0$ and $\overline{Q}(r, \tau)\to 1$ as $r\to \infty$. Moreover, we notice that $\overline{Q}(r, \tau)$ has an additional cutoff when $r\approx e^\tau$. Therefore, we consider solutions $G$ with the following form
\begin{equation}\label{eq:ansatzG}
    G(\xi_1,\xi_2,\tau)=\delta(\xi_1-\xi_2)\overline{Q}\left(\frac{\xi_2}{\varepsilon(\tau)}, \tau\right)
\end{equation} 
\smallskip

Using \eqref{eq:ansatzG2} in \eqref{eq:G_gainbis_a} we obtain 
\begin{align}\label{eq:G_gain_a}
  {K}^+[G,G](\xi)& 
 =\frac {8  \sqrt{2}} {2^a \, t}\delta(\xi_1)  \int_{\mathbb{R}} dx_1 \int_{\mathbb{R}} d\eta_2 \int_{\mathbb{R}}d\zeta_2 \int_{\frac{ |\xi_2|}{t}}^{+\infty} dz \, \delta(x_1+ z-\eta_2)\,Q(\eta_2,\tau)\nonumber\\& \qquad \times \ \delta(x_1- z-\zeta_2)\,Q(\zeta_2,\tau) \frac{1}{z^{a+1}}\Phi\left(\frac{|\xi_2|}{tz}\right)  \nonumber\\&
 = \frac {8 \sqrt{2}} {2^a \, t}\delta(\xi_1)  \int_{\mathbb{R}} d\eta_2 \int_{\mathbb{R}}d\zeta_2 \int_{\frac{ |\xi_2|}{t}}^{+\infty} dz \, Q(\eta_2,\tau)\, \delta(\eta_2-\zeta_2-2z)\, Q(\zeta_2,\tau) \frac{1}{z^{a+1}}\Phi\left(\frac{ |\xi_2|}{tz}\right)  \nonumber\\&
 = \frac {8  \sqrt{2}} {2^a \, t}\delta(\xi_1)  \int_{\mathbb{R}} d\eta_2 \int_{\frac{ |\xi_2|}{t}}^{+\infty} dz \, Q(\eta_2,\tau) \, Q(\eta_2-2z,\tau) \frac{1}{z^{a+1}}\Phi\left(\frac{ |\xi_2|}{tz}\right)  \nonumber\\& 
 =\frac {8  \sqrt{2}} {2^a \, t}\delta(\xi_1)  \int_{\mathbb{R}} d\eta_2 \int_{\frac{|\xi_2|}{t}}^{+\infty} dz \, Q(\eta_2+z,\tau) \, Q(\eta_2-z,\tau) \frac{1}{z^{a+1}}\Phi\left(\frac{|\xi_2|}{tz}\right) \nonumber\\& 
 =\frac {8  \sqrt{2}} {2^a \, t}\delta(\xi_1)  \int_{\frac{ |\xi_2|}{t}}^{+\infty} dz \,\frac{1}{z^{a+1}}\Phi\left(\frac{ |\xi_2|}{tz}\right) \int_{\mathbb{R}} d\eta_2 \, Q(\eta_2+z,\tau) \, Q(\eta_2-z,\tau) \, .
  \end{align}
Notice that, due to the symmetry of the problem, we can restrict to the region where $\xi_2$ take positive values, i.e. $\xi_2>0$. 

Similarly, using the ansatz \eqref{eq:ansatzG2}, the loss term \eqref{eq:G_loss_bis} reduces to 
\begin{align}\label{eq:G_loss_bis_a}
  {K}^-[G,G](\xi)& 
  =8\pi G(\xi_1,\xi_2,\tau) \, \int_{\mathbb{R}}\int_{\mathbb{R}} d\eta_1 d\eta_2 \,|\xi_1-\eta_{1}|^{-a} G(\eta_{1},\eta_{2},\tau) \nonumber \\&
  = 8\pi G(\xi_1,\xi_2,\tau) \, \int_{\mathbb{R}}\int_{\mathbb{R}} d\eta_1 d\eta_2 \,|\xi_1-\eta_{1}|^{-a} \delta(\eta_1-\eta_2)Q(\eta_{2},\tau) \nonumber \\& 
  = 8\pi G(\xi_1,\xi_2,\tau) \, \int_{\mathbb{R}}d\eta_2 \,|\xi_1-\eta_{2}|^{-a} Q(\eta_{2},\tau) = \Lambda(\xi_1,\tau) G(\xi_1,\xi_2,\tau)
\end{align}
with $\Lambda(\xi_1,\tau)$ given as in \eqref{def:LAMBDA}, namely
\begin{equation*}
\Lambda(\xi_1,\tau):=8\pi  \, \int_{\mathbb{R}}d\eta_2 \,|\xi_1-\eta_{2}|^{-a} Q(\eta_{2},\tau)\,. 
\end{equation*}

\medskip 

\paragraph{Reformulation of the Boltzmann equation \eqref{eq:G_a} as a boundary value problem.}

We have seen in Section \ref{ssec:bdvalue} that we can reformulate the Boltzmann equation \eqref{eq:G} satisfied by $G$ as the following boundary value problem 
\begin{align}
 \label{eq:BV_G_a} 
   & \partial_{\tau}G-\frac{1}{a}\partial_{\xi_1}(\xi_1 G)-\left(\frac{1}{a}-1\right)\partial_{\xi_2}(\xi_2 G)
   +  \partial_{\xi_1}(\xi_2 G )=   - \Lambda(\xi_1,\tau)G(\xi_1,\xi_2,\tau),  \\&
   G(0^+\! \!, \xi_2, \tau)=
\frac{1}{\xi_2}\frac {8 \cdot \sqrt{2} }{2^a \, t} \! \!  \int_{\mathbb{R}} \! \! dx_1 \! \!\int_{\mathbb{R}}\! \! d\eta_2 \! \!\int_{\mathbb{R}} \! \!d\zeta_2  \! \!\int_{\frac{ |\xi_2|}{t}}^{+\infty} \! \! dz \, G\left( x_1+ z, \eta_2, \tau\right) G\left(x_1- z,\zeta_2,\tau \right) \frac{1}{z^{a+1}}\Phi\left(\frac{ |\xi_2|}{tz}\right), \ \   \xi_2>0 \,.\label{eq:bdG_a}
\end{align} 
Using the ansatz \eqref{eq:ansatzG2}, the boundary value $G(0^+\! \!, \xi_2, \tau)$, \eqref{eq:bdG_a} then becomes 
\begin{equation}\label{eq:Gbdval_a} 
G(0^+\! \!, \xi_2, \tau)= \frac{1}{\xi_2}\frac {8 \cdot \sqrt{2} }{2^a \, t}\int_{\frac{ |\xi_2|}{t}}^{+\infty} dz \,\frac{1}{z^{a+1}}\Phi\left(\frac{ |\xi_2|}{tz}\right) \int_{\mathbb{R}} d\eta_2 \,  {Q(\eta_2+z,\tau) \, Q(\eta_2-z,\tau) }
\end{equation}
Notice that in some of the formulas we will keep $t$ or $e^\tau$ indistinctly, in order to discharge the notation.
\smallskip

Moreover, the heuristic argument in Section \ref{sec:heuristic} suggests that we can write the boundary value \eqref{eq:bdG_a} in the self-similar form 
\begin{equation}
\label{def:Gbd_selfsim}
G(0^+,\xi_2,\tau)= \frac{e^{-2\tau}}{(\varepsilon(\tau))^2 } H\left(\frac {\xi_2}{\varepsilon(\tau)e^{\tau} }, \tau\right), \ \ \tau=\log(t)  
\end{equation}  
where we recall that $t=e^{\tau}$ and, denoting by $u=\frac {\xi_2}{\varepsilon(\tau) e^\tau}>0$, the function $H$ is given by 
\begin{align}
\label{def:Hbis}  
H(u,\tau)& := \frac {8\cdot \sqrt{2} }{2^a \, u} \! \!  \int_{\mathbb{R}} \! \! dx_1 \! \!\int_{\mathbb{R}}\! \! d\eta_2 \! \!\int_{\mathbb{R}} \! \!d\zeta_2  \! \!\int_{\varepsilon(\tau) |u|}^{+\infty} \! \! dz \, G\left( x_1+ z, \eta_2, \tau\right) G\left(x_1- z,\zeta_2,\tau \right) \frac{1}{z^{a+1}}\Phi\left(\frac{ \varepsilon(\tau)|u|}{z}\right) \nonumber \\& 
=\frac {8 \cdot \sqrt{2} }{2^a \, u} \int_{ \varepsilon(\tau)|u|}^{+\infty} dz \,\frac{1}{z^{a+1}}\Phi\left(\frac{\varepsilon(\tau)|u|}{ z}\right) \int_{\mathbb{R}} d\eta_2 \,  {Q(\eta_2+z,\tau) \, Q(\eta_2-z,\tau) }
\end{align} 
for $u>0$.

\paragraph{Solving \eqref{eq:BV_G_a}, \eqref{eq:bdG_a} with the method of characteristics.} 

We recall the system of characteristics introduced in Section \ref{ssec:char} associated to \eqref{eq:BV_G_a}, \eqref{eq:bdG_a} which reads as
\begin{align}\label{eq:systemchar_1}
    \frac{d \xi_1}{d\tau} & = \xi_2  -\frac{1}{a} \xi_1;  \\
    \frac{d \xi_2}{d\tau} & = -\left(\frac{1}{a}-1\right) \xi_2; \label{eq:systemchar_2} \\
    \frac{d G}{d\tau} & =  \left(\frac{2}{a}-1\right)G- K^-[G,G] = \left(\frac{2}{a}-1\right)G-\Lambda(\xi_1,\tau) G
\label{eq:systemchar_3}
\end{align}
with initial data, at time  $\tau_0$,  
$\xi_{1}(\tau_0)=0, \,\xi_{2}(\tau_0)=\xi_{2,0}$. The system above can be solved for $0 \leq \tau_0 < \tau$. We first take into account the first two equations in the system of characteristics, namely 
\eqref{eq:systemchar_1}, \eqref{eq:systemchar_2} and, by integrating, we obtain 
\begin{align}
    \xi_1(\tau) &=\xi_{2,0}  e^{-\left(\frac{1}{a}-1\right)(\tau-\tau_0)}- \xi_{2,0} e^{-\frac{1}{a}(\tau-\tau_0)}= \xi_{2,0} e^{-\frac{1}{a}(\tau-\tau_0)} \left(e^{(\tau-\tau_0)}-1 \right)\label{eq:xi1}; \\
    \xi_2(\tau) & = \xi_{2,0} e^{-\left(\frac{1}{a}-1\right)(\tau-\tau_0)}\label{eq:xi2}
\end{align}
which yields
\begin{equation} \label{eq:xi1xi2_1}
 \frac{\xi_1}{\xi_2}=1- e^{- (\tau-\tau_0)}  \text{ or, equivalently, } e^{(\tau-\tau_0)}= \frac{\xi_2}{\xi_2-\xi_1}\,
\end{equation}
where we are not writing explicitly the dependence of $\xi_1, \, \xi_2$ on $\tau$. 
We further observe that equations  \eqref{eq:xi1}, \eqref{eq:xi2} retain the same structure in both linear and nonlinear framework, and thus coincide with the corresponding equations derived in \cite{MNV25b}  
for the linear Boltzmann-Rayleigh equation. 

\medskip

Therefore, integrating by characteristics, we obtain the following expression for $G(\xi_1,\xi_2,\tau)$:
\begin{align}
    \label{eq:Gchar_a}
    G(\xi_1,\xi_2,\tau) &=   \frac{e^{-2\tau_0}} {(\varepsilon(\tau_0))^2} H\left(\frac { \xi_{2,0}}{\varepsilon(\tau_0) e^{\tau_0}}, \tau_0\right) \,\exp\left(\left(\frac{2}{a}-1\right)(\tau-\tau_0) \right) \exp\left(- \int_{\tau_0}^{\tau}\Lambda (\xi_1(s), s) ds\right) \nonumber \\&=
\frac{e^{-2\tau_0}}{(\varepsilon(\tau_0))^2} H\left( \frac{e^{-\tau}  }{\varepsilon(\tau_0)}\frac{\xi_2^{1+\frac 1 a} }{(\xi_2-\xi_1)^\frac 1 a}, \tau_0\right) \,  \left(\frac{\xi_2}{\xi_2-\xi_1}\right)^{\left(\frac{2}{a}-1\right)} 
\exp\left(
-  \int_{\tau_0}^{\tau}\Lambda (\xi_1(s), s) ds\right)\,
\end{align}
with $e^{(\tau-\tau_0)}= \frac{\xi_2}{\xi_2-\xi_1}$ and we recall that $\xi_{2,0}=\xi_2(\tau_0)$, $\tau >\tau_0 \geq 0$. 
\bigskip

We now observe that the time $\tau_0=\tau_0(\tau,\xi)$ introduced above is such that (cf. \eqref{eq:xi1xi2_1})
\begin{equation}\label{def:tau0in}
    \tau_0=\tau- \log \left( \frac{\xi_2}{\xi_2-\xi_1}\right) =\tau-\log \left( {\xi_2}\right)+\log\left(\xi_2-\xi_1\right)
\end{equation} 
and in the different regions can be approximated as follows: 
\begin{align}\label{eq:tau0coll}
    & \tau_0\sim (1-a)\tau \quad \text{when} \ \ |\xi|\approx  \varepsilon(\tau), \\& 
     \tau_0
     \sim (1-a) \tau + a\log(\xi_2), \quad \text{when} \ \ \varepsilon(\tau)\leq |\xi| \leq \varepsilon(\tau)e^\tau \label{eq:tau0regions} 
\end{align}
We will see later that \eqref{eq:tau0coll} holds true in the collision region, i.e. when $|\xi|\approx \varepsilon(\tau)$ and  
\eqref{eq:tau0regions}  in the region where most of the mass of the distribution function $G$ concentrates.  
\medskip

\noindent \textbf{Ansatz 2:} We now assume that the dependence of $H$ on $\tau>0$ is sufficiently slow and that we can approximate 
\begin{equation}\label{eq:ansatzH}
H(u, \tau) \sim \lambda(\tau)\overline{H}(u)   , \quad \text{with} \ \ u=\frac {\xi_2}{\varepsilon(\tau)e^\tau }
\end{equation}
where, in order to determine $\overline{H}$ and $\lambda(\tau)$ uniquely, we will impose that $\overline{H}$ 
satisfies 
\begin{equation}
    \label{eq:Hbar_norm}
\int_0^\infty \, u \, \overline{H}(u)  \, du =1 \, .\end{equation}

Using now \eqref{eq:ansatzH} in \eqref{eq:Gchar_a} 
\begin{align}
    \label{eq:Gchar_b}
    G(\xi_1,\xi_2,\tau) &=   
    \frac{e^{-2\tau_0}} {(\varepsilon(\tau_0))^2} \,\lambda\left( \tau_0\right) \, \overline{H}\left( \frac {e^{-\tau} }{\varepsilon(\tau_0)} \frac{\xi_2^{1+\frac 1 a} }{(\xi_2-\xi_1)^\frac 1 a}\right) \, \left(\frac{\xi_2}{\xi_2-\xi_1}\right)^{\left(\frac{2}{a}-1\right)} 
\exp\left(
-  \int_{\tau_0}^{\tau}\Lambda (\xi_1(s), s) ds\right)\nonumber\\&
=
\lambda\left( \tau_0\right) \, \frac{e^{a\tau} (\varepsilon(\tau_0))^a}{\xi_2^{2+ a}} \, \left(\frac{e^{-\tau}\xi_2^{(1+\frac 1 a)}}{ \varepsilon(\tau_0)(\xi_2-\xi_1)^{\frac 1 a}}\right)^{\left(2+a\right)} \overline{H}\left( \left( \frac {e^{-a\tau} }{(\varepsilon(\tau_0))^a}  \frac{\xi_2^{1+ a} }{(\xi_2-\xi_1)} \right)^{\frac 1 a} \right) \, 
\exp\left(-  \int_{\tau_0}^{\tau}\Lambda (\xi_1(s), s) ds\right)\nonumber\\&
=\lambda\left( \tau_0\right) \, \frac{e^{a\tau} (\varepsilon(\tau_0))^a}{\xi_2^{2+ a}} \, \left(\frac{e^{-a\tau}\xi_2^{1+a}}{ (\varepsilon(\tau_0))^a(\xi_2-\xi_1)}\right)^{\frac{\left(2+a\right)}{a}} \overline{H}\left( \left( \frac {e^{-a\tau} }{(\varepsilon(\tau_0))^a}  \frac{\xi_2^{1+ a} }{(\xi_2-\xi_1)} \right)^{\frac 1 a} \right) \, 
\exp\left(-  \int_{\tau_0}^{\tau}\Lambda (\xi_1(s), s) ds\right)
\end{align} 
where in the second identity we used that 
$e^{(\tau-\tau_0)}= \frac{\xi_2}{\xi_2-\xi_1}$ which yields $e^{-2\tau_0}= e^{-2\tau} \left(\frac{\xi_2}{\xi_2-\xi_1}\right)^2$ and the identity 
$\left(\frac{\xi_2}{\xi_2-\xi_1}\right)^2 \, \left(\frac{\xi_2}{\xi_2-\xi_1}\right)^{\left(\frac{2}{a}-1\right)} 
= \left(\frac{\xi_2^{\frac 1 a}}{(\xi_2-\xi_1)^{\frac 1 a}}\right)^{\left(2+a\right)} =\frac 1 {\xi_2^{2+ a}} \, \left(\frac{\xi_2^{(1+\frac 1 a)}}{(\xi_2-\xi_1)^{\frac 1 a}}\right)^{\left(2+a\right)}$. 

We now define 
\begin{equation}\label{def:functionU}
U(\zeta)
=\frac 1 {\zeta^{\frac {2+a} a}} \overline{H}\left(\frac 1 {\zeta^{\frac 1 a}}\right)\, \quad \forall \, \zeta \in \mathbb{R}
\end{equation}
such that $\int_{0}^{+\infty} U(\zeta)d\zeta=U_0< \infty$.  
Notice that using the normalization condition \eqref{eq:Hbar_norm}, we can compute explicitly the value of $U_0$ and prove that it is finite. More precisely we have 
\begin{equation}\label{def:intU0}
    U_0= \int_{0}^{+\infty}  U(\zeta)\, d\zeta = \int_{0}^{+\infty} \frac 1 {\zeta^{\frac {2+a} a}} \overline{H}\left(\frac 1 {\zeta^{\frac 1 a}}\right) \, d \zeta=
    a \int_{0}^{+\infty} y \overline{H}\left(y\right) dy =a 
< \infty
\end{equation}
where we changed variables, setting $y= \frac 1 {\zeta^{\frac {1} a}}$ with $ dy= - \frac 1 a \frac {d\zeta} {\zeta^{\frac {1}{a}+1}}$. Therefore \begin{equation}
\label{eq:valueUo}
    U_0=a\,.
\end{equation}
We observe that using \eqref{def:functionU} we can rewrite \eqref{eq:Gchar_b} as 
\begin{align}
    \label{eq:Gchar_b2}
    G(\xi_1,\xi_2,\tau) &=   
\lambda\left( \tau_0\right) \, \frac{e^{a\tau}(\varepsilon(\tau_0))^a }{\xi_2^{2+ a}} \,  U\left(e^{a\tau}(\varepsilon(\tau_0))^a \frac{(\xi_2-\xi_1)}{\xi_2^{1+ a} } \right) \, 
\exp\left(-  \int_{\tau_0}^{\tau}\Lambda (\xi_1(s), s) ds\right) \, .
\end{align} 
Moreover, we observe that we  expect to have 
\begin{equation}\label{eq:approxU}
  (\varepsilon(\tau_0))^a e^{a\tau}\, U((\varepsilon(\tau_0))^a e^{a\tau} z)  \approx U_0\delta(z) \quad \text{as} \ \ \tau \to \infty \, \ \ 
\end{equation}
in the sense of distributions.  
Relying on the approximation \eqref{eq:approxU} which is expected to be valid in the region where $\xi_2 \leq \varepsilon(\tau)e^\tau$ we arrive at 
\begin{align}
    \label{eq:Gchar_c}
    G(\xi_1,\xi_2,\tau) &=   
 \frac{\lambda\left( \tau_0\right)}{\xi_2^{2+ a}} \,  U_0 \, 
 \delta\left( \frac{\xi_2-\xi_1}{\xi_2^{1+ a} } \right) \, 
\exp\left(-  \int_{\tau_0}^{\tau}\Lambda (\xi_1(s), s) ds\right) \nonumber \\& = 
\frac{\lambda\left( \tau_0\right)}{\xi_2} \,  U_0 \,  \delta\left( \xi_2-\xi_1 \right) \, \exp\left(-  \int_{\tau_0}^{\tau}\Lambda (\xi_1(s), s) ds\right) \,.
\end{align} 
Comparing this formula with \eqref{eq:ansatzG}, namely the factorization of the solution $G(\xi_1,\xi_2,\tau)=Q(\xi_2,\tau) \delta(\xi_1-\xi_2)$ proposed in \textbf{Ansatz 1} we find
\begin{equation}\label{eq:formulaQ}
    Q(\xi_2,\tau) =\frac{\lambda\left( \tau_0\right)}{\xi_2} \,  U_0 \,   \exp\left(-  \int_{\tau_0}^{\tau}\Lambda (\xi_1(s), s) ds\right) , \quad 0<\xi_2\leq \varepsilon(\tau_0) e^{\tau}
\end{equation}  
where $U_0=\int_{0}^{+\infty} U(\zeta)d\zeta < \infty$. We observe that the cutoff in the region $\xi_2\approx \varepsilon(\tau_0) e^{\tau}$ is due to the fact the function $H$ decreases very quickly when $\xi_{2,0} \geq (\varepsilon(\tau_0))e^\tau$.   
\smallskip

The next step consists in computing the integral appearing in the exponential in \eqref{eq:formulaQ}. In order to do this we recall the formula for $\xi_1(s)$ given by \eqref{eq:systemchar_2} and use it in the integral on the r.h.s. of \eqref{eq:formulaQ}. More precisely, we have 
\begin{equation}\label{eq:intLambda}
    \int_{\tau_0}^{\tau}\Lambda (\xi_1(s), s) ds 
=8\pi  \, \int_{\tau_0}^{\tau} ds \int_{\mathbb{R}}d\eta_2 \,|\xi_1(s)-\eta_{2}|^{-a} Q(\eta_{2},s) 
\end{equation}

\paragraph{Analysis of the function $Q$ and derivation of the cutoff problem.}
We expect 
\begin{equation}\label{def:B}
Q(\eta_{2},\tau)\sim \frac {\lambda(\tau_0) } {|\eta_2|}\, U_0\, \Sigma(\eta_2,\tau)
\end{equation}
where $\Sigma(\eta_2,\tau)$ should give the cutoff for $|\eta_2|\approx \varepsilon(\tau)$.  

We now compare \eqref{eq:formulaQ} with \eqref{def:B} and we obtain 
\begin{equation}\label{eq:sigma2}
\Sigma(\xi_1,\tau) = \exp \left( -\int_{\tau_0}^{\tau}\Lambda (\xi_1(s), s) ds \right)
\end{equation}

Recalling the system of characteristics \eqref{eq:systemchar_1}-\eqref{eq:systemchar_3} we notice that, since $\xi_1(s)$ decreases exponentially in $s$ it is possible to approximate the characteristics with $\xi_1 \approx \xi_2$ and, furthermore, the cutoff function $\Sigma$ has relevant modification in the region where $\xi_1 \approx \varepsilon(\tau)$.  Moreover, we have 
\begin{equation}\label{eq:approx_xi1}
    \xi_1(s)\sim \xi_1 e^{\left(\frac 1 a -1\right)(\tau-s)}, \qquad \xi_1(\tau)=\xi_1
\end{equation}

Then, from \eqref{eq:intLambda}, using 
\eqref{def:B} we obtain
\begin{equation}
\label{eq:Lambda2} 
\Lambda(\xi_1(s),\tau)\sim 8\pi  \, U_0 \int_{\mathbb{R}}d\eta_2 \,|\xi_1(s)-\eta_{2}|^{-a}   \frac {\lambda(\tau_0)} {|\eta_2|}\Sigma(\eta_2,s)= 8\pi \lambda(\tau_0) \,  U_0 \int_{\mathbb{R}}d\eta_2 \, \frac 1 {|\xi_1(s)-\eta_{2}|^{a}|\eta_2|}\Sigma(\eta_2,s) \,  
\end{equation}
with $a<1$. Observe that, as explained in the heuristic argument in Section \ref{sec:heuristic} in this regime we have $\tau_0\sim (1-a)\tau$.  Therefore, $\tau-\tau_0\sim a \tau \gg 1 $ as $\tau\to \infty$. 

We now integrate \eqref{eq:Lambda2} w.r.t. the time variable $s$ and we obtain 
\begin{align}
\label{eq:Lambda3} 
\int_{\tau_0}^{\tau}\Lambda(\xi_1(s),\tau)ds  &= 8\pi  \, U_0 \lambda(\tau_0)\int_{\tau_0}^{\tau} ds \int_{\mathbb{R}}d\eta_2 \,|\xi_1(s)-\eta_{2}|^{-a}   \frac {1} {|\eta_2|}\Sigma(\eta_2,s)\nonumber\\& 
\sim
8\pi  \, U_0 \lambda(\tau_0)\int_{\tau_0}^{\tau} ds \int_{\mathbb{R}}d\eta_2 \,| \xi_1 e^{\left(\frac 1 a -1\right)(\tau-s)}-\eta_{2}|^{-a}   \frac {1} {|\eta_2|}\Sigma(\eta_2,s)   
\end{align}
where we used \eqref{eq:approx_xi1}. 
Due to the exponential decay in $s$ of $\xi_1(s)$ and the fact that the contribution to the region where $\tau-s \gg 1$ is exponentially small we can approximate the time integral $\int_{\tau_0}^{\tau} ds [\dots]\approx\int_{-\infty}^{\tau} ds [\dots] $ and hence we arrive at 
\begin{equation}
\label{eq:Lambda3bis} 
\int_{\tau_0}^{\tau}\Lambda(\xi_1(s),\tau)ds   \sim  
8\pi  \, U_0 \lambda(\tau_0)\int_{-\infty}^{\tau} ds \int_{\mathbb{R}}d\eta_2 \,|\xi_1 e^{\left(\frac 1 a -1\right)(\tau-s)}-\eta_{2}|^{-a}   \frac {1} {| \eta_2|}\Sigma(\eta_2,s)  
\end{equation}
Notice that since the function $\Sigma(\eta_2,s) $ is changing slowly in time, as we will check a posteriori, and the main contribution will come from the region where $s\approx \tau$ we can use that $\Sigma(\eta_2,s)$ is in a quasi stationary regime and thus $\Sigma(\eta_2,s)\sim\Sigma(\eta_2,\tau)$ and obtain 
\begin{align}
\label{eq:Lambda4} 
&\int_{\tau_0}^{\tau}\Lambda(\xi_1(s),\tau)ds   \sim
8\pi  \, U_0 \lambda(\tau_0)\int_{-\infty}^{\tau} ds \int_{\mathbb{R}}d\eta_2 \,|\xi_1 e^{\left(\frac 1 a -1\right)(\tau-s)}-\eta_{2}|^{-a}   \frac {1} {| \eta_2|}\Sigma(\eta_2,\tau) \nonumber \\&
=  8\pi  \, U_0 \lambda(\tau_0)\int_{\mathbb{R}} d\eta_2 \frac { \Sigma(\eta_2,\tau)} {| \eta_2|} \int_{-\infty}^{0} ds |\xi_1 e^{-\left(\frac 1 a -1\right)s}-\eta_{2}|^{-a}\,.
\end{align}
Plugging the approximation \eqref{eq:Lambda4} in \eqref{eq:sigma2} we then get
\begin{align}\label{eq:sigma3}
\Sigma(\xi_1,\tau) &= \exp \left( - 8\pi  \, U_0 \lambda(\tau_0) \int_{\mathbb{R}} d\eta_2 \frac { \Sigma(\eta_2,\tau)} {| \eta_2|} \int_{-\infty}^{0} ds |\xi_1 e^{-\left(\frac 1 a -1\right)s}-\eta_{2}|^{-a} \right)\nonumber\\&
=\exp \left( - \frac{8 \pi\, a  \, U_0 \lambda(\tau_0) }{\left(1 -a \right)}\int_{\mathbb{R}} d\eta_2 \frac { \Sigma(\eta_2,\tau)} {| \eta_2|} 
\int_{\xi_1}^{\infty} \frac{dz}{z} |z -\eta_{2}|^{-a}\right)
\end{align}
where in the second identity we used the change of variables $z= \xi_1 e^{-\left(\frac 1 a -1\right)s}$. 
We now compute the derivative of $\Sigma(\xi_1,\tau)$ w.r.t. the $\eta_1$ variable, namely
\begin{equation}\label{eq:derivativesigma3}
\frac{\partial\Sigma(\xi_1,\tau) }{\partial \xi_1}=  \frac{ 8 \pi  \, a \, U_0 \lambda(\tau_0)}{\xi_1 (1-a)} 
\left(\int_{\mathbb{R}} d\eta_2 \frac { \Sigma(\eta_2,\tau)} {| \eta_2||\xi_1 -\eta_{2}|^{a} } \right)\Sigma(\xi_1,\tau) = \frac{ a \, U_0 \lambda(\tau_0)}{\xi_1 (1-a)} 
\Lambda(\xi_1,\tau) \Sigma(\xi_1,\tau)  
\end{equation}
and multiplying by $\left(\frac 1 a -1 \right)\xi_1>0$ both sides we obtain
\begin{equation}\label{eq:derivativesigma4}
\left(\frac 1 a -1 \right)\xi_1 \, \frac{\partial\Sigma(\xi_1,\tau) }{\partial \xi_1} = U_0 \lambda(\tau_0)
\Lambda(\xi_1,\tau) \Sigma(\xi_1,\tau)  
\end{equation}
Observe that, due to the definition of $\Sigma$ in terms of an exponential (cf. \eqref{eq:sigma2} and \eqref{eq:sigma3}) we have that $\Sigma(\xi_1,\tau) \to 1$ as $\xi_1 \to \infty$.  

We will also discuss in Section \ref{ssec:altersteadycutoff} below an alternative derivation for the cutoff problem for the kinetic equation.

We now look for functions $\Sigma(\xi_1,\tau)$ with the form 
\begin{equation}\label{eq:chvarSigma}
    \Sigma(\xi_1,\tau)=\sigma(y_1, \tau)\quad \text{where} \quad y_1=\frac{\xi_1}{\varepsilon(\tau)}
    \end{equation}
where we set the size of the collision region $\varepsilon(\tau)$ by means of 
\begin{equation}\label{def:esplamtau0}
    \varepsilon(\tau) =\left( U_0\lambda(\tau_0)\right)^{\frac{1}{a}}=\left( a\lambda(\tau_0)\right)^{\frac{1}{a}}\, , 
\end{equation}
where in the equation above we used \eqref{eq:valueUo}. 
We will further assume that $\sigma(y_1,\tau)$ is stationary. Therefore, using \eqref{eq:derivativesigma4} combined with \eqref{eq:chvarSigma} it will then follow that $\sigma=\sigma(y_1)$ satisfies the following steady problem with matching condition 
\begin{equation} \label{eq:pde_b_stat}
  \left(\frac{1}{a}-1\right) y_1  \frac{\partial {\sigma}}{\partial y_1}= \Lambda(y_1){\sigma}, \qquad {\sigma}(y_1)\to 1 \quad \text{as}\quad  y_1\to\infty
 \end{equation}
 with 
 \begin{equation}
 \label{eq:statLambda}   \Lambda(y_1)= 8\pi  \, \int_{\mathbb{R}}d\zeta_2 \, \frac {\sigma(\zeta_2)} {|y_1-\zeta_{2}|^{a}|\zeta_2|}\, .
\end{equation} 
Notice that, by assumption, $\sigma$ is an even function, i.e. $\sigma(y_1)=\sigma(-y_1)$.

The nonlinear problem \eqref{eq:pde_b_stat}, \eqref{eq:statLambda} gives a description of how the cutoff mechanism for the distribution function $\Sigma$ (cf.\eqref{def:B}) takes place due to the fact the collisions remove the particles in the region where $|\xi|\approx \varepsilon(\tau)$. 
We will study in detail the solutions to the problem \eqref{eq:pde_b_stat}, \eqref{eq:statLambda} in Section 
\ref{ss:cutoff}. We now briefly discuss how the cutoff mechanism takes place due to the fact that ${\sigma}(y_1)\to 0$ as $y_1\to 0$. A rigorous proof of this will be given in Section \ref{ss:cutoff}.

 We check the feasibility of a power law behaviour and, specifically, consider $\sigma(y_1)\sim \kappa\, y_1^{m}$ as $y_1\to 0$, for some $m>0$ and $\kappa>0$ a numerical constant. Then, using the equation \eqref{eq:pde_b_stat} satisfied by $\sigma$  and using  that 
 $$\Lambda(y_1)= 8\pi  \, \int_{\mathbb{R}}d\zeta_2 \, \frac 1 {|y_1-\zeta_{2}|^{a}|\zeta_2|}{\sigma}(\zeta_2)\sim 8\pi  \, \int_{\mathbb{R}}d\zeta_2 \, \frac \kappa {|y_1-\zeta_{2}|^{a}}\frac{\zeta_2^m}{|\zeta_2|}\sim \frac K{|y_1|^{a-m}}\qquad \text{as} \quad |y_1|\to 0 $$  with $a> m$  and $K>0$ a numerical constant. Notice that the condition $a> m$ does not bring to the right asymptotics, since the solution to   
 \begin{equation*} 
  \left(\frac{1}{a}-1\right) y_1  \frac{\partial {\sigma}}{\partial y_1}= \exp\left(-\frac{K}{|y_1|^{a-m}} \right) {\sigma}(y_1)
 \end{equation*}
does not converge to $0$ as $y_1\to 0$.

We further observe that the problem \eqref{eq:pde_b_stat} is invariant under rescalings. Thus, we can adjust the stationary solution ${\sigma}={\sigma}(y_1)$ in order to have
\begin{equation}\label{eq:exp-asymsigma}
  {\sigma}(y_1)\to 1 \ \ \text{as} \ \ y_1 \to \infty, \qquad  {\sigma}(y_1) \sim \kappa y_1^{m},\ \ \text{as} \ \ y_1 \to 0^+.
\end{equation}
On the contrary, this equation is compatible with $m>a$. To exclude the case $m=a$ is a more delicate problem that requires a more careful studies of the resulting logarithmic terms, that we will pursue in this paper.  In Section \ref{ss:cutoff} we will prove that \eqref{eq:exp-asymsigma} takes place with $m\geq a$. 
 \bigskip

  We now justify the approximation of $\tau_0$ given as in \eqref{eq:tau0coll} in the collision region. 
  Let us denote as 
  $$y:= e^{a\tau}(\varepsilon(\tau_0))^a \frac{(\xi_2-\xi_1)}{\xi_2^{1+ a} }.$$ From \eqref{eq:Gchar_b2} we notice that the relevant contribution of $U$ is due to the region where $y\approx 1$
that implies that 
\begin{equation}\label{eq:xi1xi2approx}
     (\xi_2-\xi_1)={\xi_2^{1+ a} } \frac{e^{-a\tau}}{(\varepsilon(\tau_0))^a}
     \end{equation}
     Now we plug in \eqref{eq:xi1xi2approx}  in
 \eqref{def:tau0in} and we obtain 
 $$\tau_0=\tau-\log \left( {\xi_2}\right)+\log\left({\xi_2^{1+ a} } \frac{e^{-a\tau}}{(\varepsilon(\tau_0))^a}\right)\sim \tau-a\tau +\log(\xi_2)$$
where we neglect the contribution of the terms $\log(\varepsilon(\tau))$ being much smaller than the others. Notice that this approximation holds true for all the values of $\xi$ within the region $\varepsilon(\tau)\leq |\xi|\leq \varepsilon(\tau)e^\tau$ thus proving \eqref{eq:tau0regions}. We further notice that this approximation also implies \eqref{eq:tau0coll} due to the fact that in the collision region $|\xi|\approx \varepsilon(\tau)$ the logarithmic term $\log(\xi_2)$ is there negligible.

\bigskip

The analysis performed below allows to obtain information on the shape of the region
yielding the collisions, and then the cutoff for
$Q(\cdot,\tau)$  as $|\xi_2|\approx \varepsilon(\tau)= (U_0\lambda(\tau_0))^\frac 1 a$. We remark that when $|\xi_2|$ becomes of order $\varepsilon(\tau)$ we have that $\tau_0\sim (1-a)\tau$ as $\tau\to \infty$, as given in \eqref{eq:tau0coll}.  
Our aim is now to close the delay equation and obtain an approximation for the auxiliary function $\overline{H}(z)$. 
In order to do this, we come back to the analysis of the boundary value \eqref{eq:Gbdval_a}. Combining  \eqref{eq:Gbdval_a} with \eqref{def:Gbd_selfsim} and the fact that $
H(\zeta, \tau) \sim \lambda(\tau)\overline{H}(\zeta)$, with $\zeta=\frac {\xi_2}{e^\tau \varepsilon(\tau)}$ (cf.~\eqref{eq:ansatzH}), we then obtain 
\begin{align}
  \lambda(\tau) \overline{H}(\zeta )
  &
  =\frac {8 \sqrt{2} \, \varepsilon(\tau)}{2^a \, \zeta} \int_{\varepsilon(\tau)|\zeta|}^{+\infty} dz \,\frac{1}{z^{a+1}}\Phi\left(\frac{ \varepsilon(\tau)|\zeta|}{z}\right) \int_{\mathbb{R}} d\eta_2 \,  {Q(\eta_2+z,\tau) \, Q(\eta_2-z,\tau) } \,.
  \label{eq:lamHbar}
\end{align} 
Using now the expression for $Q$ given in \eqref{def:B}, namely 
\[
Q(\eta_{2},\tau)\sim  \frac {\lambda(\tau_0)} {|\eta_2|}U_0\Sigma(\eta_2,\tau)= \frac {\lambda(\tau_0)} {|\eta_2|}U_0\sigma\left(\frac{\eta_2}{\varepsilon(\tau)},\tau\right),
\]
in the equation \eqref{eq:lamHbar} we arrive at
 \begin{align}
  \lambda(\tau) \overline{H}(\zeta )
  &
  \sim 
\frac {8 \sqrt{2}\,\varepsilon(\tau) }{2^a \, \zeta } \int_{\varepsilon(\tau) |\zeta|}^{+\infty} dz \,\frac{1}{z^{a+1}}\Phi\left(\frac{\varepsilon(\tau)|\zeta|}{z}\right) \int_{-\infty}^\infty  d\eta_2 \,  {Q(\eta_2+z,\tau) \, Q(\eta_2-z,\tau) }\nonumber \\&
= \frac {8  \sqrt{2}\,\varepsilon(\tau) }{2^a \, \zeta } \left( \lambda(\tau_0)\right)^2 U_0^2 \int_{ \varepsilon(\tau)|\zeta|}^{+\infty} dz \,\frac{1}{z^{a+1}}\Phi\left(\frac{\varepsilon(\tau)|\zeta|}{z}\right) \int_{-\infty}^\infty  d\eta_2 \, { \frac{\Sigma(\eta_2+z,\tau) }{|\eta_2+ z|}\, \frac{\Sigma(\eta_2-z,\tau) }{|\eta_2 - z|} }\nonumber \\&
\sim \frac {8 \sqrt{2}\,\varepsilon(\tau) }{2^a \, \zeta } \left( \lambda(\tau_0)\right)^2 U_0^2 \int_{\varepsilon(\tau)|\zeta|}^{+\infty} dz \,\frac{1}{z^{a+1}}\Phi\left(\frac{\varepsilon(\tau)|\zeta|}{z}\right) \int_{-\infty}^\infty  d\eta_2 \, { \frac{\sigma\left(\frac{\eta_2+z}{\varepsilon(\tau)}\right)}{|\eta_2 + z|}\,  \frac{\sigma\left(\frac{\eta_2-z}{\varepsilon(\tau)} \right)}{|\eta_2-z|} }
\label{eq:lamHbar2} 
\end{align}
where in the third approximate identity we use the stationarity  of the cutoff function $\sigma$.  
We further rescale $ \eta_2$ and $z$ as follows:
\[
  \eta_2 = {\widetilde{ \eta}}\,{\varepsilon(\tau)},
  \qquad
  z = \widetilde{z}\,\varepsilon( \tau).
\]
It follows that \eqref{eq:lamHbar2} becomes 
\begin{align}
  \lambda(\tau) \overline{H}(\zeta )
  &
  \sim 
\frac {8 \sqrt{2} }{2^a \, \zeta } \frac{\left( \lambda(\tau_0)\right)^2 U_0^2  }{\left(\varepsilon(\tau)\right)^{a}} \int_{|\zeta|}^{+\infty} d\widetilde{z} \,\frac{1}{\left(\widetilde{z}\right)^{a+1}} \,\Phi\left(\frac{|\zeta|}{\widetilde{z}}\right) \int_{-\infty}^\infty  d\widetilde{\eta} \,  { \frac{{\sigma}\left({ \widetilde{\eta} +\widetilde{z}} \right)}{|\widetilde{\eta} + \widetilde{z}|}\,  \frac{{\sigma}\left({\widetilde{\eta}-\widetilde{z}} \right)}{|\widetilde{\eta}-\widetilde{z}|} }
\label{eq:lamHbar3} 
\end{align}

We recall the normalization condition \eqref{eq:Hbar_norm}, namely $\int_{0}^{\infty} \zeta \, \overline{H}(\zeta) d\zeta=1$. To recover this normalization we multiply the formula \eqref{eq:lamHbar3} by $\zeta$ and integrate w.r.t. $\zeta$ between $0$ and $+\infty$. We then obtain
\begin{align}
  \lambda(\tau)\int_{0}^{\infty} \zeta \,  \overline{H}(\zeta ) d\zeta
  &
  \sim
\frac {8 \sqrt{2} }{2^a} \frac{\left( \lambda(\tau_0)\right)^2 U_0^2  }{\left(\varepsilon(\tau)\right)^{a}} \int_{0}^{\infty} d\zeta \int_{ \zeta}^{+\infty} d\widetilde{z} \,\frac{1}{\left(\widetilde{z}\right)^{a+1}} \,\Phi\left(\frac{\zeta}{\widetilde{z}}\right) \int_{-\infty}^\infty  d\widetilde{\eta} \,  { \frac{{\sigma}\left({ \widetilde{\eta} +\widetilde{z}} \right)}{|\widetilde{\eta} + \widetilde{z}|}\,  \frac{{\sigma}\left({\widetilde{\eta}-\widetilde{z}} \right)}{|\widetilde{\eta}-\widetilde{z}|} } 
\label{eq:lamHbar4} 
\end{align}
Using Fubini to exchange the order of integration in \eqref{eq:lamHbar4} we arrive at 
\begin{align}
  \lambda(\tau)\int_{0}^{\infty} \zeta \,  \overline{H}(\zeta ) d\zeta
  &
  \sim  
\frac {8 \sqrt{2} }{2^a}  \frac{\left( \lambda(\tau_0)\right)^2 U_0^2  }{\left(\varepsilon(\tau)\right)^{a}} \int_{0}^{\infty} d\widetilde{z}\,\frac{1}{\left(\widetilde{z}\right)^{a+1}}   \int_{0}^{\frac{\widetilde{z}}{4}} d\zeta \,\Phi\left(\frac{4\zeta}{\widetilde{z}}\right) \int_{-\infty}^\infty  d\widetilde{\eta} \,  { \frac{{\sigma}\left({ \widetilde{\eta} +\widetilde{z}} \right)}{|\widetilde{\eta} + \widetilde{z}|}\,  \frac{{\sigma}\left({\widetilde{\eta}-\widetilde{z}} \right)}{|\widetilde{\eta}-\widetilde{z}|} }\nonumber\\& =:
\frac {8 \sqrt{2} }{2^a}  \frac{\left( \lambda(\tau_0)\right)^2 U_0^2  }{\left(\varepsilon(\tau)\right)^{a}} \mathcal{I}
\label{eq:lamHbar5}
\end{align}
with $$\mathcal{I}:=\int_{0}^{\infty} d\widetilde{z}\,\frac{1}{\left(\widetilde{z}\right)^{a+1}}   \int_{0}^{\widetilde{z}} d\zeta \,\Phi\left(\frac{\zeta}{\widetilde{z}}\right) \int_{-\infty}^\infty  d\widetilde{\eta} \,  { \frac{{\sigma}\left({ \widetilde{\eta} +\widetilde{z}} \right)}{|\widetilde{\eta} + \widetilde{z}|}\,  \frac{{\sigma}\left({\widetilde{\eta}-\widetilde{z}} \right)}{|\widetilde{\eta}-\widetilde{z}|} }\, .$$
We notice that using the change of variables  $\zeta=\widetilde{z} s$ with $d\zeta= \widetilde{z}  ds$  we have $$\int_{0}^{\widetilde{z}} d\zeta \,\Phi\left(\frac{\zeta}{\widetilde{z}}\right) = \widetilde{z} \int_{0}^{1} ds \,\Phi\left(s\right) = \sqrt{2} \pi \widetilde{z} $$  where in the second identity we used \eqref{eq:intPhi} (see Appendix \ref{ssec:massG}). 
Plugging  in this identity into the integrals $\mathcal{I}$ in equation \eqref{eq:lamHbar5} we get
\begin{equation}
    \mathcal{I}= \sqrt{2}\pi  \int_{0}^{\infty} d\widetilde{z}\,\frac{1}{\widetilde{z}^{a}}    \int_{-\infty}^\infty  d\widetilde{\eta} \,  { \frac{{\sigma}\left({ \widetilde{\eta} +\widetilde{z}} \right)}{|\widetilde{\eta} + \widetilde{z}|}\,  \frac{{\sigma}\left({\widetilde{\eta}-\widetilde{z}} \right)}{|\widetilde{\eta}-\widetilde{z}|} } = \frac{\sqrt{2}\pi}{2} \int_{-\infty}^{\infty} d\widetilde{z}\,\frac{1}{|\widetilde{z}|^{a}}    \int_{-\infty}^\infty  d\widetilde{\eta} \,  { \frac{{\sigma}\left({ \widetilde{\eta} +\widetilde{z}} \right)}{|\widetilde{\eta} + \widetilde{z}|}\,  \frac{{\sigma}\left({\widetilde{\eta}-\widetilde{z}} \right)}{|\widetilde{\eta}-\widetilde{z}|} } 
\end{equation}
where in the second identity we use the symmetrization of the integral in the $\widetilde{z}$. Changing variables again and setting $\alpha=\widetilde{\eta} +\widetilde{z}$ and $\beta= \widetilde{\eta} -\widetilde{z}$ with Jacobian $d\alpha d\beta= 2 \, d\widetilde{\eta} d\widetilde{z}$ it follows that 
\begin{equation}
    \mathcal{I}= \frac{\sqrt{2}\pi}{4} \int_{-\infty}^{\infty} d\alpha \int_{-\infty}^{\infty} d\beta \, \frac{2^a}{|{\alpha-\beta}|^{a}}     \,  { \frac{{\sigma}\left(\alpha \right)}{|\alpha|}\,  \frac{{\sigma}\left( \beta \right)}{|\beta|} } 
\end{equation}
and hence, from \eqref{eq:lamHbar5}, we have 
\begin{align}
  \lambda(\tau)\int_{0}^{\infty} \zeta \,  \overline{H}(\zeta ) d\zeta
  &
  \sim 
\frac {8 \sqrt{2} }{2^a}  \frac{2^a \left( \lambda(\tau_0)\right)^2 U_0^2  }{\left(\varepsilon(\tau)\right)^{a}}   \frac{\sqrt{2}\pi}{4} \int_{-\infty}^{\infty} d\alpha \int_{-\infty}^{\infty} d\beta \, \frac{1}{|{\alpha-\beta}|^{a}}     \,  { \frac{{\sigma}\left(\alpha \right)}{|\alpha|}\,  \frac{{\sigma}\left( \beta \right)}{|\beta|} }  
\nonumber
\\& = 4 \pi \frac{\left( \lambda(\tau_0)\right)^2 U_0^2  }{\left(\varepsilon(\tau)\right)^{a}}  \int_{-\infty}^{\infty} d\alpha \int_{-\infty}^{\infty} d\beta \, \frac{1}{|{\alpha-\beta}|^{a}}     \,  { \frac{{\sigma}\left(\alpha \right)}{|\alpha|}\,  \frac{{\sigma}\left( \beta \right)}{|\beta|} } 
\label{eq:lamHbar6}
\end{align}
Using the normalization condition \eqref{eq:Hbar_norm} we then obtain
\begin{align}
  \lambda(\tau) \sim 
4 \pi \frac{\left( \lambda(\tau_0)\right)^2 U_0^2  }{\left(\varepsilon(\tau)\right)^{a}}  \int_{-\infty}^{\infty} d\alpha \int_{-\infty}^{\infty} d\beta \, \frac{1}{|{\alpha-\beta}|^{a}}     \,  { \frac{{\sigma}\left(\alpha \right)}{|\alpha|}\,  \frac{{\sigma}\left( \beta \right)}{|\beta|} } 
\label{eq:lamHbar7}
\end{align}
As it will be proved later in Section \ref{ss:cutoff3} (see Proposition \ref{prop:intLam}), it is possible to compute explicitly the integral on the r.h.s. of \eqref{eq:lamHbar7} and it turns out that 
\begin{equation}
    4\pi \int_{-\infty}^{\infty} d\alpha \int_{-\infty}^{\infty} d\beta \, \frac{1}{|{\alpha-\beta}|^{a}}     \,  { \frac{{\sigma}\left(\alpha \right)}{|\alpha|}\,  \frac{{\sigma}\left( \beta \right)}{|\beta|} } = \frac{(1-a)}{a}
\label{eq:lamHbar8}
\end{equation}
Therefore, \eqref{eq:lamHbar7} implies
\begin{equation}\label{eq:lamHbar9}
   \lambda(\tau)   \sim  \frac{\left( \lambda(\tau_0)\right)^2 U_0^2  }{\left(\varepsilon(\tau)\right)^{a}}  \frac{(1-a)}{a}
\end{equation}
Recalling now the formula for characteristic size $\varepsilon(\tau)$ of the region where the collisions take place (cf. \eqref{def:esplamtau0}) from  \eqref{eq:lamHbar8} we finally obtain 
\begin{equation}\label{eq:lamHbar10}
   \lambda(\tau)   \sim \frac{\left( \lambda(\tau_0)\right)^2 U_0^2  }{\left(U_0 \lambda(\tau_0)\right)}  \frac{(1-a)}{a}
   =  (1-a)  \lambda(\tau_0) \, \frac{U_0}{a} 
\end{equation}
Using \eqref{def:intU0} it follows that 
\begin{equation}\label{eq:lamHbar11}
   \lambda(\tau)   \sim  
   (1-a)\, \lambda(\tau_0),  
\end{equation}
for $|\xi| \lesssim \varepsilon(\tau)$, i.e. in the collision region for which $\tau_0 \sim (1-a)\tau$. 

\smallskip

We now compute the asymptotic behaviour of $\lambda(\tau)$ as $\tau \to \infty$ by using the mass conservation property for $G(\xi_1, \xi_2, \tau)$.  We recall the approximation for 
$G(\xi_1, \xi_2, \tau)$ given as in \eqref{eq:Gchar_b2} as well the definition of the cutoff function $\Sigma$ given as in \eqref{eq:sigma2}. We have 
\begin{equation}\label{eq:GapprU}
   G(\xi, \tau) = \lambda(\tau_0(\tau, z)) (\varepsilon(\tau_0))^a \frac{e^{a\tau} }{\xi_2^{2+a}} U \left( \frac{(\varepsilon(\tau_0))^a\, e^{a\tau}\, (\xi_2 - \xi_1)}{ \xi_2^{1+a}} \right) \Sigma (\xi_2,\tau)\,.
   \end{equation}
We  integrate $G$ with respect to $\xi_1, \xi_2$ and using the symmetry under the reflection $\xi \leftrightarrow -\xi$ we have that $$\int_{\mathbb{R}^2} d\xi  \, G(\xi, \tau) = 2 \int_{0}^\infty d\xi_2 \int_{0}^{\infty} d\xi_1 \, G(\xi_1, \xi_2, \tau)\approx 2 \int_{0}^\infty d\xi_2 \int_{0}^{\xi_2} d\xi_1 G(\xi_1, \xi_2, \tau)$$ and hence 
\begin{align}\label{eq:apprGdelta1}
 2 \int_{0}^\infty d\xi_2 \int_{0}^{\xi_2} d\xi_1  \lambda(\tau - \log(\xi_2) + \log(\xi_2 - \xi_1)) \frac{e^{a\tau}}{\xi_2^{2+a}} U \left( \frac{(\varepsilon(\tau_0))^a(\xi_2 - \xi_1)}{e^{-a\tau} \xi_2^{1+a}} \right) \Sigma (\xi_2, \tau) d\xi_1\,d\xi_2 \end{align}
Let $y:=\frac{(\varepsilon(\tau_0))^a(\xi_2 - \xi_1)}{e^{-a\tau} \xi_2^{1+a}}$ with $dy= - \frac{(\varepsilon(\tau_0))^a }{e^{-a\tau} \xi_2^{1+a}} d\xi_1$. 
Notice that in \eqref{eq:GapprU} $U$ gives a relevant contribution for $|y| =O(1)$.  
Using
\begin{align}\label{eq:apprGdelta2}
 & 2 \int_{0}^\infty  \int_{0}^{\infty} d\xi_2 \frac{dy}{\xi_2} \lambda(\tau - \log(\xi_2) +\log( \varepsilon(\tau_0) ^{-a}\, e^{-a\tau} \xi_2^{1+a} y) )\, U(y) \Sigma(\xi_2,\tau) 
 \nonumber \\&
 \sim 2 \int_{0}^\infty  \int_{0}^{\infty} d\xi_2 \frac{dy}{\xi_2}   
\lambda(\tau - \log(\xi_2) - a\tau + (1+a)\log(\xi_2) + \log(y))\, U(y) \Sigma(\xi_2,\tau)  \nonumber 
\\& \sim 2 \int_{0}^\infty  \int_{0}^{\infty} d\xi_2  \frac{dy}{\xi_2}   
\lambda((1-a) \tau + a\log(\xi_2))\, U(y) \Sigma(\xi_2,\tau) \quad \text{as} \quad d\tau \to \infty\,
\end{align} 
where in the first approximate identity we used that $\xi_2 + e^{-a\tau} \xi_2^{1+a} y\sim  \xi_2$ and in the second one that $\log(y)$ is negligible for $|y|\lesssim 1$ and also neglected the $\log(\varepsilon(\tau_0))$ since we expect $\varepsilon(\tau_0)$ to behave as a power law of $\tau_0$.  We then have that the r.h.s. of \eqref{eq:apprGdelta2} gives
\begin{align}
& 2 \int_{0}^{+\infty} d\xi_2 \frac {\lambda((1-a) \tau + a\log(\xi_2)) }{\xi_2} \,\Sigma(\xi_2,\tau)\left(\int_{-\infty}^{+\infty} {dy}   
 U(y) \right) \nonumber
 \\& 
 = 2\, U_0 \int_{0}^{+\infty} d\xi_2 \, \frac {\lambda((1-a) \tau + a\log(\xi_2)) }{\xi_2} \Sigma(\xi_2,\tau) 
 = 2 a \int_{0}^{+\infty} d\xi_2 \, \frac {\lambda((1-a) \tau + a\log(\xi_2)) }{\xi_2} \Sigma(\xi_2,\tau) 
\end{align}
where we used that 
\begin{equation} \label{def:valueU0}
  U_0:= \int_{-\infty}^{+\infty} {dy}  
 U(y)=a.
\end{equation}
 The mass conservation $\int_{\mathbb{R}^2} d\xi G(\xi,\tau)=1$ then implies  
 \begin{equation}\label{eq:asymGmass5}
     2 a \int_{0}^{+\infty} d\xi_2 \, \frac {\lambda((1-a) \tau + a\log(\xi_2)) }{\xi_2} \Sigma(\xi_2,\tau) \sim 1\,  \quad \text{as}\quad \tau \to \infty.
 \end{equation}
 We now use the fact that the function $\Sigma$ can be thought as a cutoff function that does not vanish for values $\varepsilon(\tau) \leq \xi_2 \leq e^\tau \varepsilon(\tau)$ and the approximation $\Sigma\sim 1$ holds true in this interval. Therefore, using the ansatz \eqref{eq:ansatzlambda}, namely  that $
\lambda(\tau)\sim \frac{C_0}{\tau}\ \text{as} \ \tau\to \infty $
where $C_0>0$ is a numerical constant, one can reduce \eqref{eq:asymGmass5} to \eqref{eq:asymGmass3} which then allows to obtain \eqref{def:Co}, namely the value of the numerical constant $
 C_0 =\frac 1 {2a\,\log \left(\frac{1}{1-a}\right)}$, 
as explained in Section \ref{sec:heuristic}. 

 \bigskip

 Hence, we obtain the following approximation:
\begin{equation}\label{eq:finalGappr_sec5}
G(\xi, \tau) = 
\frac 1 { 2\log \left(\frac{1}{1-a}\right)}\, \frac{1} 
{\, ((1-a)\tau + a \log(\vert\xi_1\vert )) \, \vert\xi_1\vert } \chi_{\{\vert\xi_1\vert  \le e^\tau \varepsilon(\tau) \}} {\sigma} \left( \frac{\xi_1}{\varepsilon(\tau)} \right) \delta(\xi_2 - \xi_1) 
 \quad \text{as} \ \ \tau \to \infty \end{equation}
which concludes the derivation of the formula providing the asymptotics of the reduced distribution function $G(\xi,\tau)$.

It is interesting to notice that we defined the size of the collision region $\varepsilon(\tau)$ by means of the loss term, specifically using the scale in which the loss term becomes relevant, as can be checked from \eqref{eq:derivativesigma4}-\eqref{def:esplamtau0}.   On the other hand, when we use this value of $\varepsilon(\tau)$  in the gain term (cf. \eqref{eq:lamHbar3}) we obtain the approximate identity \eqref{eq:lamHbar11} with $ \tau_0=(1-a) \tau$ for $|\xi|\approx \varepsilon(\tau)$, i.e.  
\begin{equation}\label{eq:approxLamdbatau_sec5}
 \lambda(\tau)  \sim   
   (1-a)\, \lambda((1-a)\tau)
\end{equation}
   that holds true since 
$\lambda(\tau)\sim \frac{C_0}{\tau}$, as $\tau\to \infty$ (cf.~\eqref{eq:ansatzlambda}) that was derived from the mass conservation. The fact that \eqref{eq:lamHbar11} has been obtained by using the value $\varepsilon(\tau)$ given in \eqref{def:esplamtau0} provides a strong argument in favour of the consistency of the asymptotics obtained in this paper. 

\smallskip

After deriving all these asymptotics for the functions that allow to characterize the typical scales of the system, i.e. $\lambda(\tau)$ and $\varepsilon(\tau)$, we can finally obtain a closed form for the auxiliary function $\overline{H}(\zeta)$ whose existence was assumed in \eqref{eq:ansatzH}. Using \eqref{eq:lamHbar3}, namely 
$$
  \lambda(\tau) \overline{H}(\zeta )
  \sim  
\frac {8 \sqrt{2} }{2^a \, \zeta } \frac{\left( \lambda(\tau_0)\right)^2 U_0^2  }{\left(\varepsilon(\tau)\right)^{a}} \int_{  |\zeta|}^{+\infty} d\widetilde{z} \,\frac{1}{\left(\widetilde{z}\right)^{a+1}} \,\Phi\left(\frac{ |\zeta|}{\widetilde{z}}\right) \int_{-\infty}^\infty  d\widetilde{\eta} \,  { \frac{{\sigma}\left({ \widetilde{\eta} +\widetilde{z}} \right)}{|\widetilde{\eta} + \widetilde{z}|}\,  \frac{{\sigma}\left({\widetilde{\eta}-\widetilde{z}} \right)}{|\widetilde{\eta}-\widetilde{z}|} }
$$ 
and combining it with \eqref{eq:lamHbar9}, namely $
\lambda(\tau)   \sim \frac{\left( \lambda(\tau_0)\right)^2 U_0^2  }{\left(\varepsilon(\tau)\right)^{a}}  \frac{(1-a)}{a}$ 
   we then obtain  \begin{equation}
      \label{def:barH_sec5}
  \overline{H}(\zeta )
=
\frac {8 \sqrt{2} }{2^a \, \zeta }  \frac{a}{(1-a)} \int_{ |\zeta|}^{+\infty} d\widetilde{z} \,\frac{1}{\left(\widetilde{z}\right)^{a+1}} \,\Phi\left(\frac{|\zeta|}{\widetilde{z}}\right) \int_{-\infty}^\infty  d\widetilde{\eta} \,  { \frac{{\sigma}\left({ \widetilde{\eta} +\widetilde{z}} \right)}{|\widetilde{\eta} + \widetilde{z}|}\,  \frac{{\sigma}\left({\widetilde{\eta}-\widetilde{z}} \right)}{|\widetilde{\eta}-\widetilde{z}|} }\,. 
\end{equation} 
We observe that the equation above provides a proper definition of the function $\overline{H}(\zeta )$. Moreovoer, we 
remark that, by construction, we obtain that $\int_0^\infty \, \zeta \, \overline{H}(\zeta)  \, du =1$ as expected (cf. \eqref{eq:Hbar_norm}). 

\bigskip

\subsection{Alternative derivation of the cutoff problem for the kinetic equation}\label{ssec:altersteadycutoff}

It is worth mentioning an alternative derivation of this cutoff problem \eqref{eq:pde_b_stat}, \eqref{eq:statLambda}, using as starting point the kinetic equation. We examine how the term $\Sigma$, given as in \eqref{def:B}, can be deduced. Looking at the system of characteristics
\eqref{eq:systemchar_1}-
\eqref{eq:systemchar_3} it appears that the only term that reduces $\Sigma$ is the term $-\Lambda G$ in \eqref{eq:systemchar_3}. From  \eqref{eq:systemchar_1}- \eqref{eq:systemchar_3}, using \eqref{eq:ansatzG} and \eqref{def:B} we then formally obtain 
\begin{align}\label{eq:systemcharB}
   & \frac{d \xi_2}{d\tau}  = -\left(\frac{1}{a}-1\right) \xi_2;  \\&
    \frac{d \Sigma}{d\tau} (\xi_2,\tau) = -\Lambda(\xi_2,\tau) \Sigma(\xi_2,\tau)
\end{align}
where we used that, thanks to the Dirac, $\xi_1=\xi_2$. 

We can change variables in order to describe the ``damping" of $\Sigma$ in the region where $\xi_2\approx \varepsilon(\tau)$, $\varepsilon(\tau)$ changing very slowly, and we set 
$\xi_2 = \varepsilon(\tau) y_2$ with $\varepsilon(\tau)=\left( a\lambda(\tau) \right)^ \frac 1 a$, $\lambda$ slowly varying (cf. \eqref{def:esplamtau0}). 
It then follows that we can rewrite $\Lambda(\xi_1,\tau)$ given as in \eqref{eq:Lambda2} as 
\begin{equation}
\label{eq:Lambda4} 
\Lambda(\xi_1,\tau)\sim 
 8\pi a \, \lambda(\tau) \, \int_{\mathbb{R}}d\eta_2 \, \frac 1 {|\xi_2-\eta_{2}|^{a}|\eta_2|}\Sigma(\eta_2,\tau) 
 = 8\pi  \, \int_{\mathbb{R}}d\zeta_2 \, \frac 1 {|y_2-\zeta_{2}|^{a}|\zeta_2|}\sigma(\zeta_2,\tau) \,. 
\end{equation}
where in the first step we used $\xi_1=\xi_2$ and in the second one the change of variables $\eta_2=\varepsilon(\tau) \zeta_2$ as well as the fact that 
$\Sigma({\eta_2},\tau)=\sigma(\zeta_2, \tau)$ where $\zeta_2=\frac{\eta_2}{\varepsilon(\tau)}$.  

Looking at the system of characteristics \eqref{eq:systemcharB} and using the fact that $\Sigma\left(\xi_2,\tau\right)=\sigma\left(\frac{\xi_2}{\varepsilon(\tau)}, \tau\right)$ 
$ \xi_2= \varepsilon(\tau)y_2$, we obtain that 
$$\frac{d \Sigma}{d\tau} (\xi_2,\tau) = \frac{d }{d\tau} b\left(\frac{y_2}{\varepsilon(\tau)} , \tau\right)=\frac{d b}{d\tau}+ \underbrace{\frac{\dot{\varepsilon}(\tau)}{\varepsilon(\tau)} y_2\, \frac{\partial \sigma}{\partial y_2}}_{=:I_1} , \qquad 
\frac{d \xi_2}{d\tau}  = \dot{\varepsilon}(\tau) y_2
+ \varepsilon(\tau) \frac{dy_2}{d\tau}= \varepsilon(\tau) \left(\underbrace{\frac{ \dot{\varepsilon}(\tau)}{\varepsilon(\tau)}y_2}_{=:I_2}  + \frac{dy_2}{d\tau}\right)  $$
where we used the notation $\dot{\varepsilon}(\tau)$ to denote the time derivative. We now observe that the terms  $I_1, \, I_2$ are negligible. More precisely, $I_1, \, I_2 \to 0$ due to the fact that $\frac{\dot{\varepsilon}(\tau)}{\varepsilon(\tau)}\to 0$ as $\tau\to \infty$. 
Therefore, we can rewrite the system of characteristics \eqref{eq:systemcharB}, which expresses the damping of $\Sigma$ due to collisions, in terms of $\sigma\left({y_2}, \tau\right)$. It reads 
\begin{align}\label{eq:systemcharb}
   & \frac{d y_2}{d\tau}  = -\left(\frac{1}{a}-1\right) y_2;  \\&
    \frac{d \sigma}{d\tau} (y_2,\tau) = -\Lambda(y_2,\tau) \sigma(y_2,\tau)
\end{align}
where $\Lambda=\Lambda(y_2,\tau)$ is given as in \eqref{eq:Lambda2}, i.e., $
\Lambda(y_2,\tau)
 = 8\pi  \, \int_{\mathbb{R}}d\zeta_2 \, \frac 1 {|y_2-\zeta_{2}|^{a}|\zeta_2|}\sigma(\zeta_2,\tau).$

 We can then look at the PDE corresponding to \eqref{eq:systemcharb}, namely
 \begin{equation} \label{eq:pde_b}
  \frac{\partial \sigma}{\partial \tau}-\left(\frac{1}{a}-1\right) y_2  \frac{\partial \sigma}{\partial y_2}= -\Lambda(y_2,\tau) \sigma,  \quad \sigma=\sigma(y_2,\tau)
 \end{equation}
where $\Lambda$ is given as in \eqref{eq:Lambda2}.  The steady states of the equation \eqref{eq:pde_b} satisfy the cutoff problem \eqref{eq:pde_b_stat}, \eqref{eq:statLambda}. 
\bigskip



\section{Asymptotics for the full distribution $F$} \label{sec:asymF}
We now recall that the full distribution $F=F(\xi,\tau)$, $\xi\in\mathbb{R}^3 $
satisfies \eqref{eq:ss_strong2}, namely 
\begin{equation*}    \partial_{\tau}F-\frac{1}{a}\partial_{\xi_1}(\xi_1F)-\left(\frac{1}{a}-1\right)\partial_{\xi_2}(\xi_2F)-\left(\frac{1}{a}-1\right)\partial_{\xi_3}(\xi_3F)
   +  \partial_{\xi_1}(\xi_2 F )
   =
     {Q}^+[F,F]({\xi})-{Q}^-[F,F](\xi)
\end{equation*}
where ${Q}^+[F,F]$ and ${Q}^-[F,F]$ given as in \eqref{eq:ss_strong2_gain} and \eqref{eq:ss_strong2_loss} respectively and with $\xi=(\xi_1,\tilde \xi)\in \mathbb{R}^3$, $\tilde \xi=(\xi_2,\xi_3)\in \mathbb{R}^2$. 
In analogy with the boundary problem \eqref{eq:evG}, \eqref{eq:bdVG}, it is  possible to write a  boundary value problem also for the original solution $F$. More precisely, we consider the following
\begin{align}\label{eq:BVPF}
     \partial_{\tau}F  -\frac{1}{a}\partial_{\xi_1}(\xi_1F) & -\left(\frac{1}{a}-1\right)\partial_{\xi_2}(\xi_2F)-\left(\frac{1}{a}-1\right)\partial_{\xi_3}(\xi_3F)
   + \xi_2 \partial_{\xi_1} F = -{Q}^-[F,F] 
    \\
   F(0^{+},\xi_2,\xi_3,\tau)&
 = \frac {32}{2^a\, \xi_2} t^{a-1}\int_{\R} dx \int_{\mathbb{R}}dy_2\int_{\mathbb{R}}dy_3\, \int_0^{\pi} d\theta \,F\left( x_1+   |\tilde \xi|\left(t \sin(2\theta)\right)^{-1}, x_2+y_2,x_3+y_3, \tau\right) \times \notag \\&  \qquad \qquad \quad \times F\left(x_1-|\tilde \xi|\left(t \sin(2\theta)\right)^{-1},x_2-y_2,x_3-y_3,\tau \right) \frac{\sin\theta \, \vert \sin(2\theta)\vert^{a-1} }{  |\tilde \xi| ^{a+1}}   \,.   \label{eq:BVPF0}
\end{align} 
We consider the system of characteristics associated to 
\eqref{eq:BVPF}, \eqref{eq:BVPF0}:
\begin{align}\label{eq:characteristicsF}
    \frac{d \xi_1}{ds} &=-\frac{1}{a}\xi_1+\xi_2; \notag \\
    \frac{d\xi_2}{ds} & = -\left(\frac{1}{a}-1\right)\xi_2; \notag \\
    \frac{d\xi_3}{ds} & = -\left(\frac{1}{a}-1\right)\xi_3; \notag \\
    \frac{d F}{ds} &=
    \left(\frac{3}{a}-2\right)F -   U_0 \frac {\lambda(\tau_0)}{\varepsilon(\tau)^a} 
   \Lambda \left( \frac {\xi_1}{\varepsilon(\tau)} \right)\,F     
\end{align}
where $\Lambda$ is given as in \eqref{eq:statLambda} in Section \ref{sec:asymG}, namely
\begin{equation}\label{def:Mz_sec7}
\Lambda ( z)  =8 \pi  \int_{\mathbb{R}} d\eta  \, \left| z -\eta\right|^{-a}\frac{\sigma(\eta)}{|\eta|}\,.
\end{equation}
We remark that the equation for $F$ in the system of characteristics \eqref{eq:characteristicsF} has been obtained combining \eqref{eq:ss_strong2_loss}, the definition of the reduced two-dimensional distribution $G(\xi_1,\xi_2, \tau) =\int_{\Rl} F(\xi_1,\xi_2,\xi_3,\tau)\, d \xi_3$ (cf. \eqref{def:G}) and \eqref{eq:Gchar_c} or, equivalently, \eqref{eq:finalGappr_sec5}.

Moreover, we notice that, using \eqref{def:esplamtau0}, i.e. $\frac{\lambda(\tau_0)}{\varepsilon(\tau)^a}=\frac 1 {U_0}$ with $U_0=a$ (cf.\eqref{def:valueU0}), we have that $ U_0 \frac {\lambda(\tau_0)}{\varepsilon(\tau)^a} =1$ and thus we can rewrite the equation for $F$ as 
$$  \frac{d F}{ds}=\left(\frac{3}{a}-2\right)F - 
   \Lambda  \left( \frac {\xi_1}{\varepsilon(\tau)} \right)F\ . $$    
Taking as initial conditions $\xi_{1,0}=0,\xi_2(0)=\xi_{2,0},\xi_3(0)=\xi_{3,0}$, the solution to the system \eqref{eq:characteristicsF} is given by
\begin{align}
    \xi_1(\tau) &=\xi_{2,0}e^{-\frac{1}{a}(\tau-\tau_0)}(e^{\tau-\tau_0}-1); \label{eq:xi1F} \\
    \xi_2(\tau) & =\xi_{2,0}e^{-\left(\frac{1}{a}-1\right)(\tau-\tau_0)}; \label{eq:xi2F} \\
    \xi_3(\tau) & =\xi_{3,0}e^{-\left(\frac{1}{a}-1\right)(\tau-\tau_0)}; \label{eq:xi3F} \\
    F(\xi_1,\xi_2,\xi_3,\tau) & = F(0^+,\xi_{2,0},\xi_{3,0},\tau_0)\exp\left(\left(\frac{3}{a}-2\right)(\tau-\tau_0)\right)
    \exp \left(- \int_{\tau_0}^{\tau}  ds\,\Lambda \left( \frac {\xi_1(s)}{\varepsilon(\tau)} \right)  \right) \label{eq:charF}.
\end{align}
We also notice that, from \eqref{eq:xi1F}, \eqref{eq:xi2F} we have
\begin{equation} \label{eq:xi1xi2_F}
 \frac{\xi_1}{\xi_2}=1- e^{- (\tau-\tau_0)}  \text{ or, equivalently, } e^{(\tau-\tau_0)}= \frac{\xi_2}{\xi_2-\xi_1}\,. 
\end{equation} 
Moreover from \eqref{eq:charF} we obtain 
\begin{equation}\label{eq:solFchar_s7_1}
F(\xi_1,\xi_2,\xi_3,\tau)  =F\left(0^+,\xi_{2,0},\xi_{3,0},\tau_0\right)  \exp \left(- \int_{\tau_0}^{\tau}  ds\,
\Lambda  \left( \frac {\xi_1(s)}{\varepsilon(\tau)} \right)  \right)  \left(1-\frac{\xi_1}{\xi_2}\right)^{-\left(\frac{3}{a}-2\right)}, 
\end{equation}

We now consider the boundary value condition \eqref{eq:BVPF0} satisfied by the rescaled function $F$ and, setting $\eta_3=x_3+y_3$ and $u_3=x_2-y_2$, that implies $x_3=\frac 1 2 (\eta_3+u_3)$ $y_3=\frac 1 2 (\eta_3-u_3)$ and with Jacobian $dx_3\,dy_3=\frac 1 2 d\eta_3\,du_3$, we can rewrite \eqref{eq:BVPF0} as 
\begin{align*}
 & F(0^{+},\xi_2,\xi_3,\tau)
 \nonumber \\&
 = \frac {32}{2^a \, \xi_2} t^{a-1} \int_{\mathbb{R}}dx_1\int_{\mathbb{R}}dx_2 \int_{\mathbb{R}}dx_3 \int_{\mathbb{R}}dy_2 \int_{\mathbb{R}}dy_3 \, \int_0^{\pi} d\theta \,F\left( x_1+  |\tilde \xi|\left(t \sin(2\theta)\right)^{-1}, x_2+y_2,x_3+y_3, \tau\right) \times \notag \\& \qquad \qquad \qquad \qquad \qquad \times F\left(x_1-|\tilde \xi|\left(t \sin(2\theta)\right)^{-1},x_2-y_2,x_3-y_3,\tau \right) \frac{\sin\theta \, \vert \sin(2\theta)\vert^{a-1} }{  |\tilde \xi| ^{a+1}}     \\& 
 = \frac {16}{2^{a}\, \xi_2} t^{a-1}\int_{\mathbb{R}}dx_1\int_{\mathbb{R}}dx_2 \int_{\mathbb{R}}d\eta_3 \int_{\mathbb{R}}dy_2 \int_{\mathbb{R}}du_3 \, \int_0^{\pi} d\theta \, F\left( x_1+  |\tilde \xi|\left(t \sin(2\theta)\right)^{-1}, x_2+y_2,\eta_3, \tau\right) \times \notag \\& \qquad \qquad \qquad \qquad \qquad \times F\left(x_1-|\tilde \xi|\left(t \sin(2\theta)\right)^{-1},x_2-y_2,u_3,\tau \right) \frac{\sin\theta \, \vert \sin(2\theta)\vert^{a-1} }{  |\tilde \xi| ^{a+1}}    \nonumber\\&
 =\frac {16}{2^{a}\, \xi_2} t^{a-1} \int_{\mathbb{R}}dx_1\int_{\mathbb{R}}dx_2 \int_{\mathbb{R}}dy_2  \, \int_0^{\pi} d\theta \, G\left( x_1+  |\tilde \xi|\left(t \sin(2\theta)\right)^{-1}, x_2+y_2, \tau\right) \times \notag \\& \qquad \qquad \qquad \qquad \qquad \times G\left(x_1-|\tilde \xi|\left(t \sin(2\theta)\right)^{-1},x_2-y_2,\tau \right) \frac{\sin\theta \, \vert \sin(2\theta)\vert^{a-1} }{  |\tilde \xi| ^{a+1}}   
 \end{align*} 
 where we used that $G\left(z_1,z_2,\tau \right)=\int_{\mathbb{R}}dz_3 F\left(z_1 ,z_2,z_3, \tau \right)$. 
 Following the same strategy used in Appendix \ref{ss:justificationGeq}
 to deduce \eqref{eq:G_gain_4} from \eqref{eq:G_gain}, even if in those computations there was an additional integral in the $\xi_3$ variable that we are not performing here, we arrive at 
 \begin{align}\label{eq:Fboundary2}
  F(0^{+},\xi_2,\xi_3,\tau)&= 
 \frac{ 16}{2^a \, \xi_2} t^{a}  \int_{\mathbb{R}^2} dx_1dx_2 \int_{\mathbb{R}}d{y_2}  \int_{\frac{|\tilde{\xi}|}{t}}^{+\infty } dz \, G\left( x_1+ z, x_2+y_2, \tau\right) G\left(x_1- z,x_2-y_2,\tau \right) \notag \\& \ \ 
  \qquad \times 
  \frac{ \Big(\sin(\frac \psi 2)+\cos(\frac \psi 2)\Big)}{ |\tilde\xi|^{a+2} } \frac{|\sin(\psi)|^{a+1}}{\sqrt{1-\sin^2(\psi)}}
  \end{align} 
  where we recall that $\psi$ is given by means of \eqref{eq:ChangeVar_zpsi}. 
Notice that the integral concentrates in the region where collisions take place, namely where $|\xi| \approx \delta$. 
\smallskip 

We now recall the transformation  
$
z = \frac{|\tilde{\xi}|}{t \sin(\psi)}$ with Jacobian $dz = -\frac{|\tilde{\xi}|}{t \sin^{2}(\psi)} \cos(\psi) d\psi,$ and with $\psi \in \left[0, \frac{\pi}{2}\right)$, $t = e^{\tau}$ and $\tilde{\xi} = (\xi_2, \xi_3)$. We consider again the gain term $K^+[G,G](\xi)$, which is given by
\begin{multline}\label{eq:K+opA}
K^{+}[G,G](\xi) = \frac{32}{2^{a}} t^{a} \delta(\xi_{1}) \int_{0}^{+\infty} d\xi_{3} \int_{\mathbb{R}^{2}} dx_{1} dx_{2} \int_{\mathbb{R}} dy_{2} \left[ \int_{\frac{|\tilde{\xi}|}{t}}^{+\infty} dz \, G(x_{1}+z, x_{2}+y_{2}, \tau) \right. \\
\left. \times G(x_{1}-z, x_{2}-y_{2}, \tau) \frac{(\sin(\frac{\psi}{2}) + \cos(\frac{\psi}{2}))}{|\tilde{\xi}|^{a+2}} \frac{|\sin(\psi)|^{a+1}}{\cos(\psi)} \right]\, . 
\end{multline}
In particular, we look at the integrand on the r.h.s. of \eqref{eq:K+opA} and introduce the auxiliary function kernel 
\begin{equation}
\mathcal{A} \equiv \frac{(\sin(\frac{\psi}{2}) + \cos(\frac{\psi}{2}))}{|\tilde{\xi}|^{a+2}} \frac{|\sin(\psi)|^{a+1}}{\sqrt{1-\sin^{2}(\psi)}}\label{def:mathcalA}
\end{equation}
which allows to rewrite
\begin{multline}\label{eq:K+opA2}
K^{+}[G,G](\xi) = \frac{32}{2^{a}} t^{a} \delta(\xi_{1}) \int_{0}^{+\infty} d\xi_{3} \int_{\mathbb{R}^{2}} dx_{1} dx_{2} \int_{\mathbb{R}} dy_{2} \left[ \int_{\frac{ |\tilde{\xi}|}{t}}^{+\infty} dz \, G(x_{1}+z, x_{2}+y_{2}, \tau) \right. \\
\left. \times G(x_{1}-z, x_{2}-y_{2}, \tau) \mathcal{A} \right]\, . \notag
\end{multline}
Using the change of variables above we write $\sin(\psi) = \frac{|\tilde{\xi}|}{t z}$ and, thanks to the trigonometric half-angle identities, it follows   
\begin{align*}
\cos\left(\frac{\psi}{2}\right) &= \frac{1}{\sqrt{2}} \sqrt{1 + \sqrt{1 - \sin^{2}(\psi)}} = \frac{1}{\sqrt{2}} \sqrt{1 + \sqrt{1 - \left(\frac{|\tilde{\xi}|}{t z}\right)^{2}}} \\
\sin\left(\frac{\psi}{2}\right) &= \frac{1}{\sqrt{2}} \sqrt{1 - \sqrt{1 - \sin^{2}(\psi)}} = \frac{1}{\sqrt{2}} \sqrt{1 - \sqrt{1 - \left(\frac{|\tilde{\xi}|}{t z}\right)^{2}}}\ .
\end{align*}
Using the two identities above in $\mathcal{A}$ we arrive at 
\begin{equation}
\mathcal{A} = \frac{\sqrt{2}}{2} \frac{1}{|\tilde{\xi}|(t z)^{a+1}} \frac{\sqrt{1 + \sqrt{1 - (\frac{|\tilde{\xi}|}{t z})^{2}}} + \sqrt{1 - \sqrt{1 - (\frac{|\tilde{\xi}|}{t z})^{2}}}}{\sqrt{1 - (\frac{|\tilde{\xi}|}{t z})^{2}}} \ .\label{def:mathcalA2}
\end{equation}
We analyze the boundary value 
$
F(0^{+}, \xi_{2}, \xi_{3}, \tau)$ that now reads
\begin{multline}
F(0^{+}, \xi_{2}, \xi_{3}, \tau) = \frac{ 16}{2^{a}}\frac{t^a}{\xi_2} \int_{\mathbb{R}^{2}} dx_{1} dx_{2} \int_{\mathbb{R}} dy_{2} \int_{\frac{|\tilde{\xi}|}{t}}^{+\infty} dz \, G(x_{1}+z, x_{2}+y_{2}, \tau) G(x_{1}-z, x_{2}-y_{2}, \tau)\mathcal{A}\,. 
\end{multline}
Replacing $\mathcal{A}$ by the right hand side of \eqref{def:mathcalA2} into the integral, we have:
\begin{align}
\label{eq:Fbdexp}
F(0^{+}, \xi_{2}, \xi_{3}, \tau) = & \frac{8\sqrt{2}}{2^{a}} \frac{1}{ \xi_{2} |\tilde{\xi}|\, t} \int_{\mathbb{R}^{2}} dx_{1} dx_{2} \int_{\mathbb{R}} dy_{2} \int_{\frac{|\tilde{\xi}|}{t}}^{+\infty} dz \, G(x_{1}+z, x_{2}+y_{2}, \tau) G(x_{1}-z, x_{2}-y_{2}, \tau) \nonumber\\&
\times \frac{1}{z^{a+1}} \frac{\sqrt{1 + \sqrt{1 - \left(\frac{|\tilde{\xi}|}{t z}\right)^{2}}} + \sqrt{1 - \sqrt{1 - \left(\frac{|\tilde{\xi}|}{t z}\right)^{2}}}}{\sqrt{1 - \left(\frac{|\tilde{\xi}|}{t z}\right)^{2}}}  
\end{align}
We now recall the approximation for $G$ in terms of the  profile $\sigma$  given as in \eqref{eq:finalGappr_sec5}, namely
\begin{equation*}
G(\xi_1, \xi_2,\tau)\sim  \delta(\xi_1-\xi_2) Q(\xi_2,\tau)\sim \delta(\xi_1-\xi_2) \frac { \lambda(\tau_0)} {|\xi_2|} U_0 \Sigma(\xi_2,\tau)= \delta(\xi_1-\xi_2)\frac {\lambda(\tau_0)} {|\xi_2|}U_0 \sigma\left(\frac{\xi_2}{\varepsilon(\tau)}, \tau\right)\,,
\end{equation*} 
and also the fact that the time $\tau_0=\tau_0(\tau, \xi)$ (cf.~\eqref{def:tau0in}). 
Using this approximation 
in \eqref{eq:Fbdexp} we then obtain
\begin{align}\label{eq:Fbdexp2}
& F(0^{+}, \xi_{2}, \xi_{3}, \tau) = \nonumber\\&    
\frac{{8}\sqrt{2}}{2^{a}} \frac{1}{ \left(\xi_{2} |\tilde{\xi}|\, t\right)} 
\int_{\mathbb{R}} dx_{1} \int_{\mathbb{R}} dx_{2} \int_{\frac{|\tilde{\xi}|}{t}}^{+\infty} dz \frac{(\lambda(\tau_0))^{2} U_0^2}{|x_{2}+z||x_{2}-z|} 
\delta(x_{1}+z-x_{2}-y_{2}) \delta(x_{1}-z-x_{2}+y_{2})
 \notag \\&
\qquad \times \frac{1}{z^{a+1}} 
\sigma\left(\frac{x_{2}+y_{2}}{\varepsilon(\tau)}, \tau\right) \sigma\left(\frac{x_{2}-y_{2}}{\varepsilon(\tau)}, \tau\right) 
\frac{\sqrt{1 + \sqrt{1 - \left(\frac{|\tilde{\xi}|}{t z}\right)^{2}}} + \sqrt{1 - \sqrt{1 - \left(\frac{|\tilde{\xi}|}{t z}\right)^{2}}}}{\sqrt{1 - \left(\frac{|\tilde{\xi}|}{t z}\right)^{2}}}\,.
\end{align}
Since the product of the Dirac deltas simplifies as:
\begin{equation*}
\delta(x_{1}+z-x_{2}-y_{2}) \delta(x_{1}-z-x_{2}+y_{2}) = \delta(2(z-y_{2})) \delta(x_{1}-z-x_{2}+y_{2})=\frac{1}{2} \delta(z-y_{2}) \delta(x_{1}-x_{2})\, ,
\end{equation*}
from \eqref{eq:Fbdexp2}
it then follows
\begin{align*}
    F(0^{+}, \xi_{2}, \xi_{3}, \tau) =&    
 \frac{8\sqrt{2}}{2^{a}} \frac{1}{ \left(\xi_{2} |\tilde{\xi}|\, t\right)} 
 \int_{\mathbb{R}} dx_{1} \int_{\mathbb{R}} dx_{2} \int_{\frac{|\tilde{\xi}|}{t}}^{+\infty} dz \frac{(\lambda(\tau_0))^{2} U_0^2}{|x_{2}+z||x_{2}-z|} \frac{1}{2} 
 \delta(z-y_{2}) \delta(x_{1}-x_{2})
 \\&
\qquad \times \frac{1}{z^{a+1}} 
\sigma\left(\frac{x_{2}+y_{2}}{\varepsilon(\tau)}, \tau\right) \sigma\left(\frac{x_{2}-y_{2}}{\varepsilon(\tau)}, \tau\right) 
\frac{\sqrt{1 + \sqrt{1 - \left(\frac{|\tilde{\xi}|}{t z}\right)^{2}}} + \sqrt{1 - \sqrt{1 - \left(\frac{|\tilde{\xi}|}{t z}\right)^{2}}}}{\sqrt{1 - \left(\frac{|\tilde{\xi}|}{t z}\right)^{2}}}\, .
\end{align*}
Thus, integrating over $x_1$ and $y_2$, we obtain
\begin{align}
\label{eq:Fbdexp4}
F(0^{+}, \xi_{2}, \xi_{3}, \tau) = \frac{4\sqrt{2}}{2^{a}} \frac 1 {\left(\xi_{2} |\tilde{\xi}|\, t\right)} \int_{\mathbb{R}} dx_{2} \int_{\frac{|\tilde{\xi}|}{t}}^{+\infty} dz \frac{(\lambda(\tau_0))^{2} U_0^2}{|x_{2}+z||x_{2}-z|} \frac{1}{z^{a+1}} \sigma\left(\frac{x_{2}+z}{\varepsilon(\tau)}, \tau\right) \sigma\left(\frac{x_{2}-z}{\varepsilon(\tau)}, \tau\right)\notag \\
 \times \frac{\sqrt{1 + \sqrt{1 - \left(\frac{|\tilde{\xi}|}{t z}\right)^{2}}} + \sqrt{1 - \sqrt{1 - \left(\frac{|\tilde{\xi}|}{t z}\right)^{2}}}}{\sqrt{1 - \left(\frac{|\tilde{\xi}|}{t z}\right)^{2}}}
\end{align}
Performing the change of variables 
$ x_2= \varepsilon(\tau)\eta $ and $ z= \varepsilon(\tau)\zeta $ in \eqref{eq:Fbdexp4} then yields
\begin{align}\label{eq:Fbdexp5}
F(0^{+}, \xi_{2}, \xi_{3}, \tau) = \frac{4\sqrt{2}}{2^{a}}\frac {U_0^2} {\left(\xi_{2} |\tilde{\xi}|\, t\right)} \frac{(\lambda(\tau_0))^{2}}{(\varepsilon(\tau))^{a+1}} \int_{\mathbb{R}} d\eta \int_{\frac{|\tilde{\xi}|}{t \varepsilon(\tau)}}^{+\infty} \frac{d\zeta }{\zeta^{a+1}}  \frac{\sigma(\eta+\zeta, \tau)\,\sigma(\eta-\zeta, \tau)}{\vert \eta+\zeta\vert \, \vert \eta-\zeta\vert } \notag 
\\
\times \frac{\sqrt{1 + \sqrt{1 - \left(\frac{|\tilde{\xi}|}{t \varepsilon(\tau) \zeta}\right)^{2}}} + \sqrt{1 - \sqrt{1 - \left(\frac{|\tilde{\xi}|}{t \varepsilon(\tau) \zeta}\right)^{2}}}}{\sqrt{1 - \left(\frac{|\tilde{\xi}|}{t \varepsilon(\tau) \zeta}\right)^{2}}}
\end{align}
Moreover, using
\eqref{eq:lamHbar9} and \eqref{eq:lamHbar10}, i.e. 
$\lambda(\tau)   \sim \frac{\left( \lambda(\tau_0)\right)^2 U_0^2  }{\left(\varepsilon(\tau)\right)^{a}}  \frac{(1-a)}{a}
$  for $|\xi|\approx \varepsilon(\tau)$, 
we obtain  $\frac{\left( \lambda(\tau_0)\right)^2 U_0^2  }{\left(\varepsilon(\tau)\right)^{a}}\sim  \lambda(\tau) \frac{a}{(1-a)}$ which allows to simplify the formula \eqref{eq:Fbdexp5} that thus becomes
\begin{align}\label{eq:Fbdexp5bis}
F(0^{+}, \xi_{2}, \xi_{3}, \tau) = \frac{4\sqrt{2}}{2^{a}}\frac {1} {\left(\xi_{2} |\tilde{\xi}|\, t\right)}  \frac{a}{(1-a)} \frac{\lambda(\tau)}{\varepsilon(\tau)} \int_{\mathbb{R}} d\eta \int_{\frac{|\tilde{\xi}|}{t \varepsilon(\tau)}}^{+\infty} \frac{d\zeta }{\zeta^{a+1}}  \frac{\sigma(\eta+\zeta, \tau)\,\sigma(\eta-\zeta, \tau)}{\vert \eta+\zeta\vert \, \vert \eta-\zeta\vert } \notag 
\\
\times \frac{\sqrt{1 + \sqrt{1 - \left(\frac{|\tilde{\xi}|}{t \varepsilon(\tau) \zeta}\right)^{2}}} + \sqrt{1 - \sqrt{1 - \left(\frac{|\tilde{\xi}|}{t \varepsilon(\tau) \zeta}\right)^{2}}}}{\sqrt{1 - \left(\frac{|\tilde{\xi}|}{t \varepsilon(\tau) \zeta}\right)^{2}}}
\end{align} 
We notice that in Section \ref{sec:asymG}, we obtained the asymptotics of the solution $G(,\xi_1,\xi_2,\tau)$ to the reduced model, reformulating the Boltzmann equation as the boundary value  problem 
\eqref{eq:BV_G_a}, 
\eqref{eq:bdG_a}. The boundary value $G(0^+,\xi_1,\xi_2,\tau)$ was written in the self-similar form \eqref{def:Gbd_selfsim}, (see also \eqref{eq:ansatzH}). It is then natural to expect that, also in this case, we could write $F(0^+,\xi_2,\xi_3,\tau)$ in self-similar form, proportional to the function $\lambda((\tau)$. In fact, we observe that in the equation \eqref{eq:Fbdexp5} for $F(0^+,\xi_2,\xi_3,\tau)$  the factor
$(\lambda((\tau_0))^2$ appears as a consequence of the fact that we obtain $F(0^+,\xi_2,\xi_3,\tau)$ using the boundary condition and the gain term $Q_{+}(F,F)$, that can be written in terms of quadratic terms of $G$. This implies that the boundary value $F(0^+,\xi_2,\xi_3,\tau)$  is proportional to  $(\lambda((\tau_0))^2$. In Section \ref{sec:asymG} we obtained relations between the functions $\lambda(\tau_0), \lambda(\tau), \varepsilon(\tau)$. Using these relations we can then rewrite the boundary value $F(0^+,\xi_2,\xi_3,\tau)$  as a function proportional to $\lambda(\tau)$, similarly to what has been done in Section \ref{sec:asymG} for $G(0^+,\xi_2,\tau)$.

We introduce the auxiliary function $\Omega(\zeta, \tau)$ defined as:
\begin{equation}\label{def:Omegazeta}
\Omega(\zeta, \tau) \equiv \int_{\mathbb{R}} d\eta \frac{\sigma(\eta+\zeta, \tau) \sigma(\eta-\zeta, \tau)}{|\eta+\zeta| |\eta-\zeta|},\quad  \zeta>0\,.
\end{equation}
We will assume that $\sigma(\eta,\tau)$ approaches a steady state $\sigma(\eta)$ (cf.~\eqref{eq:pde_b_stat}). Thus, also $\Omega(\zeta, \tau)=\Omega(\zeta)$.  
Additionally, we notice that $\Omega(\zeta)$ is well-defined provided $\sigma(s)$ vanishes sufficiently fast as $s \to \infty$ and behaves as $\sigma(s)\leq C s^\beta$, with $\beta \geq a$ near the origin, as $s \to 0$.

Using  \eqref{def:Omegazeta}, we can then rewrite \eqref{eq:Fbdexp5} as 
\begin{align}\label{eq:Fbdexp6}
F(0^{+}, \xi_{2}, \xi_{3}, \tau) = \frac{4\sqrt{2}}{2^{a}}\frac {1} {\left(\xi_{2} |\tilde{\xi}|\, t\right)}   \frac{a}{(1-a)} \frac{\lambda(\tau)}{\varepsilon(\tau)} 
\int_{\frac{|\tilde{\xi}|}{t \varepsilon(\tau)}}^{+\infty} \frac{d\zeta}{\zeta^{a+1}} \Omega(\zeta, \tau)\notag \\
\times \frac{\sqrt{1 + \sqrt{1 - \left(\frac{|\tilde{\xi}|}{t \varepsilon(\tau) \zeta}\right)^{2}}} + \sqrt{1 - \sqrt{1 - \left(\frac{|\tilde{\xi}|}{t \varepsilon(\tau) \zeta}\right)^{2}}}}{\sqrt{1 - \left(\frac{|\tilde{\xi}|}{t \varepsilon(\tau) \zeta}\right)^{2}}}
\end{align}
Setting
$
r=\frac{|\widetilde{\xi}|}{t\varepsilon(\tau)}$ we introduce 
\begin{align}\label{def:Wr}
W(r):=\int_{r}^{\infty}\frac{d\zeta}{\zeta^{a+1}}\frac{\sqrt{1 + \sqrt{1 - \left(\frac{r}{\zeta} \right)^{2}}} + \sqrt{1 - \sqrt{1 - \left(\frac{r}{\zeta}\right)^{2}}}}{\sqrt{1 - \left(\frac{r}{\zeta}\right)^{2}}}\,\Omega(\zeta)\,. 
\end{align}
Hence, using \eqref{def:Wr}, \eqref{eq:Fbdexp6} reads as
\begin{align}\label{eq:Fbdexp7}
F(0^{+}, \xi_{2}, \xi_{3}, \tau) = \frac{4\sqrt{2}}{2^{a}}\frac {e^{-\tau}} {\left(\xi_{2} |\tilde{\xi}|\right)} {\frac{a}{(1-a)} \frac{\lambda(\tau)}{\varepsilon(\tau)} } W\left(\frac{|\widetilde{\xi}|}{e^{\tau}\,\varepsilon(\tau)}\right) \, .
\end{align} 
Combining \eqref{eq:Fbdexp7} (with $\tau$ replaced by $\tau_0$) with \eqref{eq:solFchar_s7_1}
\begin{equation}\label{eq:solFchar_s7_2}
F(\xi_1,\xi_2,\xi_3,\tau)  =\frac{4\sqrt{2}}{2^{a}} \frac {e^{-\tau_0}}{\left(\xi_{2,0} |\tilde{\xi}_0|\, \right)} \frac{a}{(1-a)} \frac{\lambda(\tau_0)}{\varepsilon(\tau_0)}  W\left(\frac{|\widetilde{\xi}_0|}{e^{\tau_0}\varepsilon(\tau _0)}\right) 
\exp \left(- \int_{\tau_0}^{\tau} ds\,  \Lambda  \left( \frac {\xi_1(s)}{\varepsilon(\tau)} \right)  \right)  \left(1-\frac{\xi_1}{\xi_2}\right)^{-\left(\frac{3}{a}-2\right)}, 
\end{equation} 
where $\tilde{\xi}_0=(\xi_{2,0}, \xi_{3,0})$. 
Using now \eqref{eq:xi1F}-\eqref{eq:xi3F} we have $\xi_{2,0}=\frac{\xi_2^{\frac{1}{a}-1}}{(\xi_2-\xi_1)^{\frac{1}{a}-1}} $ and $\xi_{3,0}=\frac{\xi_2^{\frac{1}{a}-1}\xi_3}{(\xi_2-\xi_1)^{\frac{1}{a}-1}}$ can rewrite 
$$\tilde{\xi}_0=(\xi_{2,0}, \xi_{3,0})= \frac{\xi_2^{\frac{1}{a}-1}}{(\xi_2-\xi_1)^{\frac{1}{a}-1}} (\xi_{2}, \xi_{3})$$ 
and 
$$ \vert \tilde{\xi}_0 \vert = \frac{\xi_2^{\frac{1}{a}-1}}{(\xi_2-\xi_1)^{\frac{1}{a}-1}} \vert \tilde{\xi}\vert $$
so that \eqref{eq:solFchar_s7_2} becomes 
\begin{align}\label{eq:solFchar_s7_3}
F(\xi_1,\xi_2,\xi_3,\tau)  =&\frac{4\sqrt{2}}{2^{a}}\frac {e^{-\tau}}{\xi_{2} |\tilde{\xi}|} \frac{a}{(1-a)}\frac{\lambda(\tau_0)}{\varepsilon(\tau_0)} \, e^{\tau-\tau_0} \,\left(\frac {\xi_2} {\xi_2-\xi_1}\right)^{\frac{1}{a}}     \nonumber \\& 
\quad \times W \left( \frac{\xi_2^{\frac{1}{a}-1} \vert \tilde{\xi}\vert \,e^{\tau-\tau_0}}{(\xi_2-\xi_1)^{\frac{1}{a}-1}  \, e^{\tau}\varepsilon(\tau_0)}\right) 
\exp \left(- \int_{\tau_0}^{\tau} ds\, \Lambda   \left( \frac {\xi_1(s)}{\varepsilon(\tau)} \right)  ds \right) \nonumber \\&
= \frac{4\sqrt{2}}{2^{a}}\frac {e^{-\tau}}{\xi_{2} |\tilde{\xi}|} \frac{a}{(1-a)}\frac{\lambda(\tau_0)}{\varepsilon(\tau_0)}  \,\left(\frac {\xi_2} {\xi_2-\xi_1}\right)^{\frac{1}{a}+1}    \nonumber \\& 
\quad \times 
W \left( \frac{ \xi_2^{\frac{1}{a}} \,\vert \tilde{\xi}\vert }{(\xi_2-\xi_1)^{\frac{1}{a}}  \, e^{\tau}\varepsilon(\tau_0)}\right) 
\exp \left(- \int_{\tau_0}^{\tau} ds\, \Lambda \left( \frac {\xi_1(s)}{\varepsilon(\tau)} \right)  ds \right) 
\end{align}
where in the second equality we used that $e^{\tau-\tau_0}=\frac{\xi_2}{\xi_2-\xi_1}$ (cf.~\eqref{eq:xi1xi2_F}). 
\smallskip

We now define 
\begin{equation}\label{def:sigmastar}
    \sigma_{\ast}(\xi_1,\xi_2 , \tau):=
\exp \left(- \int_{\tau_0}^{\tau} ds\, \Lambda   \left( \frac {\xi_1(s)}{\varepsilon(\tau)} \right)  \right) 
\end{equation}
and, for easy reading ,we set
\begin{equation}\label{eq:auxmathcalAdef}
\mathcal{I}:=\frac{4\sqrt{2}}{2^{a}} \, W \left( \frac{ \xi_2^{\frac{1}{a}} \,\vert \tilde{\xi}\vert }{(\xi_2-\xi_1)^{\frac{1}{a}}  \, e^{\tau}\varepsilon(\tau_0)}\right) \frac{\xi_2^{\frac 1 a +1}} {(\xi_2-\xi_1)^{\frac{1}{a}+1}e^{\tau}\varepsilon(\tau_0)} \frac 1 {\vert \tilde{\xi}\vert\,\xi_2}
\end{equation}
so that \eqref{eq:solFchar_s7_3} becomes
\begin{equation}
\label{eq:solFchar_s7_4a}
F(\xi_1,\xi_2,\xi_3,\tau)  = \frac{a}{(1-a)} \lambda(\tau_0) \,\mathcal{I}\, \sigma_{\ast}\left( \xi_1,\xi_2 ,\tau\right)
\end{equation} 
We now rewrite $\mathcal{I}$ given as in \eqref{eq:auxmathcalAdef} as follows:
\begin{align}\label{eq:auxmathcalA}
    & \mathcal{I}
    = \frac{4\sqrt{2}}{2^{a}} \,
    W \left( \left( 
    \frac{\xi_2 \vert \tilde{\xi}\vert^{a} }{(\xi_2-\xi_1) \, e^{a \tau}(\varepsilon(\tau_0))^{a}}\right)^{\frac{1}{a}} \right)
    \left( 
    \frac{\xi_2 \vert \tilde{\xi}\vert^{a} }{(\xi_2-\xi_1) \, e^{a \tau}(\varepsilon(\tau_0))^{a}}\right)^{\frac 1 a +1}
    \frac {1}{
     \vert \tilde{\xi}\vert ^{a+2}\, \xi_2} 
\end{align}
In order to simplify the formula we now set 
$$s= \frac{(\xi_2-\xi_1) \, e^{a \tau}(\varepsilon(\tau_0))^{a}}{\xi_2 \vert \tilde{\xi}\vert^{a} }$$
and define 
\begin{equation}
    \label{def:tildeW_F}
    \widetilde{W}(s):= \frac{4 \sqrt{2}}{2^{a}}W\left( s ^{{-\frac 1 a }}\right)   s ^{-\left(\frac 1 a+1\right)}
\end{equation}
so that we can rewrite \eqref{eq:auxmathcalA} as 
\begin{equation}\label{eq:auxiliaryA_2}
\mathcal{I}=\frac{e^{a  \tau}\,(\varepsilon(\tau_0))^{a}}  { \vert \tilde{\xi}\vert ^{a+2}\, \xi_2}    \widetilde{W}\left( \frac{(\xi_2-\xi_1) \, e^{a \tau}(\varepsilon(\tau_0))^{a}}{\xi_2 \vert \tilde{\xi}\vert^{a} }\right) \,.
\end{equation}
We expect that, if $W_0:=\int_{0}^\infty \, \widetilde{W}(s) ds <\infty$, we can approximate the r.h.s. of \eqref{eq:auxiliaryA_2}
by a Dirac delta. Notice that the integrability of the function $\widetilde{W}(s)$ will be studied in Section \ref{ssec:asymp_auxiliary} (cf. Lemma \ref{lem:tildeWr}). 
 More precisely, the following approximation holds true:
\begin{equation}\label{eq:sizeDiracApprox}
e^{a  \tau}\,(\varepsilon(\tau_0))^{a} \;    \widetilde{W}\left( \frac{(\xi_2-\xi_1) \, e^{a \tau}(\varepsilon(\tau_0))^{a}}{\xi_2 \vert \tilde{\xi}\vert^{a} }\right) 
    \sim 
W_0\, \delta\left(\frac {(\xi_2-\xi_1)}{\xi_2 \vert \tilde{\xi}\vert^{a} } \right)=W_0\, \xi_2 \vert \tilde{\xi}\vert^{a}\,  \delta\left(\xi_2-\xi_1  \right) \quad \text{as} \ \ \tau\to \infty\,,
\end{equation}
where $W_0=\int_{0}^\infty \, \widetilde{W}(s) ds <\infty$. 
Therefore, using  \eqref{eq:sizeDiracApprox} in \eqref{eq:auxiliaryA_2} we obtain 
\begin{equation}\label{eq:auxiliaryA_3}
     \mathcal{I}\sim  \frac {W_0\,} { \vert \tilde{\xi}\vert ^2}  \delta\left(\xi_2-\xi_1  \right)
 \end{equation}
 and then, using \eqref{eq:auxiliaryA_3} in \eqref{eq:solFchar_s7_4a}, it follows
 \begin{equation}
\label{eq:solFchar_s7_5}
F(\xi_1,\xi_2,\xi_3,\tau)  \sim 
\frac a {1-a}\,\lambda(\tau_0)  
\frac {W_0\,} { \vert \tilde{\xi}\vert ^2}  \delta\left(\xi_2-\xi_1  \right)\sigma_{\ast}\left( \xi_1,\xi_2 ,\tau\right)
\end{equation}
which holds true for $|\xi| \lesssim \varepsilon(\tau)e^\tau$.

We further notice that, in the collision region, i.e. in the region where $|\xi| \approx \varepsilon(\tau)$,
we can use that  $  \lambda(\tau)   \sim   (1-a)\, \lambda(\tau_0)$ (see \eqref{eq:lamHbar11}) and thus $  \lambda(\tau_0)   \sim    \frac{\lambda(\tau)}{(1-a)}$. Thus,  in the collision region, \eqref{eq:solFchar_s7_5} becomes 
\begin{equation}
 F(\xi_1,\xi_2,\xi_3,\tau)  \sim \frac{a}{(1-a)^2} \lambda(\tau)\frac {W_0\,} { \vert \tilde{\xi}\vert ^2}  \delta\left(\xi_2-\xi_1  \right)\sigma_{\ast}\left( \xi_1,\xi_2 ,\tau\right) \,.
\end{equation}

 \medskip

 We now examine the cutoff properties of the function $\sigma_{\ast}(\xi_1,\xi_2 , \tau)$ introduced in \eqref{def:sigmastar}. We first consider the case  of the external region $|\xi |\gtrsim e^{\tau}\varepsilon(\tau)$. 
From the definition of the function $\Lambda(z)$ given as in \eqref{def:Mz_sec7}, a direct computation implies the following estimate
\begin{equation}\label{est:Mzasym} 
\Lambda( z)  = 8\pi\int_{\mathbb{R}} d\eta  \, \left| z -\eta\right|^{-a}\frac{\sigma(\eta)}{|\eta|} \leq C \, \frac{\log(|z|)}{|z|^a}\,.
\end{equation}
with the constant $C>0$. Therefore, 
due to the fact that $\xi_1(s)$ (see \eqref{eq:xi1F}) increases exponentially as $\tau-s\to \infty $, it then follows that the main contribution 
to the integral on the r.h.s. of  \begin{equation}\label{eq:approxsigmastar}
   \sigma_{\ast}(\xi_1,\xi_2 , \tau)=\exp \left(-\int_{\tau_0}^{\tau} ds  \Lambda  \left( \frac {\xi_1(s)}{\varepsilon(\tau)} \right)  \right)
   \end{equation}
 is due to the range of values of $s$ in which $\tau-s=O(1)$ for $\vert \xi \vert \lesssim 1$.   
From the equation for the characteristics (see \eqref{eq:xi1F}) we have 
\begin{equation}\label{eq:xi1xi2_sec7}
\xi_1(s)=\frac{\xi_2^{\frac 1 a}}{(\xi_2-\xi_1)^{\frac 1 a -1}} e^{-{\frac 1 a} (s-\tau_0)}\left(e^{s-\tau_0}-1\right)
\end{equation}
Using  \eqref{eq:xi1xi2_sec7} in \eqref{eq:approxsigmastar} we then obtain  
\begin{align}\label{eq:approxsigmastar2}
    \sigma_{\ast}(\xi_1,\xi_2 , \tau)&= 
    \exp \left(- 
    \int_{\tau_0}^{\tau} ds  \,\Lambda  \left( \frac {\xi_1(s)}{\varepsilon(\tau)} \right)  \right)\nonumber \\&  = 
    \exp \left(-
    \int_{\tau_0}^{\tau} ds  \,\Lambda  \left( 
    \frac{\xi_2^{\frac 1 a}}{(\xi_2-\xi_1)^{\frac 1 a -1}} e^{-{\frac 1 a} (s-\tau_0)}\left(e^{s-\tau_0}-1\right) \frac{e^{-{\frac 1 a} (s-\tau_0)}\left(e^{s-\tau_0}-1\right)}{\varepsilon(\tau)} \right)  \right)\nonumber \\&
= \exp \left(-
\int_{0}^{\tau-\tau_0} d\sigma  \,\Lambda  \left( 
    \frac{\xi_2^{\frac 1 a}}{(\xi_2-\xi_1)^{\frac 1 a -1}} \frac{e^{-{\frac 1 a} \sigma}\left(e^{\sigma}-1\right) }{\varepsilon(\tau)} \right) \right) \nonumber \\&
= \exp \left(- 
\int_{0}^{\tau-\tau_0} d\sigma  \, \Lambda \left( 
    \frac{\xi_2^{\frac 1 a}}{(\xi_2-\xi_1)^{\frac 1 a -1}} \frac{e^{(1-{\frac 1 a})\sigma}\left(1-e^{-\sigma}\right) }{\varepsilon(\tau)} \right) \right) 
\end{align}
where in the third line we used the change of variable $s=\tau_0+\sigma$. We now perform a further change of variable $\sigma=(\tau-\tau_0)-u$ with $d\sigma=-du$ so that the r.h.s. of \eqref{eq:approxsigmastar2} becomes
 \begin{align}\label{eq:approxsigmastar3}
    \sigma_{\ast}(\xi_1,\xi_2 , \tau)=
    \exp \left(-    \int_{0}^{\tau-\tau_0} du  \, \Lambda \left( 
    \frac{\xi_2^{\frac 1 a}}{(\xi_2-\xi_1)^{\frac 1 a -1}} \frac{e^{(1-{\frac 1 a})(\tau-\tau_0)}e^{-(1-{\frac 1 a})u }\left(1-e^{-(\tau-\tau_0)}e^{u}\right) }{\varepsilon(\tau)} \right) 
    \right) 
\end{align}
We observe that 
\begin{align*}
     & \frac{\xi_2^{\frac 1 a}}{(\xi_2-\xi_1)^{\frac 1 a -1}} \frac{e^{(1-{\frac 1 a})(\tau-\tau_0)}e^{-(1-{\frac 1 a}u) }\left(1-e^{-(\tau-\tau_0)}e^{u}\right) }{\varepsilon(\tau)} \nonumber \\&
     \sim 
     \frac{\xi_2^{\frac 1 a}}{(\xi_2-\xi_1)^{\frac 1 a -1}} \left(\frac {\xi_2}{\xi_2-\xi_1}\right)^{(1-{\frac 1 a}) } \frac {e^{({\frac 1 a}-1)u }} {\varepsilon(\tau)} \left( 1- \frac {(\xi_2-\xi_1) e^{u}} {\xi_2} e^{u}\right)\nonumber \\&
           \sim \frac {e^{({\frac 1 a}-1)u }} {\varepsilon(\tau)} \left(  \xi_2-(\xi_2-\xi_1) e^{u}\right)
\end{align*}
where  we used that $e^{(\tau-\tau_0)}=\frac {\xi_2}{\xi_2-\xi_1}$.
 Then, \eqref{eq:approxsigmastar3} becomes 
\begin{align}\label{eq:approxsigmastar4}
    \sigma_{\ast}(\xi_1,\xi_2 , \tau)     \sim  \exp \left(- 
    \int_{0}^{\infty} du   \, \Lambda \left( 
    \frac {e^{({\frac 1 a}-1)u }} {\varepsilon(\tau)} \left(  \xi_2-(\xi_2-\xi_1) e^{u}\right)
    \right) \right)\,. 
\end{align}
We expect the term $(\xi_2-\xi_1) e^{u}$ to be small, because there is a term yielding the Dirac delta function $ (\xi_2-\xi_1) e^{u}$. Given that the width of the region where the function approximating the Dirac is supported 
is of order $e^{-b\tau_0}$ with $b=\frac a {1-a}>0$ (cf.\eqref{eq:sizeDiracApprox})  then $ (\xi_2-\xi_1) e^{u}$ would become relevant for values $e^u\sim  e^{b\tau_0}$ but, due to the term $e^{({\frac 1 a}-1)u }$ and the fast decay of $ \Lambda(z)$ as $z\to \infty$  we expect the contribution of this term to be small as $\tau_0\to \infty$. 

Therefore, in \eqref{eq:solFchar_s7_3} we can approximate  the function the function $\sigma_{\ast}(\xi_1,\xi_2 , \tau)$, using also that $\varepsilon(\tau)$ is a slowly varying function and the main contribution to the integral is due to the region where $u=O(1)$, whence $\varepsilon(\tau-u)\sim \varepsilon(\tau)$,  as 
\begin{align}\label{eq:approxsigmastar5}
    \sigma_{\ast}(\xi_1,\xi_2 , \tau)\sim \exp \left(-
    \int_{0}^{\tau-\tau_0} du   \,\Lambda  \left( 
    \frac {e^{({\frac 1 a}-1)u }} {\varepsilon(\tau)}   \xi_2
    \right) \right)\,. 
\end{align}
Since the r.h.s. of \eqref{eq:approxsigmastar5} depends only on the second component of the variable $\xi_2$ with a slight abuse of notation we will replace from now on the function $ \sigma_{\ast}(\xi_1,\xi_2 , \tau)$ by a new function $ \sigma_{\ast}(\xi_2 , \tau)$. 
We now observe that from \eqref{eq:approxsigmastar5} we recover the function $\Sigma(\xi_1,\tau)$ given as in \eqref{eq:sigma2}, with the approximation of the characteristics  \eqref{eq:approx_xi1}.

\smallskip 

\subsection{Asymptotics for the auxiliary functions $\Omega(\zeta)$ (cf.~\eqref{def:Omegazeta}), $W(r)$ (cf.~\eqref{def:Wr}) and $\widetilde{W}(s)>0$ (cf.~\eqref{def:tildeW_F2})} \label{ssec:asymp_auxiliary}

We can prove the following results. 

\begin{lem}[Asymptotics of $\Omega(\zeta)$ for large values of $\zeta$] \label{lem:asymOmegainfty}
Let $\Omega(\zeta)$ be defined as in \eqref{def:Omegazeta}. The following asymptotic formula 
 holds when $\zeta\to \infty$: 
 \begin{equation}
\label{eq:Omegazeta_split2}
    \Omega(\zeta)=
      \frac{2}{\zeta} \left[\log (\zeta)+C_2\right] +
O\left(\frac{\log\zeta}{\zeta^{1+\omega}}\right)   \quad \text{as} \quad \zeta\to \infty\, ,
\end{equation}
 for some $\omega>0$ and where the constant $C_2\in \mathbb{R}$. 
\end{lem}

\begin{lem}[Asymptotics of $\Omega(\zeta)$ for small values of $\zeta$]
\label{lem:asymOmega0}
Let $\beta\in \mathbb{R}$ be given as in \eqref{eq:thmasym} (see Theorem \ref{thm:sigma}), i.e.\begin{equation}\label{eq:thmasym_2}
\lim_{y\to 0^+} \frac{\log(\sigma(y))}{\log(y)}=\beta  \quad \text{with} \ \ \beta \geq a\,.
\end{equation} 
Let $\Omega(\zeta)$ be defined as in \eqref{def:Omegazeta}. The following asymptotic formulas hold when $\zeta\to 0$. 
\begin{itemize}
    \item When $\beta<\tfrac 1 2$ we have 
\begin{equation}\label{Omega_b<12}
        \lim_{\zeta \to 0}\frac{\log\left(\Omega(\zeta)\right)}{\log(\zeta)}= 2\beta-1\,. 
    \end{equation}
\item When $\beta>\tfrac 1 2$ we have 
\begin{equation}\label{Omega_b>12}
    \lim_{\zeta \to 0} \Omega(\zeta)=c_2
\end{equation}
     where the constant $c_2$ is given by 
     \begin{equation}\label{def:c_2}
    c_2:=2\int_{0}^{\infty}\frac{(\sigma(\xi))^2}{\xi^2}d\xi
    < \infty\,.
    \end{equation}
\end{itemize}

\end{lem}
\begin{remark}
Notice that the limit case $\beta =\frac 1 2 $ gives an additional logarithmic divergence that we will not take into account. 
\end{remark}
\smallskip

\noindent 
We first prove Lemma \ref{lem:asymOmegainfty}.

\begin{proofof}[Proof of Lemma \ref{lem:asymOmegainfty}] 
We now compute the asymptotics of $\Omega(\zeta)$ as $\zeta \to \infty$. Observe that if $\zeta\to \infty$ we have  $\sigma\big(\zeta(1+y)\big) \to 1 $  uniformly in $y\in[0,\infty)$. More precisely,
$
\sigma(s)=1+O\big(\frac 1 {s^{\omega}}\big),$ for $s\to\infty,\ \omega>0$. 
It then follows that we can rewrite \eqref{def:Omegazeta} as 
\begin{align}\label{eq:Omegazeta_split}
\Omega(\zeta)
&=\frac{2}{\zeta}\int_{0}^{\infty}\frac{\sigma\big(\zeta(y-1)\big)}{|y-1||1+y|}dy
+O\Big(\frac{1}{\zeta^{1+\omega}}\Big)\int_{0}^{\infty}\frac{\sigma\big(\zeta(y-1)\big)}{|1+y|^{1+\omega}|1-y|}dy =: \frac{2}{\zeta}I(\zeta)+O\Big(\frac{1}{\zeta^{1+\omega}}\Big) J(\zeta)
\end{align}
Thus, the leading asymptotic reduces to analyze 
\begin{align}\label{def:Izeta}
I(\zeta)=&\int_{0}^{\infty}\frac{\sigma\big(\zeta(y-1)\big)}{|y-1|(1+y)}dy= \int_{-1}^{\infty}\frac{\sigma\big(\zeta x\big)}{|x|(2+x)}dx=
\int_{-\zeta}^{\infty}\frac{\sigma\big(y\big)}{|y| \left(2+\frac y {|\zeta|}\right)}dy\, 
\end{align}
where in the first identity we performed the change of variables $x=y-1$ and in the second one we set $y=\zeta x $, with $dy=\zeta dx$. Moreover, we can estimate the integral \eqref{def:Izeta} splitting it as follows:  
\begin{align}\label{eq:Izeta}
I(\zeta)
&=\frac12\int_{-\zeta}^{\zeta}\frac{\sigma(y)}{|y|\left(2+\frac y {\zeta}\right)} dy
+\int_{\zeta}^{\infty}\frac{\sigma(y)}{y\left(2+\frac y {\zeta}\right)} dy
\notag \\& 
= \frac 1 2 \int_{-\zeta}^{\zeta}\frac{\sigma(y)}{|y|} dy
+ \int_{-\zeta}^{\zeta}\frac{\sigma(y)}{|y|} \left( \frac 1 {2+\frac y {\zeta}} - \frac 1 2  \right) dy +\int_{1}^{\infty}\frac{\sigma(\zeta \xi)}{\xi(2+\xi)}d\xi \notag \\& 
= 
\int_{0}^{\zeta}\frac{\sigma(y)}{y} dy
+ \frac 1 2 \int_{-\zeta}^{\zeta}\frac{\sigma(y)}{|y|}  \frac {y} {{\zeta} \left(2+\frac y {\zeta}\right)}   dy +\int_{1}^{\infty}\frac{1}{\xi(2+\xi)}d\xi+\int_{1}^{\infty}\frac{(\sigma(\zeta \xi)-1)}{\xi(2+\xi)}d\xi
\notag 
\\& 
=\int_{0}^{1}\frac{\sigma(y)}{y} dy + \int_{1}^{\zeta}\frac{1}{y} dy+ \int_{1}^{\zeta}\frac{(\sigma(y)-1)}{y} dy \notag \\&
\quad + \frac 1 {2{\zeta} } \int_{-\zeta}^{\zeta}\frac{\sigma(y) \, \text{sgn}(y)} {\left(2+\frac y {\zeta}\right)}   dy +\int_{1}^{\infty}\frac{1}{\xi(2+\xi)}d\xi+\int_{1}^{\infty}\frac{(\sigma(\zeta \xi)-1)}{\xi(2+\xi)}d\xi
\notag \\& 
= \int_{0}^{1}\frac{\sigma(y)}{y} dy + \int_{1}^{\zeta}\frac{1}{y} dy+ \int_{1}^{\zeta}\frac{(\sigma(y)-1)}{y} dy \notag \\&
\quad 
+\frac 1 2 \int_{0}^{\zeta} \left[\frac{\sigma(y)} {\left(2+\frac y {\zeta}\right)} - \frac{\sigma(y)} {\left(2-\frac y {\zeta}\right)}    \right] dy
+\int_{1}^{\infty}\frac{1}{\xi(2+\xi)}d\xi+\int_{1}^{\infty}\frac{(\sigma(\zeta \xi)-1)}{\xi(2+\xi)}d\xi \notag \\&
=:I_1+I_2+I_3+I_4+I_5+I_6\, ,
\end{align}
where in the second identity we changed variables in the third integral setting $y=\zeta \xi $. 
Considering the different contributions $I_i$ on the right hand side of \eqref{eq:Izeta} we have 
\begin{equation}
    \label{eq:I15}
 I_1+I_5= \int_{0}^{1}\frac{\sigma(y)}{y} dy +\int_{1}^{\infty}\frac{1}{\xi(2+\xi)}d\xi=:C_0<\infty   
\end{equation}
\begin{equation}
    \label{eq:I2}
I_2=\log(\zeta)
\end{equation}
Regarding $I_3$, we notice that \begin{equation}\label{eq:I3}
I_3=
\int_{1}^{\infty}\frac{(\sigma(y)-1)}{y} dy-\int_{\zeta}^{\infty}\frac{(\sigma(y)-1)}{y} dy = \tilde{C} + O\left(\frac 1 {\zeta^\omega}\right) \quad \text{as} \quad  \zeta \to \infty
\end{equation}
where we used that  $\vert \sigma(y)-1\vert \leq \frac {C}{y^{\omega}}$ for $y\geq 1$. We also  set $C_1:=C_0+\tilde{C}$ with $C_0$ as in \eqref{eq:I15} and $\tilde{C}$   as in \eqref{eq:I3}.   
Furthermore, using the former estimate for $\sigma(y)$ also to bound $I_6$, it follows that 
\begin{equation}\label{eq:I6}
    I_6=\int_{1}^{\infty}\frac{(\sigma(\zeta \xi)-1)}{\xi(2+\xi)}d\xi= O\left(\frac 1 {\zeta^\omega}\right) \int_{1}^{\infty}\frac{1}{\xi^{\omega+1}(2+\xi)}d\xi, \quad \text{as} \ \ \zeta \to \infty\,. 
\end{equation}
Moreover, 
\begin{align}\label{eq:I4}
I_4&=\frac 1 2 \int_{0}^{\zeta} \left[\frac{\sigma(y)} {\left(2+\frac y {\zeta}\right)} - \frac{\sigma(y)} {\left(2-\frac y {\zeta}\right)}    \right] dy= -\frac{2}{\zeta^2} \int_{0}^{\zeta} \sigma(y)  \frac{y} {\left(2+\frac y {\zeta}\right) \left(2-\frac y {\zeta}\right)}  dy\notag \\& 
= -2 \int_{0}^{1}  \frac{\xi\, \sigma(\zeta \xi) } {\left(2+\xi\right) \left(2-\xi\right)}  d\xi 
= -2 \int_{0}^{1}  \frac{\xi } {\left(2+\xi\right) \left(2-\xi\right)}  d\xi -2 \int_{0}^{1}  \frac{\xi\, \left(\sigma(\zeta \xi)-1\right) } {\left(2+\xi\right) \left(2-\xi\right)}  d\xi 
\end{align} 
We now denote $C_2:= C_1-2 \int_{0}^{1}  \frac{\xi } {\left(2+\xi\right) \left(2-\xi\right)}  d\xi <\infty$. Moreover, the remaining term yields
$$-2 \int_{0}^{1}  \frac{\xi\, \left(\sigma(\zeta \xi)-1\right) } {\left(2+\xi\right) \left(2-\xi\right)}  d\xi = O\left(\frac 1 {\zeta^\omega}\right) \quad \text{as} \quad \zeta\to \infty\, . $$
Combining \eqref{eq:I15}-\eqref{eq:I4} with \eqref{eq:Izeta}, we thus obtain
\begin{align}\label{eq:Izeta2}
I(\zeta)= \log(\zeta) + C_2+ O\left(\frac 1 {\zeta^\omega}\right)\, \quad \text{as} \quad \zeta\to \infty\,,
\end{align}
with $\omega>0$. 
Similar arguments can be used to analyze the contribution of $J(\zeta)$ on the right hand side of \eqref{eq:Omegazeta_split} and arrive at
\begin{equation}
    \label{eq:Jzeta2}
J(\zeta)=\int_{0}^{\infty}\frac{\sigma\big(\zeta(y-1)\big)}{|1+y|^{1+\omega}|1-y|}dy =O\left(\log\zeta\right)   \quad \text{as} \quad \zeta\to \infty\, .
\end{equation}
Combining \eqref{eq:Izeta2} and \eqref{eq:Jzeta2} with \eqref{eq:Omegazeta_split} we finally obtain
\begin{equation}
\label{eq:Omegazeta_split2bis}
    \Omega(\zeta)=
    \frac{2}{\zeta}I(\zeta)+O\Big(\frac{1}{\zeta^{1+\omega}}\Big) J(\zeta)= \frac{2}{\zeta} \left[\log (\zeta)+C_2\right] +
O\left(\frac{\log\zeta}{\zeta^{1+\omega}}\right)   \quad \text{as} \quad \zeta\to \infty\, 
\end{equation}
and the Lemma follows. 

\end{proofof}
\medskip

\noindent We now prove Lemma \ref{lem:asymOmega0}. 

\begin{proofof}[Proof of Lemma \ref{lem:asymOmega0}]
We now consider the   asymptotics of $\Omega(\zeta)$ as $\zeta \to 0^+$ where $\Omega(\zeta)$ is
given as in \eqref{def:Omegazeta}, i.e., 
\begin{align}\label{def:Omegazetaproof}
\Omega(\zeta)=\frac{2}{\zeta}\int_{0}^{\infty}dy \frac{\sigma\big(\zeta(1+y)\big) \sigma\big(\zeta(y-1)\big)}{(1+y)\,|1-y|}.
\end{align}
We first recall the asymptotic  behavior of $\sigma(s)$ for small values of $s$ given as in \eqref{eq:thmasym_2}, namely $$
\lim_{s\to 0^+} \frac{\log(\sigma(s))}{\log(s)}=\beta  \quad \text{with} \ \ \beta \geq a\,.$$
This implies that given $\beta_1, \beta_2$, arbitrarily close to $\beta$ and satisfying $\beta_1>\beta> \beta_2$ 
the following inequalities hold 
\begin{align}\label{eq:b_asymp}
s^{\beta_1}\leq \sigma(s)\leq    s^{\beta_2} \ \
\quad\text{for }0\leq s \leq \delta
\,, \end{align} 
where $\delta=\delta(\beta_1,\beta_2)>0$ can be chosen arbitrarily small. 

We distinguish the two cases $\beta<\frac 1 2 $ and $\beta>\frac{1}{2}$. We first consider the case $\beta < \frac 1 2$.   We first split the integral on the r.h.s. of \eqref{def:Omegazetaproof} as follows 
\begin{align}\label{eq:asymOmegabsmall12split}
    \Omega(\zeta)= \frac{2}{\zeta} \int_{0}^{\frac{\delta}{\zeta}-1}dy \frac{\sigma\big(\zeta(1+y)\big) \sigma\big(\zeta(y-1)\big)}{(1+y)\,|1-y|} + \frac{2}{\zeta} \int_{\frac{\delta}{\zeta}-1}^{\infty}dy \frac{\sigma\big(\zeta(1+y)\big) \sigma\big(\zeta(y-1)\big)}{(1+y)\,|1-y|}:=\Omega_1(\zeta)+\Omega_2(\zeta)
\end{align}
where  $\delta>0$ is fixed, arbitrary small and independent on $\zeta$. 

We now set 
\begin{equation}
    \label{def:barOmega1}
\overline{\Omega}_1(\zeta, \theta):=\frac{2 \zeta^{2\theta}}{\zeta}\int_{0}^{\frac{\delta}{\zeta}-1} 
    \frac{(1+y)^{\theta}|y-1|^{\theta}}{(1+y)|1-y|} dy ={2 \zeta^{2\theta-1}}\int_{0}^{\frac{\delta}{\zeta}-1} 
    \frac{1}{(1+y)^{1-\theta}|1-y|^{1-\theta}} dy \qquad \theta>0 \,.
    \end{equation}
We can then estimate the first term of the r.h.s. of \eqref{eq:asymOmegabsmall12split}, i.e. $ \Omega_1(\zeta)$ as follows
\begin{align}\label{est:barOmega1}
  \overline{\Omega}_1(\zeta, \beta_1) \leq \Omega_1(\zeta) \leq 
   \overline{\Omega}_1(\zeta, \beta_2) 
   \end{align}
On the other hand, using the boundedness of $\sigma$, i.e. $\sigma(s)<1$ we obtain 
\begin{equation}
    \label{est:barOmega2}
0\leq \Omega_2(\zeta)\leq K_\delta \quad \text{as} \ \ \zeta \to 0
\end{equation}
where $K_\delta=\frac 3 \delta$ since we can estimate from above by $\frac 2 {\zeta} \int_{\frac{\delta}{\zeta}-1}^{\infty} \frac{1}{(1+y)|1-y| }dy \leq  \frac 3 \delta$ the integral defining $\Omega_2$. 
Combining \eqref{eq:asymOmegabsmall12split} with \eqref{est:barOmega1} and \eqref{est:barOmega2} we then obtain
$$2 \beta_2-1 \leq \frac {\log(\Omega(\zeta))}{\log\zeta} \leq 2 \beta_1-1 $$
for $\zeta$ small enough. Here we used that the limit of the integral on the r.h.s. of \eqref{def:barOmega1} , i.e. $\int_{0}^{\frac{\delta}{\zeta}-1} 
    \frac{1}{(1+y)^{1-\theta}|1-y|^{1-\theta}} dy$, converges to the constant 
\begin{equation}\label{def:c_1}
    c_1:=\int_{0}^{\infty}\frac{1}{(1+y)^{1-\theta}|1-y|^{1-\theta}}dy< \infty\,
    \end{equation}
as $\zeta\to 0$ if $\theta<\frac 1 2$. 
Taking the limit as $\zeta \to 0$ and $\vert \beta_1-\beta_2 \vert \to 0$ we then obtain 
\eqref{Omega_b<12} which proves the first item of the Lemma.

\medskip

We now consider the second case,  when $ \beta>\tfrac 1 2$. Performing the  change of variables  $\xi=\zeta y$ in \eqref{def:Omegazetaproof} we obtain 
    \begin{align}
\Omega(\zeta)=2\int_{0}^{\infty} \frac{\sigma(\zeta+\xi)\sigma(\xi-\zeta)}{|\xi+\zeta||\xi-\zeta|} d\xi\,.
    \end{align}
We now  use the fact that  
 \begin{equation}
\label{est:sigma2beta12} 
  \sigma (\xi)\leq \xi^{\beta-\varepsilon}, \quad \text{for} \ \ 0< \xi \leq \delta(\varepsilon)
 \end{equation}
 for a fixed positive arbitrary small $\varepsilon>0$. This estimate, combined with the boundedness of $\sigma(\xi)$ ($\sigma(\xi)\leq 1$) for large values gives 
    \begin{align}
    \Omega(\zeta)\to c_2
    \end{align}
    where $c_2< \infty$ is given as in \eqref{def:c_2}. 
    Notice that the integral defining $c_2$ converges due to the power-law estimate \eqref{est:sigma2beta12} provided $\beta>1/2$.   
    Thus $\Omega(\zeta)$ approaches a finite limit when $\beta>1/2$ (cf.~\eqref{Omega_b>12}). We have then proved the two items of the Lemma.

\end{proofof}

\bigskip 

We now consider the auxiliary function  $ \widetilde{W}(s)$ introduced in the previous section, 
defined as in \eqref{def:tildeW_F}, namely 
\begin{equation}
    \label{def:tildeW_F2}
    \widetilde{W}(s):= \frac{4\sqrt{2}}{2^{a}}W\left(\, s ^{{-\frac 1 a }}\right)   s ^{-\left(\frac 1 a+1\right)}
\end{equation}
where $W(r)$ (cf. \eqref{def:Wr}) reads 
\begin{align}\label{def:Wr2}
W(r):=\int_{r}^{\infty}\frac{d\zeta}{\zeta^{a+1}}\frac{\sqrt{1 + \sqrt{1 - \left(\frac{r}{\zeta} \right)^{2}}} + \sqrt{1 - \sqrt{1 - \left(\frac{r}{\zeta}\right)^{2}}}}{\sqrt{1 - \left(\frac{r}{\zeta}\right)^{2}}}\,\Omega(\zeta)\,. 
\end{align}
We have the following results. 
\begin{lem}\label{lem:Wr}
    Let ${W}(r)$ be defined as in \eqref{def:Wr}. We have the following asymptotics: 
    \begin{itemize}
        \item[(i)] $W(r)=O\left(\frac{\log(r)}{r^{a+1}}\right)$ as $r\to \infty$;
        \item[(ii)] $\displaystyle W(r)=\frac{C_0}{r^{a}} $ as $r\to 0^+$ if $\beta>\frac 1 2$;
        \item[(iii)] $\displaystyle W(r)=\frac{C_1}{r^{1+a-2\beta}} $ as $r\to 0^+$ if $\beta< \frac 1 2$ ,
    \end{itemize}
    where $C_0, C_1$ are two positive constants. 
\end{lem}

\begin{lem}\label{lem:tildeWr}
    Let $\widetilde{W}(s)>0$ be defined as in \eqref{def:tildeW_F2}. Then  $\displaystyle W_0=\int_{0}^\infty \, \widetilde{W}(s)ds<\infty$.  
\end{lem}

\begin{proofof}[Proof of Lemma \ref{lem:Wr}]
We start proving item $(i)$. From \eqref{def:Wr}, performing the change of variables $\zeta=rx$ we obtain 
\begin{align}\label{def:Wr3}
W(r)=\frac 1 {r^a}\int_{1}^{\infty} \, dx \frac{\Omega(rx) }{x^{a+1}}\frac{\sqrt{1 + \sqrt{1 - \left(\frac{1}{x} \right)^{2}}} + \sqrt{1 - \sqrt{1 - \left(\frac{1}{x}\right)^{2}}}}{\sqrt{1 - \left(\frac{1}{x}\right)^{2}}}\,\,. 
\end{align}
From Lemma \ref{lem:asymOmegainfty} (see \eqref{eq:Omegazeta_split2}) we have that $\Omega(rx)=O\left(
      \frac{\log (rx)}{rx} \right)$ 
      as $r\to \infty\, $ 
for $x\geq 1$ which then directly implies item $(i)$.  

We now consider item $(ii)$. Thanks to Lemma \ref{lem:asymOmega0} (see \eqref{Omega_b>12}) we have that  $\Omega(rx)\sim c_2 $ as $r \to 0$ when $\beta>\tfrac 1 2$. Using this asymptotics in \eqref{def:Wr3} we then obtain, as $r\to 0^+$, 
\begin{align}\label{def:Wr4}
W(r)\sim \frac 1 {r^a}\int_{1}^{\infty} \, dx \frac{c_2 }{x^{a+1}}\frac{\sqrt{1 + \sqrt{1 - \left(\frac{1}{x} \right)^{2}}} + \sqrt{1 - \sqrt{1 - \left(\frac{1}{x}\right)^{2}}}}{\sqrt{1 - \left(\frac{1}{x}\right)^{2}}}\, =: \frac{\tilde{c}_2}{r^a}\,
\end{align}
where $\tilde{c}_2$ is a positive constant. This implies item $(ii)$. 

We finally consider item $(iii)$. From Lemma \ref{lem:asymOmega0} (see \eqref{Omega_b<12}) we have that  $\Omega(rx)\sim  c_3\,(rx)^{2\beta-1-\bar{\delta}}\,$ as ${r \to 0}$ with $\bar{\delta}>0$ arbitrary small. Using this asymptotics in \eqref{def:Wr3} we obtain, as $r\to 0^+$, 
\begin{align}\label{def:Wr5}
W(r)\sim \frac 1 {r^{a -2\beta+1+\bar{\delta}}} \int_{1}^{\infty} \, dx \frac{c_3}{x^{a+1-2\beta+1+\bar{\delta}}}\frac{\sqrt{1 + \sqrt{1 - \left(\frac{1}{x} \right)^{2}}} + \sqrt{1 - \sqrt{1 - \left(\frac{1}{x}\right)^{2}}}}{\sqrt{1 - \left(\frac{1}{x}\right)^{2}}}\, = \frac{\tilde{c}_3}{r^{1+a-2\beta+\bar{\delta}}} \,
\end{align}
where $\tilde{c}_3$ is a positive constant. This  implies item $(iii)$ and concludes the proof.  
\end{proofof}

\bigskip

\begin{proofof}[Proof of Lemma \ref{lem:tildeWr}]
We first study the integrability at the origin of 
$$\widetilde{W}(s)=\frac{4\sqrt{2}}{2^{a}}W\left(\, s ^{{-\frac 1 a }}\right)   s ^{-\left(\frac 1 a+1\right)}\,.$$
 We observe that, as $s\to 0^+$, we have that $s ^{{1-\frac 1 a }}\to \infty$. Using the asymptotics of $W(r)$  as $r\to \infty$, given by item $(i)$ of Lemma \ref{lem:Wr}, namely $W(r)=O\left(\frac{\log(r)}{r^{a+1}}\right)$, we thus obtain
$$\widetilde{W}(s)= O\left({\log\left(\frac 1 s \right)}{ \left( s^{-\frac {1} a} \right)^{a+1}}   \frac 1 {s^{\frac 1 a+1}} \right) = O\left({\log\left(\frac 1 s \right)} s^{\frac {1} {a}  +1 } s ^{-\left(\frac 1 a+1\right)} \right)=O\left({\log\left(\frac 1 s \right)} \right) \quad \text{as} \quad s\to 0^+$$
which implies the integrability of $\widetilde{W}(s)$ at $\infty$. 

We now analyze the integrability as $s\to \infty$ of 
$\widetilde{W}(s)$.  We first suppose that $\beta>\frac 1 2$. Then, thanks to item $(ii)$ of Lemma \ref{lem:Wr}, which states that $\displaystyle W(r)=\frac{C_0}{r^{a}} $ as $r\to 0^+$, we obtain
$$\widetilde{W}(s)=\frac{4\sqrt{2}}{2^{a}}W\left( s ^{{-\frac 1 a }}\right)   s ^{-\left(\frac 1 a+1\right)}=O\left( \frac s {s ^{ \frac 1 a+1 }}\right)= O\left(\frac 1 {s ^{\frac 1 a}}\right)
\quad \text{as} \quad s\to \infty\, .$$    
This gives the integrability as $s\to \infty$ of 
$\widetilde{W}(s)$ when $\beta>\frac 1 2$. 

On the other hand, when $\beta<\frac 1 2$, thanks to item $(iii)$ of Lemma \ref{lem:Wr} we have that 
$ W(r)=\frac{C_1}{r^{1+a-2\beta}}$ as $r\to 0^+$ and then, as $s\to \infty$ we obtain
$$
\widetilde{W}(s)=\frac{4\sqrt{2}}{2^{a}} W\left( s ^{{-\frac 1 a }}\right)   s ^{-\left(\frac 1 a +1\right)} =O\left(\left(s^{{-\frac 1 a}}\right)^{-(1+a-2\beta)} s ^{-\left(\frac 1 a +1\right)}\right)
=O\left(  s^{ -\frac{2\beta}{a} }\right)\quad \text{as} \quad s\to \infty \,. 
$$
We notice that if $\beta \geq a$ the r.h.s. of the equation above is than controlled by a quantity of order $s^{-2}$, i.e. $O\left(  s^{ -2}\right) $ which then gives integrability of $\widetilde{W}(s)$ as  $s\to \infty $ in this case. This concludes the proof. 
\end{proofof}
\medskip

\subsection{Derivation of the asymptotics of $F$ as $\tau \to \infty$ and test of the consistency of the asymptotics}

With the argument presented in this section we  will show that the asymptotics for the distribution function $F$ and the reduced one $G$ are consistent. 
We start from $G(0^+,\xi_2,\tau)$ given as in \eqref{def:Gbd_selfsim} and combined with \eqref{eq:ansatzH}, namely \begin{equation}
\label{def:Gbd_selfsim_sec7}
G(0^+,\xi_2,\tau)= \frac{e^{-2\tau}}{(\varepsilon(\tau))^2 } \lambda(\tau)\overline{H}\left(\frac {\xi_2}{\varepsilon(\tau)e^{\tau} }\right), \ \ \tau=\log(t) \,.  \end{equation}
On the other hand, we recall the formula for $F(0^+,\tilde{\xi},\tau)$ (cf.\eqref{eq:Fbdexp7}), namely 
\begin{align}\label{eq:Fbdexp7_sec7}
F(0^{+}, \xi_{2}, \xi_{3}, \tau) = \frac{ 4 \sqrt{2}}{2^{a}}\frac {e^{-\tau}} {\left(\xi_{2} |\tilde{\xi}|\right)} \frac{a}{(1-a)} \frac{\lambda(\tau)}{\varepsilon(\tau)} W\left(\frac{|\widetilde{\xi}|}{e^{\tau}\,\varepsilon(\tau)}\right) \, .
\end{align} 
where the function $W$ is as in \eqref{def:Wr}. We now integrate with respect to the variable $\xi_3$ and then obtain
\begin{align}\label{eq:Fbint1_sec7}
\int_{\mathbb{R}}d\xi_3\, F(0^{+}, \xi_{2}, \xi_{3}, \tau) &= \frac{4\sqrt{2}}{2^{a}} \int_{\mathbb{R}}d\xi_3\, \frac {e^{-\tau}} {\left(\xi_{2} |\tilde{\xi}|\right)} \frac{a}{(1-a)} \frac{\lambda(\tau)}{\varepsilon(\tau)} W\left(\frac{|\widetilde{\xi}|}{e^{\tau}\,\varepsilon(\tau)}\right)\nonumber \\&
=\frac{8\sqrt{2}}{2^{a}} \frac{a}{(1-a)}   \frac {\lambda(\tau)}{\xi_2}\int_{\frac{\xi_2}{e^{\tau}\,\varepsilon(\tau)}}^\infty ds\, \frac {W\left(s\right)}{\xi_3} \nonumber
\\& 
=\frac{8\sqrt{2}}{2^{a}} \frac{a}{(1-a)}   \frac {\lambda(\tau)}{\xi_2}\int_{\frac{\xi_2}{e^{\tau}\,\varepsilon(\tau)}}^\infty ds\, \frac {W\left(s\right)}{ 
e^{\tau}\,\varepsilon(\tau) \sqrt{s^2 -\left(\frac{\xi_2}{e^\tau \varepsilon(\tau)}\right)^2}}
\nonumber
\\& = \frac{8\sqrt{2}}{2^{a}} \frac{ e^{-2\tau}}{(\varepsilon(\tau))^2}
\frac{a}{(1-a)}  \frac {\lambda(\tau)}{\zeta } \int_{\zeta}^\infty 
ds\, \frac {W\left(s\right)}{  \sqrt{s^2 - \zeta^2}}
\nonumber
\\& 
= \frac{e^{-2\tau}}{(\varepsilon(\tau))^2} \lambda(\tau)\overline{H}(\zeta)\, \qquad \text{with}\quad \zeta=\frac{\xi_2}{e^\tau (\varepsilon(\tau))}
\end{align} 
where we used the change of variables 
$s=\frac{|\widetilde{\xi}|}{e^{\tau}\,\varepsilon(\tau)}=\frac{\sqrt{{\xi_2}^2+{\xi_3}^2}}{e^{\tau}\,\varepsilon(\tau)}$
which implies 
$\left(e^{\tau}\,\varepsilon(\tau)\right)^2 s^2 -{\xi_2}^2= {\xi_3}^2$ and thus $\xi_3= e^{\tau}\,\varepsilon(\tau) \sqrt{s^2 -\left(\frac{\xi_2}{e^\tau \varepsilon(\tau)}\right)^2}$, with jacobian   $ds= \frac{\xi_3}{|\widetilde{\xi}|e^\tau \varepsilon(\tau)} d\xi_3$. 

Therefore, 
$$
\overline{H}(\zeta) = \frac{8\sqrt{2}}{2^{a}} 
\frac{a}{(1-a)}  \frac 1{\zeta } \int_{\zeta}^\infty 
ds\, \frac {W\left(s\right)}{  \sqrt{s^2 - \zeta^2}} $$

Then, using that $\int_{0}^\infty \zeta \overline{H}(\zeta) d \zeta= 1$ we obtain
\begin{align}
   1&= \int_{0}^\infty d \zeta\, \zeta \,\overline{H}(\zeta) =    \frac{16\sqrt{2}}{2^{a}} 
\frac{a}{(1-a)}  \int_{0}^\infty d \zeta\,  \int_{\zeta}^\infty 
ds\, \frac {W\left(s\right)}{  \sqrt{s^2 - \zeta^2}} 
\nonumber \\&
=\frac{8\sqrt{2}}{2^{a}} 
\frac{a}{(1-a)}  \int_{0}^\infty 
ds\,W\left(s\right) \int_{0}^{s} \frac {d\zeta}{  \sqrt{s^2 - \zeta^2}} = \frac{16\sqrt{2}}{2^{a}} 
\frac{a}{(1-a)}  \left(\int_{0}^\infty 
ds\,W\left(s\right) \right) \left(\int_{0}^{1}  \frac {dt}{  \sqrt{1 - t^2}}\right)
\nonumber\\&
= \frac{8\sqrt{2}}{2^{a}} 
\frac{a}{(1-a)}  \frac \pi 2\left(\int_{0}^\infty 
ds\,W\left(s\right) \right)
\end{align}
where we first used Fubini and then the change of variables $st=\zeta$ and the fact that $\left(\int_{0}^{1}  \frac {dt}{  \sqrt{1 - t^2}}\right)=\frac \pi 2 $.  Therefore, we have 
\begin{equation}
    \label{eq:integralWs}\int_{0}^\infty 
ds\,W\left(s\right)= \frac{2^{a}} {4\sqrt{2}\pi}
\frac{(1-a)} {a} \,. 
\end{equation}
We now want to compute the value $W_0=\int_{0}^\infty ds\, \tilde{W}(s)$ (cf. Lemma \ref{lem:tildeWr}) where $\tilde{W}(s)$  is as in \eqref{def:tildeW_F}, namely  \begin{equation*}
    \widetilde{W}(s):= \frac{4\sqrt{2}}{2^{a}}W\left(\, s ^{{-\frac 1 a }}\right)  s ^{-\left(\frac 1 a+1\right)}.
\end{equation*} To this aim, we consider 
\begin{align}
 W_0&=\int_{0}^\infty ds\, \tilde{W}(s)   = 
 \frac{4\sqrt{2}}{2^{a}} \int_{0}^\infty ds\, W\left(\, s ^{{-\frac 1 a }}\right) s ^{-\left(\frac 1 a+1\right)}= \frac{4\sqrt{2}}{2^{a}} a
 \int_{0}^\infty dz\, W\left(z\right) 
\nonumber\\&
=\frac{4\sqrt{2}}{2^{a}} \frac{2^a}{4\sqrt{2}\pi} \frac{ a(1-a)} {a} =\frac 1 \pi(1-a)\,,
\end{align} 
where we changed variables setting $ z=s ^{-\frac 1 a }$ with $dz=-\frac 1 a \, s^{-\left(1+\frac 1 a \right)} ds $ and we used \eqref{eq:integralWs} in the last identity. 
We thus obtained that 
$W_0=\frac 1 \pi (1-a)$. 

Then, \eqref{eq:solFchar_s7_5} becomes 
\begin{align}
\label{eq:solFchar_s7_5final}
F(\xi_1,\xi_2,\xi_3,\tau)  &\sim 
\frac a {1-a} \,\lambda(\tau_0) 
\frac {W_0\,} { \vert \tilde{\xi}\vert ^2}  \delta\left(\xi_2-\xi_1  \right)\sigma_{\ast}\left( \xi_1,\xi_2 ,\tau\right)
\sim 
\frac a {1-a} \, 
\frac {1-a} { \pi\,\vert \tilde{\xi}\vert ^2}  \lambda(\tau_0)\,\delta\left(\xi_2-\xi_1  \right)\sigma_{\ast}\left( \xi_1,\xi_2 ,\tau\right)
\nonumber \\& 
\sim \lambda(\tau_0) 
\frac {a} { \pi\,\vert \tilde{\xi}\vert ^2}  \delta\left(\xi_2-\xi_1  \right)
\sigma_{\ast}\left( \frac{\xi_1}{\varepsilon(\tau)} ,\tau\right)
\end{align}
where in the second approximate identity we used the fact that 
$$\sigma_{\ast}\left( \xi_1,\xi_2 ,\tau\right)=\sigma_{\ast}\left( \frac{\xi_1}{\varepsilon(\tau)} ,\tau\right):=\chi_{\{\xi_1 \le e^\tau \varepsilon(\tau) \}} {\sigma} \left( \frac{\xi_1}{\varepsilon(\tau)} \right) $$
where $\sigma$ is the one appearing in the asymptotics of $G$, specifically \eqref{eq:sigma2}, \eqref{eq:chvarSigma}. 

On the other hand,  the asymptotics for the reduced distribution $G$ are given by \eqref{eq:finalGappr_sec5}, namely 
\begin{equation*}
G(\xi, \tau) = 
\frac 1 {2\,\log \left(\frac{1}{1-a}\right)}\, \frac{1 }{ ((1-a)\tau + a \log(\xi_1)) \, \xi_1} \chi_{\{\xi_1 \le e^\tau \varepsilon(\tau) \}} {\sigma} \left( \frac{\xi_1}{\varepsilon(\tau)} \right) \delta(\xi_2 - \xi_1) 
 \quad \text{as} \ \ \tau \to \infty \end{equation*}
or, equivalently,
\begin{equation*}
G(\xi, \tau) = a \lambda(\tau_0)
\chi_{\{\xi_1 \le e^\tau \varepsilon(\tau) \}} 
\frac 1 {\xi_1}{\sigma} \left( \frac{\xi_1}{\varepsilon(\tau)} \right) \delta(\xi_2 - \xi_1) 
 \quad \text{as} \ \ \tau \to \infty \end{equation*}
We can then finally check that $\int_{\mathbb{R}} d\xi_3\, F(\xi_1,\xi_2,\xi_3,\tau)$ has the same asymptotic behavior as $G(\xi_1,\xi_2,\tau)$, as $\tau\to \infty$. 
Indeed, using \eqref{eq:solFchar_s7_5final}, we have 
\begin{align}\label{eq:consistency_asymFG}
 \int_{\mathbb{R}} d\xi_3\, F(\xi_1,\xi_2,\xi_3,\tau)&\sim  
\frac {a \,\lambda(\tau_0) } { \pi} 
\int_{\mathbb{R}} \, \frac {d\xi_3}{( \xi_2^2+\xi_3^2)}  \delta\left(\xi_2-\xi_1  \right)
\sigma_{\ast}\left( \frac{\xi_1}{\varepsilon(\tau)} ,\tau\right)\nonumber\\&
\sim 
\frac{ a\lambda(\tau_0) }{ \pi} 
\frac 1 {\vert \xi_2 \vert }\int_{\mathbb{R}} \, \frac {dt}{( 1+t^2)}  \delta\left(\xi_2-\xi_1  \right)
\sigma_{\ast}\left( \frac{\xi_1}{\varepsilon(\tau)} ,\tau\right)\nonumber \\&= 
\frac{a\lambda(\tau_0) }  { \pi} 
\frac {\pi} {\vert \xi_2 \vert }  \delta\left(\xi_2-\xi_1  \right)
\sigma_{\ast}\left( \frac{\xi_1}{\varepsilon(\tau)} ,\tau\right) 
\end{align}
where we changed variables and set $t\xi_2=\xi_3$ and the fact that $\int_{\mathbb{R}} \, \frac {dt}{( 1+t^2)} =\pi$. 
Notice that the r.h.s. of \eqref{eq:consistency_asymFG} 
provides the same asymptotic behavior as the function $G(\xi_1,\xi_2,\tau)$. Therefore, the formulas obtained are consistent. Actually, it is possible to show that, by means of a direct computation, $\int_{\mathbb{R}} d\xi_3\, F(\xi_1,\xi_2,\xi_3,\tau)$ with $F(\xi_1,\xi_2,\xi_3,\tau)$ given as in \eqref{eq:solFchar_s7_3} coincides with the function $G(\xi_1,\xi_2,\tau)$ given as in \eqref{eq:GapprU}.


\bigskip

\section{On the removal mechanism of particles in the region where $|\xi|\approx \varepsilon(\tau)$ } 
\label{ss:cutoff}

In this section we study the stationary problem \eqref{eq:pde_b_stat}, \eqref{eq:statLambda} for the cutoff function $\sigma$, that has been introduced in Section \ref{sec:asymG}. This problem  describes how the particles are removed from the collision region, specifically how they are transported to the region where $|\xi|\gtrsim\varepsilon(\tau) e^\tau$.

\subsection{Main results and reformulation of the stationary problem \eqref{eq:pde_b_stat}, \eqref{eq:statLambda}} \label{ss:refomulationSteadyPb}

The goal of this section is to study the damping problem \eqref{eq:pde_b_stat}, \eqref{eq:statLambda} defined by the following equation:
\begin{equation}\label{def:dampingPb}
\left(\frac{1}{a}-1\right)y\frac{\partial \sigma}{\partial y}={\Lambda}[\sigma] \sigma, \quad y>0
\end{equation}
with boundary condition 
\begin{equation}\label{eq:bdcdtsigma}
\sigma(y)\to 1 \quad \text{as} \quad y\to \infty
\end{equation}
and where the operator ${\Lambda}[\sigma]$ is given by:
\begin{equation}\label{eq:statLambda_bis}
{\Lambda}[\sigma](y)=8\pi\int_{\mathbb{R}} d\zeta \, |y-\zeta|^{-a}\frac{\sigma(\zeta)}{|\zeta|}
\end{equation}
with $\sigma(\zeta)=\sigma(-\zeta)$. 
Notice that, in this Section, to keep the notation lighter we are renaming the variables $y_2 \to y$,  $\zeta_2 \to \zeta$; on the other hand we are emphasizing the dependence of $\Lambda$ on $\sigma$, writing $\Lambda=\Lambda[\sigma]$.
We will now prove the following result. 
\begin{thm}\label{thm:sigma}
    Let be $X:=\{\sigma\in C_b(\mathbb{R})\, : \,\sigma(y)=\sigma(-y), \, 0 \le \sigma(y) \le 1\}\}$. There exists at least one solution $\sigma\in X$ of the problem \eqref{def:dampingPb}-\eqref{eq:statLambda_bis}. 
Moreover, every non-negative solution to the problem \eqref{def:dampingPb}-\eqref{eq:statLambda_bis} satisfies $0<\sigma(y)\leq 1$ for $y\in \mathbb{R}$ and the asymptotics 
\begin{equation}\label{eq:thmasym}
\lim_{y\to 0^+} \frac{\log(\sigma(y))}{\log(y)}=\beta  \quad \text{with} \ \ \beta \geq a\,.
\end{equation}
\end{thm}

Before proving  Theorem \ref{thm:sigma} we will first  explain the main ideas that we will be used for the proof.  
We first observe that we can equivalently rewrite the problem \eqref{def:dampingPb} as 
\begin{equation}\label{def:dampingPb2}
y\frac{\partial \sigma}{\partial y}=\widetilde{\Lambda}[\sigma]\sigma, \quad y>0
\end{equation}
with 
\begin{equation}\label{def:tildeLambda}
\widetilde{\Lambda}[\sigma](y)=\frac{8\pi a}{1-a}\int_{\mathbb{R}} d\zeta  \, |y-\zeta|^{-a}\frac{\sigma(\zeta)}{|\zeta|}\,
\end{equation}
and the boundary condition \eqref{eq:bdcdtsigma}. 
Notice that, due to the symmetry of $\sigma$ it is enough to determine $\sigma(y)$ for $y>0$. 

We further notice that  
using the symmetry of $\sigma(\zeta)$, i.e., $\sigma(\zeta)=\sigma(-\zeta)$, we can write
\begin{equation}\label{def:tildeLambda2}
\widetilde{\Lambda}[\sigma](y)=\frac{8\pi a}{1-a}\int_{0}^{\infty}\frac{d\zeta}{\zeta }\left(\frac{1}{|y-\zeta|^{a}}+\frac{1}{|y+\zeta|^{a}}\right)\sigma(\zeta) , \quad y>0\,.
\end{equation}
and then, using \eqref{eq:bdcdtsigma} and \eqref{def:dampingPb2}  we have that 
\begin{equation}\label{eq:solsigmasteady}
\sigma(y)=\exp\left(-\int_{y}^{\infty}{\widetilde{\Lambda}}[\sigma](u)\frac{du}{u}\right),\ \ y>0
\end{equation}
where 
\begin{align}\label{eq:integralLambda}
\int_{y}^{\infty}{\widetilde{\Lambda}}[\sigma](u)\frac{du}{u} =  \frac{8\pi a}{1-a} \int_{0}^{\infty}\frac{d\zeta}{\zeta}\sigma(\zeta)\int_{y}^{\infty} \left(\frac{1}{|u-\zeta|^{a}}+\frac{1}{|u+\zeta|^{a}}\right) \frac{du}{u}   
\end{align}

It is convenient to reformulate the problem using a more convenient set of variables. 
In order to do this, we set \begin{equation}\label{eq:chexpvar}
\varphi=\log(\sigma), \quad  y=e^\xi,\quad  \zeta=e^z,\quad  u=e^\eta    
\end{equation} 
and thus \eqref{eq:solsigmasteady} becomes
\begin{equation}\label{eq:solsigmasteady2}
\varphi(\xi)= -\int_{-\infty}^{\infty} dz e^{\varphi(z)}\left(\int_{\xi}^{\infty}{\widetilde{\Lambda}}[\sigma](u)\frac{du}{u} \right)= 
-\frac{8\pi a}{1-a}\int_{-\infty}^{\infty} dz \, e^{\varphi(z)}\left(\int_{\xi}^{\infty} 
\left(\frac{1}{|e^\eta-e^z|^{a}}+\frac{1}{|e^\eta+e^z|^{a}}\right) {d\eta}\right)
\end{equation}
Notice that if the function $\varphi$ is bounded for values $z \geq C$ with $C>0$ we have that $\varphi(\xi)\to 0 $ as $\xi\to \infty$ and $\sigma(y)\to 1$ (cf.\eqref{eq:bdcdtsigma}).  

We can reformulate \eqref{eq:solsigmasteady2} as the fixed point problem
\begin{align}\label{eq:fixedpointPhi}
   & \varphi(\xi)= T[\varphi](\xi), \quad \text{with} \\&
 T[\varphi](\xi):= 
-\frac{8\pi a}{1-a}\int_{-\infty}^{\infty} dz \, e^{\varphi(z)}\left(\int_{\xi}^{\infty} 
\left(\frac{1}{|e^\eta-e^z|^{a}}+\frac{1}{|e^\eta+e^z|^{a}}\right) {d\eta}\right)  \label{def:fixedpointTPhi}
\end{align}
We remark that this problem is equivalent to the problem \eqref{def:dampingPb2}, \eqref{def:tildeLambda} with boundary condition \eqref{eq:bdcdtsigma}. 

In order to prove Theorem \ref{thm:sigma} we will first prove the existence of a solution to \eqref{eq:fixedpointPhi}, \eqref{def:fixedpointTPhi}.  
In order to implement this program and  prove the existence of a solution to \eqref{eq:fixedpointPhi}, \eqref{def:fixedpointTPhi} we will introduce a regularized problem, namely 
$$\varphi(\xi)= T_{\varepsilon}[\varphi](\xi) $$
where $T_{\varepsilon}$ represents a suitable regularization of the operator $T$ and, 
using a Schauder fixed point theorem we will then prove the existence of a solution for the regularized problem. Thus, in order to be able to take the limit as the cutoff parameter $\varepsilon \to 0$, we will derive suitable uniform estimates (uniform in $\varepsilon$) that will then allow to pass to the limit and recover the existence of a solution $\varphi(\xi)$ or, equivalently of a solution $\sigma(y)$ to the original problem.  
We will later prove the asymptotics for the solution $\sigma(y)$.

The main idea to obtain these uniform estimate will be to find an a priori estimate for a solution of the problem \eqref{eq:fixedpointPhi},\eqref{def:fixedpointTPhi}  that we will present now.  

\subsubsection{An a priori estimate for a solution of the problem \eqref{eq:fixedpointPhi}, \eqref{def:fixedpointTPhi}} \label{ssec:estimatephi}

We will now provide a heuristic argument that gives a clue on the a priori estimate 
for the solution to \eqref{eq:fixedpointPhi}, \eqref{def:fixedpointTPhi} that will be used in the proof. 
Suppose that there exists a solution to \eqref{eq:fixedpointPhi},~\eqref{def:fixedpointTPhi}. We want to prove that  
\begin{equation}
    \label{est:phi}
\varphi(\xi) \leq K \xi, \quad \xi \leq 0
\end{equation} 
and for some suitable constant $K>0$. 
Notice that the solutions to \eqref{eq:fixedpointPhi},~\eqref{def:fixedpointTPhi} are negative, i.e. $\varphi(\xi)<0$ for any $\xi\in\mathbb{R}$. Moreover, the function is increasing, 
\begin{equation}\label{eq:derviativephi}
    \varphi'(\xi)>0 \quad \text{for any }\quad \xi\in\mathbb{R}.
\end{equation}
The main idea is the following. Given a solution to \eqref{def:fixedpointTPhi}, the asymptotic behavior of the solution $\varphi(\xi)$ as $\xi \to \infty$  is given by 
\begin{equation}
    \varphi(\xi)\sim - \frac{16 \pi a}{1-a} \xi \left(
\int_{-\infty}^{\infty} dz \, e^{\varphi(z)-za} \right)  \quad \text{as}\quad \xi \to - \infty
\end{equation}
We now observe that if $\int_{-\infty}^{\infty} dz \, e^{\varphi(z)-za} \geq 1$, including $+\infty$, then \eqref{est:phi} holds for large negative values of $\xi$, i.e. $\xi\to -\infty$ (even if the integral is  divergent). On the other hand, if $\int_{-\infty}^{\infty} dz \, e^{\varphi(z)-za} \leq 1$ we have $\varphi(z)-za \to -\infty$ in some integral sense. This will implies $\varphi(x)\leq a z$ in some integral sense and the inequality \eqref{est:phi} will also hold. 

We now prove in a rigorous way the a priori estimate \eqref{est:phi}. 
From \eqref{eq:solsigmasteady2} we have
\begin{align}\label{eq:solsigmasteady3}
-\varphi(\xi)=& 
\frac{8\pi a}{1-a}\int_{-\infty}^{\infty} dz \, e^{\varphi(z)}\left(\int_{\xi}^{\infty}  e^{-az}
\left(\frac{1}{|1-e^{\eta-z}|^{a}}+\frac{1}{|1+e^{\eta-z}|^{a}}\right) {d\eta}\right) \nonumber \\& =
\frac{8\pi a}{1-a}  \int_{\xi}^{\infty}d\eta \int_{-\infty}^{\infty} dz \, e^{\varphi(z)}  e^{-az}
\left(\frac{1}{|1-e^{\eta-z}|^{a}}+\frac{1}{|1+e^{\eta-z}|^{a}}\right)  \nonumber\\&
\geq \frac{8\pi a}{1-a}  \int_{\xi}^{\infty} d\eta \int_{\eta}^{\infty} dz \, e^{\varphi(z)}  e^{-az} 
\end{align}
where we used that for $z \geq \eta$, $0\leq x:=e^{\eta-z}\leq 1$ it holds $\frac{1}{|1-x|^a}+\frac 1 {|1+x|^a} \geq 1 $. 
Therefore, applying Fubini, we arrive at 
\begin{equation}
\label{eq:solsigmasteady4}
-\varphi(\xi) \geq \frac{8\pi a}{1-a}   \int_{\xi}^{\infty} dz \int_{\xi}^{z} d\eta \, e^{\varphi(z)}  e^{-az} = \int_{\xi}^{\infty} dz  \, e^{\varphi(z)}  e^{-az} \, (z -\xi)
\end{equation}
Using the monotonicity of $\varphi$ (cf. \eqref{eq:derviativephi}) we obtain
\begin{equation}
\label{eq:solsigmasteady5}
-\varphi(\xi) \geq e^{\varphi(\xi)} \int_{\xi}^{\infty} dz  \,  e^{-az} \, (z -\xi) = e^{\varphi(\xi)} \int_{0}^{\infty} dt  \,  e^{-a (t+\xi)} \, t= e^{\varphi(\xi)} e^{-a\xi} \int_{0}^{\infty} dt  \,  e^{-at} \, t= \frac 1{a^2} e^{\varphi(\xi)- a\xi} 
\end{equation}
where we used the change of variables $z -\xi=t$. Using that $ |\varphi(\xi)|= - \varphi(\xi)$ we can then rewrite the equation above as 
\begin{equation}\label{est:phiabsvalue}
 \frac {e^{a |\xi|}}{a^2}= \frac {e^{-a \xi}}{a^2}   \leq  |\varphi(\xi)|\, e^{|\varphi(\xi)|} \leq e^{2 |\varphi(\xi)|}\quad \text{for} \ \ \xi \leq 0 
 \end{equation} 
where we used that $xe^x \leq e^{2x}$ for $x >0$. 
This implies
$$ |\varphi(\xi)| \geq \frac a 2 |\xi| + 2 \log\left(\frac 1 a\right) \quad \text{for} \ \ \xi \leq 0$$
or equivalently 
$ -\varphi(\xi) \geq -\frac a 2 \xi + 2 \log\left(\frac 1 a\right)$ which implies $ \varphi(\xi) \leq \frac a 2 \xi - 2 \log\left(\frac 1 a\right)$  whence \eqref{est:phi} follows. 

\medskip

Using the estimate \eqref{est:phi} it is easy to prove that $|\varphi(\xi)| \leq C $ for any $\xi \geq 0$.  This estimate will be derived later, in Section \ref{ssec:regularPb} where we will study the regularized version of the problem. 
\bigskip

\subsection{Proof of Theorem \ref{thm:sigma}: proof of the existence of a solution} \label{ss:existenceb}

The argument presented in Section \ref{ssec:estimatephi} has to be applied to a regularized version of the problem for which it is possible to prove existence of solutions, as well as  the estimate uniformly in the regularization parameter.

\subsubsection{Regularization of the problem \eqref{eq:fixedpointPhi}, \eqref{def:fixedpointTPhi}}
\label{ssec:regularPb} 

We consider a parameter $\varepsilon>0$ and  the following regularization of the fixed point problem \eqref{eq:fixedpointPhi}, \eqref{def:fixedpointTPhi}:
\begin{align}\label{eq:fixedpointPhi_reg}
   & \varphi(\xi)= T_{\varepsilon}[\varphi](\xi), \quad \text{with} \\&
 T_{\varepsilon}[\varphi](\xi):= 
-\frac{8\pi a}{1-a}\int_{-\infty}^{\infty} dz \, \frac{e^{\varphi(z)}}{1+\varepsilon e^{-z}}\left(\int_{\xi}^{\infty} 
\left(\frac{1}{|e^\eta-e^z|^{a}}+\frac{1}{|e^\eta+e^z|^{a}}\right) {d\eta}\right)  \label{def:fixedpointTPhi_reg}
\end{align}
We observe that this regularized problem has been obtained from \eqref{def:tildeLambda2}, \eqref{eq:solsigmasteady} replacing $\frac{d\zeta}{\zeta }$ by $\frac{d\zeta}{\zeta + \varepsilon}$ with $\varepsilon>0$ and using later the exponential change of variables \eqref{eq:chexpvar}.

We have the following Lemma.
\begin{lem}\label{lem:existencesol_reg}
Let be $Y_{\varepsilon}:=\left\{ \varphi\in C\left(\mathbb{R}\right) \, :\, - M_\varepsilon \leq \varphi \leq 0 \right\} $ 
endowed with the supremum norm $\|\varphi\|_{\infty}:=\sup_{\xi\in\mathbb{R}}\vert \varphi(\xi)\vert$ 
and where 
\begin{equation}
\label{def:Meps}M_\varepsilon:=\frac{8\pi a}{1-a}\int_{-\infty}^{\infty} dz \, \frac{e^{\varphi(z)}}{1+\varepsilon e^{-z}}\left(\int_{-\infty}^{\infty} 
\left(\frac{1}{|e^\eta-e^z|^{a}}+\frac{1}{|e^\eta+e^z|^{a}}\right) {d\eta}\right)
\end{equation} 
The fixed point problem \eqref{def:tildeLambda2}, \eqref{eq:solsigmasteady} has at least one solution $\varphi_{\varepsilon}\in Y_\varepsilon$ for any $\varepsilon>0$.  
\end{lem}

\begin{proof}     
We first observe that the operator $T_{\varepsilon}$ transforms $ Y_{\varepsilon}$ into itself due to the fact that for any $\varphi\leq 0$ we have the bound $-M_{\varepsilon} \leq T_{\varepsilon}[\varphi] \leq 0$. 

Additionally, it is possible to verify that $T_{\varepsilon}$ is a compact operator since we can control the derivative of $T_\varepsilon$ differentiating \eqref{def:fixedpointTPhi}. Moreover, we can obtain uniform,  exponential decaying estimate for $T_\varepsilon[\varphi](\xi)$ as $|\xi| \to \infty$ for $\varphi\in Y_\varepsilon$. 
Therefore, the space $Y_\varepsilon$ is invariant for the operator $T_\varepsilon$. The compactness of $T_\varepsilon$ then implies the existence of at least one solution to the fixed point problem \eqref{def:tildeLambda2}, \eqref{eq:solsigmasteady} as a consequence of Schauder Fixed point Theorem. 
\end{proof} 

\subsubsection{Derivation of uniform estimates for the   solution to the regularized problem \eqref{eq:fixedpointPhi_reg}, \eqref{def:fixedpointTPhi_reg}}
\label{ssec:der_unifest}

In this section we prove a uniform estimate for $\varphi_\varepsilon$, solution to the regularized problem, whose existence has been obtained in Lemma \ref{lem:existencesol_reg}. We have the following. 
\begin{lem}\label{lem:estimates}
Let $\varphi_{\varepsilon}\in Y_\varepsilon$, for $\varepsilon>0$, be a solution to the fixed point problem \eqref{eq:fixedpointPhi_reg}, \eqref{def:fixedpointTPhi_reg}, given as in Lemma \ref{lem:existencesol_reg}. The following estimates hold:  
\begin{equation}
 \label{eq:estphiep3lem_neg}
     e^{|\varphi_{\varepsilon}(\xi)|} \geq  \left(\frac{1}{4 a^2 }\right)^{\frac 4 5} \left(\frac{e^{-a\xi}}{(1+\varepsilon e^{-\xi})}  \right)^{\frac 4 5} , \quad \text{for} \ \ \xi \leq 0\, ,
\end{equation} 
\begin{equation}
 \label{eq:estphiep3lem_neg2}
   |\varphi_{\varepsilon}(\xi)|  e^{|\varphi_{\varepsilon}(\xi)|} \geq    \frac{1}{a^2 }\frac{e^{-a\xi}}{(1+\varepsilon e^{-\xi})} , \quad \text{for} \ \ \xi\leq 0\,, 
\end{equation}
\begin{equation}
 \label{eq:estphiep3lem_pos}
 \vert \varphi_{\varepsilon}(\xi) \vert \leq \overline{K}, \quad \text{for} \ \  \xi>0
\end{equation}
for some $\overline{K}>0$ that depends on $a$ but it is independent on $\varepsilon$. 
\end{lem}

\begin{proof} 
For the sake of notational simplicity in the argument of this proof we drop the dependence of $\varphi$ on $\varepsilon$ and use in $\varphi$ place of $\varphi_\varepsilon$. 
We consider \eqref{eq:fixedpointPhi_reg}, \eqref{def:fixedpointTPhi_reg} or, equivalently, 
\begin{align}\label{eq:Phi_reg}
   & - \varphi(\xi)= 
\frac{8\pi a}{1-a}\int_{-\infty}^{\infty} dz \, \frac{e^{\varphi(z)}}{1+\varepsilon e^{-z}}\left(\int_{\xi}^{\infty} 
\left(\frac{1}{|e^\eta-e^z|^{a}}+\frac{1}{|e^\eta+e^z|^{a}}\right) {d\eta}\right)
\end{align}
From \eqref{eq:Phi_reg} we obtain
\begin{align}
    - \varphi(\xi) & \geq 
\frac{8\pi a}{1-a}\int_{\xi}^{\infty} dz \, \frac{e^{\varphi(z)}}{1+\varepsilon e^{-z}}\left(\int_{\xi}^{\infty} 
\left(\frac{1}{|e^\eta-e^z|^{a}}+\frac{1}{|e^\eta+e^z|^{a}}\right) {d\eta}\right) \nonumber
\\&
\geq \frac{8\pi a \, e^{\varphi(\xi) }}{1-a}\int_{\xi}^{\infty} dz \, \frac{1}{1+\varepsilon e^{-z}}\left(\int_{\xi}^{\infty} 
\left(\frac{1}{|e^\eta-e^z|^{a}}+\frac{1}{|e^\eta+e^z|^{a}}\right) {d\eta}\right) \nonumber \\& 
= \frac{8\pi a \, e^{\varphi(\xi) }}{1-a}\int_{\xi}^{\infty} dz \, \frac{e^{-az}}{1+\varepsilon e^{-z}}\left(\int_{\xi}^{\infty} 
\left(\frac{1}{|1-e^{\eta-z}|^{a}}+\frac{1}{|1+e^{\eta-z}|^{a}}\right) {d\eta}\right)
\end{align}
where in the second inequality we used the monotonicity of $\varphi,$ namely if $z  \geq \xi$ then $\varphi(z)\geq\varphi(\xi)$. 

Therefore, using the same argument as in Section \ref{ssec:estimatephi} for negative values of $\xi$, that is, using $|\varphi(\xi)|=-\varphi(\xi)$ and multiplying the l.h.s. of the above inequality by $e^{|\varphi(\xi)|}$, estimating the argument in parentheses in the integrand by $1$ below and changing variables we finally arrive at 
$$ 
   |\varphi(\xi)|  e^{|\varphi(\xi)|} \geq e^{-a\xi}\int_0^\infty \frac{t\,e^{-at}}{1+\varepsilon e^{-\xi}e^{-t}} dt\,,\quad \xi\leq 0.
$$
Using that $e^{-t}\leq 1$ so that $(1+\varepsilon e^{-\xi}e^{-t}) \leq (1+\varepsilon e^{-\xi})$ then we obtain 
\begin{equation}
 \label{eq:estphiep}
   |\varphi(\xi)|  e^{|\varphi(\xi)|} \geq \frac{e^{-a\xi}}{(1+\varepsilon e^{-\xi})} \int_0^\infty {t\,e^{-at}} dt = \frac{1}{a^2 }\frac{e^{-a\xi}}{(1+\varepsilon e^{-\xi})}   ,\quad \xi\leq 0\,. 
\end{equation}
This estimate corresponds to the one we obtained in Section \ref{ssec:estimatephi} (cf. \eqref{est:phiabsvalue}) with a cutoff at $\varepsilon \approx e^\xi$, i.e. $\xi \approx \log(\varepsilon) \to - \infty$.  

We now estimate $\varphi(\xi)$ in the region where  $\xi > 0$ starting from \eqref{eq:Phi_reg}.  
We consider \eqref{eq:Phi_reg} and observe that we can split the integral in the $z$ variable, separating the contributions due to the region where $z\leq 0$ and $z>0$: 
\begin{align}\label{eq:estPhi_reg2_split}
     - \varphi(\xi)&= 
\frac{8\pi a}{1-a} \int_{-\infty}^{0} dz \, \frac{e^{\varphi(z)}}{1+\varepsilon e^{-z}}\left(\int_{\xi}^{\infty} 
\left(\frac{1}{|e^\eta-e^z|^{a}}+\frac{1}{|e^\eta+e^z|^{a}}\right) {d\eta}\right) \nonumber\\& \quad +\frac{8\pi a}{1-a} 
\int_{0}^{\infty} dz \, \frac{e^{\varphi(z)}}{1+\varepsilon e^{-z}}\left(\int_{\xi}^{\infty} 
\left(\frac{1}{|e^\eta-e^z|^{a}}+\frac{1}{|e^\eta+e^z|^{a}}\right) {d\eta}\right) 
\end{align}
We then have the following estimate
\begin{align}\label{eq:estPhi_reg2}
 - \varphi(\xi)&
\leq 
\frac{8\pi a}{1-a} \int_{-\infty}^{0} dz \, \frac{e^{\varphi(z)}}{1+\varepsilon e^{-z}}\left(\int_{\xi}^{\infty} 
\left(\frac{1}{|e^\eta-e^z|^{a}}+\frac{1}{|e^\eta+e^z|^{a}}\right) {d\eta}\right) \nonumber\\& \quad +\frac{8\pi a}{1-a} 
\int_{0}^{\infty} dz \, \frac{e^{-az}}{1+\varepsilon e^{-z}}\left(\int_{\xi}^{\infty} 
\left(\frac{1}{|1-e^{\eta-z}|^{a}}+\frac{1}{|1+e^{\eta-z}|^{a}}\right) {d\eta}\right)
\nonumber\\&
= \frac{8\pi a}{1-a} \int_{-\infty}^{0} dz \, \frac{e^{\varphi(z)}}{1+\varepsilon e^{-z}}\left(\int_{\xi}^{\infty} 
\left(\frac{1}{|e^\eta-e^z|^{a}}+\frac{1}{|e^\eta+e^z|^{a}}\right) {d\eta}\right) \nonumber\\& \quad +\frac{8\pi a}{1-a} 
\int_{0}^{\infty} dz \, \frac{e^{-az}}{1+\varepsilon e^{-z}}\left(\int_{\xi-z}^{\infty} 
\left(\frac{1}{|1-e^u|^{a}}+\frac{1}{|1+e^u|^{a}}\right) {d\eta}\right)\nonumber \\&
\leq \frac{8\pi a}{1-a}  \, C_a 
\int_{-\infty}^{0} dz \, \frac{e^{\varphi(z)}}{1+\varepsilon e^{-z}} e^{-a\xi}  + \frac{8\pi a}{1-a} \, C \int_{0}^{\infty} dz \, \frac{e^{-az}}{1+\varepsilon e^{-z}}(1+(z-\xi)_{+}) \nonumber \\&
\leq
\frac{8\pi a}{1-a}\, C_a  e^{-a\xi}
\int_{-\infty}^{0} dz \, \frac{e^{\varphi(z)}}{1+\varepsilon e^{-z}}   + \frac{8\pi a}{1-a} \, C \int_{0}^{\infty} dz \,  {e^{-az}}(1+z)\nonumber\\& \leq \frac{8\pi a}{1-a}  \, C_a e^{-a\xi}
\int_{-\infty}^{0} dz \, \frac{e^{\varphi(z)}}{1+\varepsilon e^{-z}}   + \frac{8\pi a}{1-a} \, \tilde{C}\,.
\end{align}
 In the first inequality we used that, for the second term in the sum,  $\varphi(z) <0$ implies $e^{\varphi(z)}\leq 1$ for $z \geq 0$, and we applied the change of variables $\eta-z=u$ in the subsequent step. In the second inequality we used Moreover, in the second inequality we used that $\left(\int_{\xi}^{\infty} \left(\frac{1}{|e^\eta-e^z|^{a}}+\frac{1}{|e^\eta+e^z|^{a}}\right) {d\eta}\right) \leq C_a e^{-a \xi}$  for $z \leq 0 $ and $C_a>0$  and that $\int_{\xi-z}^{\infty} \left(\frac{1}{|1-e^u|^{a}}+\frac{1}{|1+e^u|^{a}}\right) {d\eta}  \leq C(1+(\xi-z)_{+})$.   
We now use \eqref{eq:estphiep} combined with the inequality 
$ |\varphi(\xi)|  e^{|\varphi(\xi)|} \leq 4\, e^{{\frac 5 4} |\varphi|} $ which implies 
\begin{equation}
 \label{eq:estphiep2}
   4  e^{\frac 5 4 |\varphi(\xi)|} \geq  \frac{1}{a^2 }\frac{e^{-a\xi}}{(1+\varepsilon e^{-\xi})} \,. 
\end{equation}
Therefore, 
\begin{equation}
 \label{eq:estphiep3}
     e^{|\varphi(\xi)|} \geq  \left(\frac{1}{4 a^2 }\right)^{\frac 4 5} \left(\frac{e^{-a\xi}}{(1+\varepsilon e^{-\xi})}  \right)^{\frac 4 5} \, ,
\end{equation}
whence 
\begin{equation}
 \label{eq:estphiep4}
     e^{\varphi(\xi)}= e^{-|\varphi(\xi)|} \leq  \left({4 a^2 }\right)^{\frac 4 5} \left((1+\varepsilon e^{-\xi})  \right)^{\frac 4 5} e^{\frac 4 5 a\xi}\,.
\end{equation}
Going back to  \eqref{eq:estPhi_reg2}, using \eqref{eq:estphiep4} we finally arrive at the following bound:
\begin{align}
 \label{eq:estphiep4}
  - \varphi(\xi) &  \leq \frac{8\pi a}{1-a}\, C_a  e^{-a\xi}
\int_{-\infty}^{0} dz \, \frac{e^{\varphi(z)}}{1+\varepsilon e^{-z}}   + \frac{8\pi a}{1-a} \, C \nonumber \\&
\leq \frac{8\pi a}{1-a}\, C_a \left({4 a^2 }\right)^{\frac 4 5} e^{-a\xi} \int_{-\infty}^{0} dz \, \frac{ \left(1+\varepsilon e^{-\xi}  \right)^{\frac 4 5} e^{\frac 4 5 a\xi} }{1+\varepsilon e^{-z}}   + \frac{8\pi a}{1-a} \bar{C} \nonumber\\&
\leq  K_a e^{-a\xi} \int_{-\infty}^{0} dz \, \frac{  e^{\frac 4 5 a\xi} }{\left(1+\varepsilon e^{-\xi}  \right)^{\frac 1 5}}   +  \widetilde{K}_a \leq K_a e^{-a\xi} \int_{-\infty}^{0} dz \, {  e^{\frac 4 5 a\xi} }   +  \widetilde{K}_a \leq \overline{K}
\end{align}
where we set $K_a =  \frac{8\pi a}{1-a}\, C_a \left({4 a^2 }\right)^{\frac 4 5} $
and $\widetilde{K}_a=\frac{8\pi a}{1-a} \bar{C}$ and the constant $\overline{K}>0$ depends on $a$ but it is independent on $\varepsilon$. Therefore, this yields \eqref{eq:estphiep3lem_pos}, that is
$$\vert \varphi(\xi) \vert \leq \overline{K} ,\quad \xi>0$$
for some $\overline{K}>0$ that depends on $a$ but it is independent on $\varepsilon$. This implies that $e^{\varphi(\xi)}$ remains bounded from below for $\xi>0$ and we do not have the vanishing of the exponential,  that is $e^{\varphi(\xi)}\not\to 0$ as $\varepsilon\to 0$ in the region where $\{\xi >0\}$. 
\end{proof}

\bigskip

\subsubsection{Taking the limit $\varepsilon\to 0$ in the regularized problem \eqref{eq:fixedpointPhi_reg}, \eqref{def:fixedpointTPhi_reg}} 
\label{ssec:limit}

We now show that we can take a subsequence $\{\varepsilon_k\}_{k=1}^\infty$ such that $\varepsilon_k\to 0$ as $k\to \infty$ and  such that $\varphi_{\varepsilon_k} \to \varphi$ where the limit $\varphi$ solves the fixed point problem $\varphi(\xi)= T[\varphi](\xi)$ (cf. \eqref{eq:fixedpointPhi}, \eqref{def:fixedpointTPhi}). 

It is easy to prove compactness, in compact sets, in the topology of the uniform convergence for  $\{\varphi_{\varepsilon_k}\}$ taking the derivative $\frac{d\varphi_\varepsilon}{d\xi}(\xi)$  in the integral equation \eqref{eq:fixedpointPhi_reg}, \eqref{def:fixedpointTPhi_reg}.  

Notice that we cannot have $\varphi_{\varepsilon_k} \to -\infty$ due to the estimate \eqref{eq:estphiep3lem_pos}. 
From Lemma \ref{lem:estimates} we have bounds for $\varphi_{\varepsilon_k}$ for positive and negative values of $\xi$.  
We consider first the second estimate \eqref{eq:estphiep3lem_neg2},  that is
\begin{equation*}
   |\varphi_{\varepsilon}(z)|  e^{|\varphi_{\varepsilon}(z)|} \geq    \frac{1}{a^2 }\frac{e^{-az}}{(1+\varepsilon e^{-z})} , \quad \text{for} \ \ z\leq 0\,, 
\end{equation*}
which implies 
\begin{equation}
 \label{eq:estphiep4bis2}
     e^{\varphi_{\varepsilon}(z)}
     \leq  \left({4 a^2 }\right)^{\frac 4 5} \left((1+\varepsilon e^{-z})  \right)^{\frac 4 5} e^{\frac 4 5 az}\,. 
\end{equation}
Moreover, notice that taking the limit $\varepsilon\to 0$ in 
\eqref{eq:estphiep4bis2} we obtain 
\begin{equation}
\label{est:limitphi}
   e^{\varphi(z)} 
     \leq  \left({4 a^2 }\right)^{\frac 4 5}  e^{\frac 4 5 az}\,.  
\end{equation}
which, for large values, implies a linear estimate for $\varphi$. 

By shortness we will write $\varepsilon=\varepsilon_k$ in what follows. We examine the limit of the integral  as $\varepsilon \to 0$ (or, equivalently, $k\to \infty$)
$$
\frac{8\pi a}{1-a}\int_{-\infty}^{\infty} dz \, \frac{e^{\varphi_{\varepsilon}(z)}}{1+\varepsilon e^{-z}}\left(\int_{\xi}^{\infty} 
\left(\frac{1}{|e^\eta-e^z|^{a}}+\frac{1}{|e^\eta+e^z|^{a}}\right) {d\eta}\right)\,.
$$
We split the integral in the regions where $\{z \leq 0\}$ and  $\{z > 0\}$ and consider first the contribution due to  $\{z \leq 0\}$ where we can use the estimate above. We have 
\begin{align}
& \frac{8\pi a}{1-a}\int_{-\infty}^{0} dz \, \frac{e^{\varphi_{\varepsilon}(z)}}{1+\varepsilon e^{-z}}\left(\int_{\xi}^{\infty} 
\left(\frac{1}{|e^\eta-e^z|^{a}}+\frac{1}{|e^\eta+e^z|^{a}}\right) {d\eta}\right)  \nonumber \\& 
\leq \frac{8\pi a}{1-a}\left({4 a^2 }\right)^{\frac 4 5} \int_{-\infty}^{0} dz \, \frac{\left((1+\varepsilon e^{-z})  \right)^{\frac 4 5} e^{\frac 4 5 az}}{1+\varepsilon e^{-z}}\left(\int_{\xi}^{\infty} 
\left(\frac{1}{|e^\eta-e^z|^{a}}+\frac{1}{|e^\eta+e^z|^{a}}\right) {d\eta}\right)  \nonumber \\&
\leq \frac{8\pi a}{1-a}\left({4 a^2 }\right)^{\frac 4 5} \int_{-\infty}^{0} dz \, \frac{  e^{\frac 4 5 az}}{(1+\varepsilon e^{-z})^{\frac 1 5} }\left(\int_{\xi}^{\infty} 
\left(\frac{1}{|e^\eta-e^z|^{a}}+\frac{1}{|e^\eta+e^z|^{a}}\right) {d\eta}\right)  
\nonumber \\&
\leq \frac{8\pi a}{1-a}\left({4 a^2 }\right)^{\frac 4 5} \int_{-\infty}^{0} dz \,   e^{\frac 4 5 az}\left(\int_{\xi}^{\infty} 
\left(\frac{1}{|e^\eta-e^z|^{a}}+\frac{1}{|e^\eta+e^z|^{a}}\right) {d\eta}\right) <  \infty
\end{align}
where in the last inequality we used that $(1+\varepsilon e^{-z})^{-\frac 1 5}  <1$. 

We now consider the contribution due to the region $\{z>0\}$. In this region we can use the fact that $\varphi_{\varepsilon} \leq 0$ whence $e^{\varphi_{\varepsilon}(z)} \leq 1$. We then have 
\begin{align}
&\frac{8\pi a}{1-a}\int_{0}^{\infty} dz \, \frac{e^{\varphi_{\varepsilon}(z)}}{1+\varepsilon e^{-z}}\left(\int_{\xi}^{\infty} 
\left(\frac{1}{|e^\eta-e^z|^{a}}+\frac{1}{|e^\eta+e^z|^{a}}\right) {d\eta}\right) \nonumber \\&
\leq \frac{8\pi a}{1-a}\int_{0}^{\infty} dz \, \left(\int_{\xi}^{\infty} 
\left(\frac{1}{|e^\eta-e^z|^{a}}+\frac{1}{|e^\eta+e^z|^{a}}\right) {d\eta}\right)< \infty\,.
\end{align} 
Notice that in order to check that the last integral is finite  we observe that for $|\eta -z| \leq 1$ the integrand can be estimated by $\frac {e^{-a\eta}}{|\eta -z|^a}$ which gives the desired exponential decay. 

We can then take, up to subsequences, the limit $\varepsilon \to 0$ using the Lebesgue's dominated convergence Theorem. This gives the existence of a solution $\varphi$ satisfying  
$\varphi(\xi)= T[\varphi](\xi)$ (cf. \eqref{eq:fixedpointPhi}, \eqref{def:fixedpointTPhi}). 
More precisely, we summarize here the result obtained in this Section. 
\begin{lem}\label{lem:existencelimit}
    There exists a solution $\varphi\in C(\mathbb{R})$ to $\varphi(\xi)= T[\varphi](\xi)$ (cf.~\eqref{eq:fixedpointPhi}, \eqref{def:fixedpointTPhi}) such that $-\infty<\varphi(\xi) < 0$, and 
\begin{align}\label{eq:estphi1lem_pos2}
        &\vert \varphi(\xi)\vert \leq C_1 \quad \text{ for} \ \  \xi \geq 0 ;  \\&
        \varphi(\xi) \leq -C_2 \vert \xi\vert \quad \text{ for} \ \ \xi \leq 0,\label{eq:estphi2lem_neg2}
    \end{align}
        with  constants $C_1, C_2>0$. 
Moreover, we have
\begin{equation}
 \label{eq:estphi3lem_neg2}
   |\varphi(\xi)|  e^{|\varphi(\xi)|} \geq    \frac{1}{a^2 }e^{a|\xi|} , \quad \text{for} \ \ \xi\leq 0\,. 
\end{equation} 
\end{lem}
Notice that \eqref{eq:estphi1lem_pos2} 
follows from the estimate \eqref{eq:estphiep3lem_pos} for $\varphi_\varepsilon$ that implies that a similar estimate holds for the limit function $\varphi$. On the other hand, the estimate \eqref{eq:estphi2lem_neg2} 
follows from \eqref{est:limitphi}. Moreover, \eqref{eq:estphi3lem_neg2} follows taking the limit $\varepsilon\to 0$ in \eqref{eq:estphiep3lem_neg2}. 
Using the change of variables \eqref{eq:chexpvar} as well as the symmetry of $\sigma$, i.e. 
$\sigma(y)=\sigma(-y)$, Lemma \ref{lem:existencelimit} allows to conclude the existence part of the proof of Theorem \ref{thm:sigma}.

\bigskip

\subsection{
Proof of Theorem \ref{thm:sigma}: Proof of the asymptotic behaviour of the solution to \eqref{eq:pde_b_stat}
} 

\label{ss:asymb}

We rewrite the problem 
\eqref{def:dampingPb2}, \eqref{def:tildeLambda}, with the boundary condition \eqref{eq:bdcdtsigma}, using a more convenient notation. 
More precisely, from \eqref{eq:integralLambda}, that is 
\begin{equation}
\label{eq:integralLambda2}
\int_{y}^{\infty}{\widetilde{\Lambda}}[\sigma](u)\frac{du}{u} =  \frac{8\pi a}{1-a} \int_{0}^{\infty}\frac{d\zeta}{\zeta}\sigma(\zeta)\int_{y}^{\infty} \left(\frac{1}{|u-\zeta|^{a}}+\frac{1}{|u+\zeta|^{a}}\right) \frac{du}{u}   
\end{equation}
we first consider the inner integral in the formula above and setting $u=\zeta z$, we obtain
\begin{equation}
\int_{y}^{\infty}\left(\frac{1}{|u-\zeta|^{a}}+\frac{1}{|u+\zeta|^{a}}\right)\frac{du}{u}
=: \frac 1 {\zeta^a} Q\left(\frac{y}{\zeta}\right)
\end{equation}
where we defined
\begin{equation}\label{def:QopT}
Q(s)=\int_{s}^{\infty}\left(\frac{1}{|z-1|^{a}}+\frac{1}{|z+1|^{a}}\right)\frac{dz}{z}  , \quad s \in \mathbb{R}\, 
\end{equation}
and thus \eqref{eq:integralLambda2} becomes 
\begin{equation}
\label{eq:integralLambda3} \int_{y}^{\infty}\Lambda[\sigma](u)\frac{du}{u}=\int_{0}^{\infty}\frac{\sigma(\zeta)}{\zeta^{1+a}}Q\left(\frac{y}{\zeta}\right) d\zeta\,. 
\end{equation}
We have 
\begin{equation}\label{eq:asymQs}
Q(s) \sim 2 \log\left(\frac 1 s\right) \text{ as } s \rightarrow 0, \qquad  Q(s) \sim  \frac {2}{a}\frac 1 {s^a}  \ \text{as } s\to\infty \,.
\end{equation} 
We observe also that $s\rightarrow Q(s)$ is decreasing. 

Using \eqref{def:QopT} we can rewrite \eqref{eq:solsigmasteady}, i.e. $\sigma(y)=\exp\left(-\int_{y}^{\infty}{\widetilde{\Lambda}}[\sigma](u)\frac{du}{u}\right)$, with $y>0$, as  
\begin{equation} \label{eq:MathU}
\sigma(y)=\exp\left(-\int_{0}^{\infty}\frac{\sigma(\zeta)}{\zeta^{1+a}} \, Q\left(\frac{y}{\zeta}\right)d\zeta\right):=\mathcal{U}[\sigma](y) \,.
\end{equation}
Notice that the existence of at least one solution $\sigma\in C((0,+\infty))$ to the equation above follows from the combination of the existence result in Section \ref{ssec:limit} (Lemma \ref{lem:existencelimit}) 
with the change of variables \eqref{eq:chexpvar}. 
We now prove the following properties of the solution $\sigma$. 
\begin{lem}\label{lem:asymptsigma}
There exists at least one solution  $\sigma \in C((0,+\infty))$ of \eqref{eq:MathU} such that the following properties hold:
\begin{itemize}
    \item[(i)] There exists $\lim_{y\to 0^{+}}\sigma(y)=:L\in \mathbb{R}$, $L<1$ and if we extend the function $\sigma$ adding the limit value in $0$, i.e. $\sigma(0)=L$, then  the resulting function $\sigma(\cdot)\in C([0,\infty))$; 
    \item[(ii)] $\sigma(0)=0$;
    \item[(iii)] $\lim_{y\to 0^+} \frac{\log(\sigma(y))}{\log(y)} =\beta \geq a$. 
\end{itemize} 
\end{lem}

\begin{proof}
We start proving the properties $(i)-(ii)$. They immediately follow from the fact that $\sigma=e^\varphi$, $y=e^\xi,$ (cf. \eqref{eq:chexpvar}) and 
from the estimate $\varphi(\xi) \leq -C_2 \vert \xi\vert$ for $\xi \leq 0$, with $C_2>0$ (cf. Lemma \ref{lem:existencelimit}). Moreover, this estimate implies that $\sigma(y) \leq y^{C_2}$ for $0\leq y\leq 1$.

\medskip

Our goal is now to prove $(iii)$.   In order to do that, we begin proving that  
\begin{equation}\label{eq:sigma_integrab}
  \int_0^\infty \frac{\sigma(\zeta)}{\zeta^{1+a}} \,d\zeta <  +\infty.   
\end{equation}
Since  $\int_0^\infty \frac{\sigma(\zeta)}{\zeta^{1+a}} \,d\zeta= \int_0^1 \frac{\sigma(\zeta)}{\zeta^{1+a}} \,d\zeta +\int_1^\infty \frac{\sigma(\zeta)}{\zeta^{1+a}} \,d\zeta$ and $0\leq \sigma(\zeta) \leq 1$ in the region where $\zeta >1$, it follows that $\int_1^\infty \frac{\sigma(\zeta)}{\zeta^{1+a}} \,d\zeta <  +\infty$. Therefore, it is enough to analyse the integrability of $\int_0^1 \frac{\sigma(\zeta)}{\zeta^{1+a}} \,d\zeta$.  In order to do this we argue by contradiction. Suppose that 
\begin{equation}
    \label{eq:contrass}
\int_{0}^{1} \frac{\sigma(\zeta)}{\zeta^{1+a}} d\zeta = \infty. 
\end{equation}
Then
\begin{equation}
\label{eq:est_int1}\int_{0}^{\infty} \frac{\sigma(\zeta)}{\zeta^{1+a}} Q\left(\frac{y}{\zeta}\right)\, d\zeta \geq \int_{0}^{1} \frac{\sigma(\zeta)}{\zeta^{1+a}} Q\left(\frac{y}{\zeta}\right)\, d\zeta\,.
\end{equation}
Suppose that we fix an arbitrary $\delta > 0$. Using Lebesgue's dominated convergence Theorem we have 
$$\lim_{y\to 0^{+}}\int_{\delta}^{1} \frac{\sigma(\zeta)}{\zeta^{1+a}} \frac{Q\left(\frac{y}{\zeta}\right)}{\log(\frac{1}{y})} d\zeta = 2 \int_{\delta}^{1} \frac{\sigma(\zeta)}{\zeta^{1+a}} d\zeta$$ 
from which it follows:
\begin{equation}\label{eq:est_int2}
    \liminf_{y \to 0^+} \frac{1}{\log(\frac{1}{y})} \int_{0}^{\infty} \frac{\sigma(\zeta)}{\zeta^{1+a}} Q\left(\frac{y}{\zeta}\right) d\zeta \ge \lim_{y \to 0^+} \frac{1}{\log(\frac{1}{y})} \int_{\delta}^{1} \frac{\sigma(\zeta)}{\zeta^{1+a}} Q\left(\frac{y}{\zeta}\right) d\zeta = 2 \int_{\delta}^{1} \frac{\sigma(\zeta)}{\zeta^{1+a}} d\zeta\,.
\end{equation}
Choosing $M$ arbitrarily large we obtain that, for $\delta>0$ small, 
\begin{equation}  \label{eq:est_int3}
2 \int_{\delta}^{1} \frac{\sigma(\zeta)}{\zeta^{1+a}} d\zeta \ge 2M .
\end{equation}
From \eqref{eq:est_int1}, using \eqref{eq:est_int2} and \eqref{eq:est_int3}, we arrive at
\begin{equation}
\int_{0}^{\infty} \frac{\sigma(\zeta)}{\zeta^{1+a}} Q\left(\frac{y}{\zeta}\right) ds \ge M \log\left(\frac{1}{y}\right)
\end{equation}
which holds true for $0\leq y \leq r_0$ with $r_0>0$ small. 
Then 
\begin{equation}
\label{eq:estsigma(y)}    
\sigma(y)= \exp\left(-\int_{0}^{\infty} \frac{\sigma(\zeta)}{\zeta^{1+a}} Q\left(\frac{y}{\zeta}\right) d\zeta\right)  \le \exp(-M \log(1/y)) \leq  y^M,
\end{equation} 
for $0\leq y \leq r_0$.  
This implies that $\int_{0}^{1} \frac{\sigma(\zeta)}{\zeta ^{1+a}} d\zeta < \infty $ if $M$ is chosen large enough, contradicting the assumption \eqref{eq:contrass}.  We have thus proved \eqref{eq:sigma_integrab}. 
\bigskip 

We compute the asymptotic behaviour of $\sigma(y)$ as $y\to 0^+$, where $\sigma(y)$ is given by
\begin{equation}
    \label{def:sigmay_A}
    \sigma(y)=\exp\left(-\int_{0}^{\infty}\frac{\sigma(\zeta)}{\zeta^{1+a}}Q\left(\frac{y}{\zeta}\right) d\zeta\right)\,.
\end{equation} 
As a next step we will derive the following asymptotics: 
\begin{align}\label{eq:proofasym}
& \int_0^\infty 
\frac{\sigma(\zeta)}{\zeta^{1+a}}Q\left(\frac{y}{\zeta}\right) \,d\zeta \sim -\beta \log(y) \qquad \text{as} \ \ y\to 0^+
\nonumber 
\\&
\beta := 2  \int_0^\infty 
\frac{\sigma(\zeta)}{\zeta^{1+a}} \, d\zeta >0\,.
\end{align}
To this end we notice that 
\begin{equation}\label{eq:proofasym2}
    \int_0^\infty 
\frac{\sigma(\zeta)}{\zeta^{1+a}}Q\left(\frac{y}{\zeta}\right) \,d\zeta = \int_0^{My}
\frac{\sigma(\zeta)}{\zeta^{1+a}}Q\left(\frac{y}{\zeta}\right) \,d\zeta+\int_{My}^\infty 
\frac{\sigma(\zeta)}{\zeta^{1+a}}Q\left(\frac{y}{\zeta}\right) \,d\zeta =: J_1+J_2\,,
\end{equation}
where $M\geq 1$ is a large number to be determined. 

We first consider $J_1$: In the region where $\zeta \leq My$ we have that $Q\left(\frac{y}{\zeta}\right) \leq K_{M}$ where $K_M>0$ is a constant depending on $M\geq 1$. Therefore, 
\begin{equation}
    \label{eq:proof_estJ1}
J_1=\int_0^{My}
\frac{\sigma(\zeta)}{\zeta^{1+a}}Q\left(\frac{y}{\zeta}\right) \,d\zeta \leq K_M \int_0^{My}
\frac{\sigma(\zeta)}{\zeta^{1+a}}
\,d\zeta \end{equation}
and due to integrability \eqref{eq:sigma_integrab},  the r.h.s. of \eqref{eq:proof_estJ1} converges to $0$ as $y\to 0^+$. 

We now estimate $J_2$. Using the asymptotics for $Q(s)$ as $s\to 0$ (cf. \eqref{eq:asymQs}) we obtain
\begin{equation}\label{eq:proof_estJ2}
  (2-\varepsilon)\int_{My}^\infty 
\frac{\sigma(\zeta)}{\zeta^{1+a}}\log\left(\frac{y}{\zeta} \right) \,d\zeta   \leq J_2 \leq (2+\varepsilon)\int_{My}^\infty 
\frac{\sigma(\zeta)}{\zeta^{1+a}}\log\left(\frac{y}{\zeta} \right)   \,d\zeta , \quad \forall \varepsilon>0. 
\end{equation}
Therefore, we only need to study the asymptotics of the integral 
$\int_{My}^\infty 
\frac{\sigma(\zeta)}{\zeta^{1+a}}\log\left(\frac{y}{\zeta} \right) \,d\zeta \to 0$. 
To this aim we consider 
\begin{align}\label{eq:intJ1J2}
    \int_{My}^\infty 
\frac{\sigma(\zeta)}{\zeta^{1+a}}\log\left(\frac{y}{\zeta} \right) \,d\zeta &= -\log\left( y \right) \int_{My}^\infty 
\frac{\sigma(\zeta)}{\zeta^{1+a}}\,d\zeta
+\int_{My}^\infty 
\frac{\sigma(\zeta)}{\zeta^{1+a}}\log\left({\zeta} \right) \,d\zeta =: \tilde{J}_1+\tilde{J}_2\,.
\end{align}
Notice that, using \eqref{eq:sigma_integrab}, for the first term on the r.h.s. we have $$\tilde{J}_1=\log(y) \int_{My}^\infty 
\frac{\sigma(\zeta)}{\zeta^{1+a}}\,d\zeta \rightarrow \log(y) \int_{0}^\infty 
\frac{\sigma(\zeta)}{\zeta^{1+a}}\,d\zeta \quad  \text{as} \ \  y\to 0^+.$$ 
Regarding the second term on  the r.h.s. instead, we can split the integral  in two regions, specifically
\begin{align}
 \tilde{J}_2=\int_{My}^\infty 
\frac{\sigma(\zeta)}{\zeta^{1+a}}\log\left({\zeta} \right) \,d\zeta = \int_{My}^\delta  
\frac{\sigma(\zeta)}{\zeta^{1+a}}\log\left({\zeta} \right) \,d\zeta +\int_{\delta}^\infty 
\frac{\sigma(\zeta)}{\zeta^{1+a}}\log\left({\zeta} \right) \,d\zeta  =:\tilde{J}_{2,1}+\tilde{J}_{2,2}
\end{align}
with $\delta >0$. 
We observe that $J_{2,2}$ is bounded by a constant that depends on $\delta$, namely $\tilde{J}_{2,2} \leq K_\delta$ and that 
\begin{equation} \label{eq:proof_estJ2tilde}
\vert \tilde{J}_{2,1} \vert =\left\vert \int_{My}^\delta  
\frac{\sigma(\zeta)}{\zeta^{1+a}}\log\left({\zeta} \right) \,d\zeta     \right \vert \leq \vert \log(My) \vert \int_{My}^\delta  
\frac{\sigma(\zeta)}{\zeta^{1+a}} \leq  \eta_\delta \vert \log(My) \vert
 \end{equation}
where $\eta_\delta \to 0$ as $\delta \to 0$. 

Then, using \eqref{eq:proof_estJ2}, and defining $\beta:= 2\int_{0}^\infty 
\frac{\sigma(\zeta)}{\zeta^{1+a}} \,d\zeta$ we have
\begin{equation} \label{eq:proof_estJ2_bis}
\beta- O(\eta_{\delta})-O(\varepsilon) \leq \liminf_{y \to 0^+} \frac{J_2}{\left\vert \log\left( y\right) \right\vert } \leq \limsup_{y \to 0^+} \frac{J_2}{\left\vert \log\left(y\right) \right\vert } \leq \beta+O(\eta_{\delta})+O(\varepsilon)\,.  
\end{equation} 
From \eqref{eq:proofasym2}, using \eqref{eq:proof_estJ1} and taking the limits $\delta \to 0$ and $\varepsilon\to 0$  in \eqref{eq:proof_estJ2_bis} it then implies  \eqref{eq:proofasym}.

We now take the  $\log(\cdot)$ of both sides of   \eqref{def:sigmay_A} and dividing by $\log\left( y\right)$, using \eqref{eq:proofasym} and taking the limit $y\to 0^+$, we obtain
$$\lim_{y\to 0^+} \frac{\log(\sigma(y))}{\log(y)}=\beta\,.$$

We further observe that, from \eqref{eq:estphi3lem_neg2} in Lemma \ref{lem:existencelimit}, using the change of variables \eqref{eq:chexpvar} and taking the limit $y \to 0^+$ we obtain that $\beta \geq a$.   This concludes the proof of item $(iii)$. 

\end{proof}

\subsection{End of the Proof of Theorem \ref{thm:sigma}}

\begin{proof}
    The existence of at least one solution to the problem \eqref{def:dampingPb}-\eqref{eq:statLambda_bis} with the properties stated in Theorem \ref{thm:sigma} follows from Lemma \ref{lem:existencelimit} for what concerns the existence of a solution $\sigma$ and from Lemma \ref{lem:asymptsigma} for the asymptotic behaviour. We now show that any non-negative solution to \eqref{def:dampingPb}-\eqref{eq:statLambda_bis} satisfies the properties  stated in Theorem \ref{thm:sigma}, namely that $0<\sigma(y)\leq 1$ for $y\in \mathbb{R}$ and the asymptotics \eqref{eq:thmasym}. Indeed, it follows from  \eqref{def:dampingPb} that the solution $\sigma$ is increasing and if $\sigma(y)=0$ for $y\in I\subset \mathbb{R}_+$, where $I$ is any finite interval or a singleton, then it would imply $\sigma(y)=0$ for any 
$y\in \mathbb{R}_{+}$ and this would violate the boundary condition $\sigma(y)\to 1 $ as $y \to \infty$. 

In order to prove  \eqref{eq:thmasym}, that is 
$$\lim_{y\to 0^+} \frac{\log(\sigma(y))}{\log(y)}=\beta  \quad \text{with} \ \ \beta \geq a\,,$$
we remark that the heuristic argument in Section \ref{ssec:estimatephi} providing the a priori estimate \eqref{est:phi} is now rigorous if $\varphi$ is defined by means of \eqref{eq:chexpvar}.  Moreover, it holds $\vert \varphi(\xi)\vert \leq C$, $\xi \geq 0$ for any solution $\varphi$. 
Arguing then as in Section \ref{ssec:der_unifest} with $\varepsilon=0$ we can prove \eqref{eq:estphi3lem_neg2}. Therefore, we can argue as in the proof of the  Lemma \ref{lem:asymptsigma} and then deduce the asymptotics \eqref{eq:thmasym}. 

\end{proof}
 
\subsection{Proof of the identity \eqref{eq:lamHbar8}} \label{ss:cutoff3}
 
In this Section we prove the following result, that has been used in Section \ref{sec:asymG}. 
\begin{prop}\label{prop:intLam}
Let be $0<a<1$ and suppose that $\sigma=\sigma(y)$ solves the problem \eqref{def:dampingPb} with $\sigma(0)=0$. Then the following identity holds:
\begin{equation}
    4\pi \int_{-\infty}^{\infty} dy \int_{-\infty}^{\infty} d\zeta \, \frac{1}{|{y-\zeta}|^{a}}     \,  { \frac{{\sigma}\left( y \right)}{|y|}\,  \frac{{\sigma}\left( \zeta \right)}{|\zeta|} } = \frac{(1-a)}{a}
\label{eq:lamHbar8bis}
\end{equation}\end{prop}

\begin{proof}

By assumption we have that $\sigma$ is a solution to \eqref{def:dampingPb}, namely 
\begin{equation}\label{def:dampingPb_bis}
\left(\frac{1}{a}-1\right)y\frac{\partial \sigma}{\partial y}={\Lambda}[\sigma](y) \sigma, \qquad y>0  
\end{equation}
where 
\begin{equation*}
{\Lambda}[\sigma](y)=8\pi\int_{\mathbb{R}} d\zeta \, \frac 1 {|y-\zeta|^{a}}\frac{\sigma(\zeta)}{|\zeta|}
\end{equation*}

We first recall that $\sigma$ is a symmetric function, i.e. $\sigma(y)=\sigma(-y)$.   We consider now the integral on the l.h.s. of \eqref{eq:lamHbar8bis} and, using the symmetry of $\sigma$ we can rewrite it  as 
\begin{equation}
    \label{eq:intLam}
4\pi \int_{-\infty}^{\infty} dy \int_{-\infty}^{\infty} d\zeta \, \frac{1}{|{y-\zeta}|^{a}}     \,  { \frac{{\sigma}\left( y \right)}{|y|}\,  \frac{{\sigma}\left( \zeta \right)}{|\zeta|} } = 8\pi \int_{0}^{\infty} dy \int_{-\infty}^{\infty} d\zeta \, \frac{1}{|{y-\zeta}|^{a}}     \,  { \frac{{\sigma}\left( y \right)}{y}\,  \frac{{\sigma}\left( \zeta \right)}{|\zeta|} }
\end{equation}

Therefore, since the equation is satisfied for $y>0$ we can equivalently rewrite \eqref{def:dampingPb_bis} as 
\begin{equation*}
\left(\frac{1}{a}-1\right)\frac{\partial \sigma}{\partial y}=8\pi\int_{\mathbb{R}} d\zeta \, \frac 1 {|y-\zeta|^{a}}\frac{\sigma(\zeta)}{|\zeta|} \frac{\sigma(y)}{y}, \qquad y>0  
\end{equation*} 
We now integrate in $y$ the equation above. Using the symmetry of $\sigma$ and the fact that $\sigma(0)=0$ and $\sigma(\infty)=1$, we thus obtain 
\begin{equation} \label{eq:intLam2}
\left(\frac{1}{a}-1\right) =8\pi \int_{0}^\infty  dy \,\int_{\mathbb{R}} d\zeta  \frac 1 {|y-\zeta|^{a}}\frac{\sigma(\zeta)}{|\zeta|} \frac{\sigma(y)}{y}.  
\end{equation} 
Combining \eqref{eq:intLam2} with \eqref{eq:intLam} we then prove \eqref{eq:lamHbar8bis}. 
 
\end{proof}


\newpage 
\appendix
\section{Appendix A}
\subsection{Conservation of mass property for the equation satisfied by the distribution $F$}
 \label{ssec:massF}

We consider equation \eqref{eq:ss_strong2}  satisfied by the rescaled function $F$, that is,
\begin{equation*}
    \partial_{\tau}F-\frac{1}{a}\partial_{\xi_1}(\xi_1F)-\left(\frac{1}{a}-1\right)\partial_{\xi_2}(\xi_2F)-\left(\frac{1}{a}-1\right)\partial_{\xi_3}(\xi_3F)
   +  \partial_{\xi_1}(\xi_2 F )=Q^+[F,F]-Q^-[F,F] 
\end{equation*}
where
\begin{align*}
 & {Q}^+[F,F]({\xi})=  \frac {32}{2^a} \,t^{a-1}\,\delta(\xi_1)\,\int_{\R} dx \int_{\mathbb{R}^2}d\tilde{y}\, \int_0^{\pi} d\theta \,F\left( x_1+  |\tilde \xi|\left(t \sin(2\theta)\right)^{-1}, x_2+y_2,x_3+y_3, \tau\right) \times \notag \\& \qquad \qquad \qquad \qquad \qquad \times F\left(x_1- |\tilde \xi|\left(t \sin(2\theta)\right)^{-1},x_2-y_2,x_3-y_3,\tau \right) \frac{\sin\theta \, \vert \sin(2\theta)\vert^{a-1} }{  |\tilde \xi| ^{a+1}}    
 \end{align*}
 \begin{equation*}
  {Q}^-[F,F](\xi)=2F(\xi,\tau) \int_{S^2} d\omega\int_{\R}d\xi_{\ast} \,|\xi_1-\xi_{1,\ast}|^{-a} F(\xi_{\ast},\tau)\, 
\end{equation*} 
and prove that this equation satisfies the conservation of mass property.  
In order to do this it is enough to check the  integral of the collision operator. By definition of the collision integral, indeed, the net number of particles must be conserved, i.e.  $\int_{\R} d\xi \, Q[F,F](\xi,\tau) =0$. We start considering the contribution of the loss term \eqref{eq:ss_strong2_loss}. We have 
\begin{equation}\label{eq:MassF_loss}
\int_{\R}d\xi  \, Q^-[F,F](\xi,\tau)  = 8\pi \int_{\R} d\xi \int_{\R} d\xi_{\ast}\, {|\xi_1-\xi_{1,\ast}|^{-a}} \, {F(\xi,\tau)F(\xi_{\ast},\tau)}\, ,
\end{equation}
using that $\int_{S^2} d\omega=4\pi$. 
On the other hand, the contribution of the gain term \eqref{eq:ss_strong2_gain} gives
\begin{align}
\label{eq:MassF_gain}
&\int_{\R}d\xi \,  Q^+[F,F](\xi,\tau)  \notag \\& 
=  \frac {32}{2^a} t^{a-1} \int_{\R}d\xi \, \delta(\xi_1)\int_{\R} dx \int_{\mathbb{R}^2}d\tilde{y}\, \int_0^{\pi} d\theta  \,F\left( x_1+  |\tilde \xi|\left(t \sin(2\theta)\right)^{-1}, x_2+y_2,x_3+y_3, \tau\right) \times \notag \\& \qquad \qquad \qquad \qquad \qquad \times F\left(x_1-|\tilde \xi|\left(t \sin(2\theta)\right)^{-1},x_2-y_2,x_3-y_3,\tau \right) \frac{  \sin\theta \, \vert \sin(2\theta)\vert^{a-1} }{  |\tilde \xi| ^{a+1}} \nonumber\\& 
=  \frac {32}{2^a} t^{a-1} \int_{\mathbb{R}^2}d\tilde{\xi} \, \int_{\R} dx \int_{\mathbb{R}^2}d\tilde{y}\, \int_0^{\pi} d\theta  \,F\left( x_1+  |\tilde \xi|\left(t \sin(2\theta)\right)^{-1}, x_2+y_2,x_3+y_3, \tau\right) \times \notag \\& \qquad \qquad \qquad \qquad \qquad \times F\left(x_1-|\tilde \xi|\left(t \sin(2\theta)\right)^{-1},x_2-y_2,x_3-y_3,\tau \right) \frac{  \sin\theta \, \vert \sin(2\theta)\vert^{a-1} }{  |\tilde \xi| ^{a+1}}\,.
\end{align}
We now change variables and set $x_2+y_2=\eta_2,\, x_3+y_3=\eta_3$, $x_2-y_2=\zeta_2,\, x_3-y_3=\zeta_3$ with Jacobian $dx_2dy_2=   \frac 1 2 d\eta_2d\zeta_2$, $dx_3dy_3=   \frac 1 2 d\eta_3 d\zeta_3$ respectively, so that \eqref{eq:MassF_gain} becomes
\begin{align}
\label{eq:MassF_gain2}
&\int_{\R}d\xi \,  Q^+[F,F](\xi,\tau)  \notag \\&
=  \frac {32}{2^a} t^{a-1}  \frac 1 4 \int_{\mathbb{R}^2}d\tilde{\xi} \, \int_{\mathbb{R}} dx_1 \int_{\mathbb{R}}d\eta_2  \int_{\mathbb{R}} d\eta_3  \int_{\mathbb{R}} d\zeta_2 \int_{\mathbb{R}} d\zeta_3 \int_0^{\pi} d\theta  \,F\left( x_1+  |\tilde \xi|\left(t \sin(2\theta)\right)^{-1},\eta_2,\eta_3, \tau\right) \times \notag \\& \qquad \qquad \qquad \qquad \qquad \times F\left(x_1-|\tilde \xi|\left(t \sin(2\theta)\right)^{-1},\zeta_2,\zeta_3,\tau \right) \frac{  \sin\theta \, \vert \sin(2\theta)\vert^{a-1} }{  |\tilde \xi| ^{a+1}}\end{align}
where in the second identity we used the change of variable $z=\frac{|\tilde \xi|}{t \sin(2\theta)}$. However, the function $\frac{|\tilde \xi|}{t \sin(2\theta)}$ is multivalued in $\theta \in [0,\pi]$. 
We introduce the following auxiliary function lighten the notation and the computations. We define 
\begin{equation}\label{eq:defL_app}
L(x_1, \tau):=\int_{\mathbb{R}}d\eta_2  \int_{\mathbb{R}} d\eta_3  F\left( x_1,\eta_2,\eta_3, \tau\right) \,.
\end{equation}
Using \eqref{eq:defL_app} in \eqref{eq:MassF_gain2} we arrive at 
\begin{align}
\label{eq:MassF_gain3}
&\int_{\R}d\xi \,  Q^+[F,F](\xi,\tau)  \notag \\&
=  \frac {8}{2^a} t^{a-1} \int_{\mathbb{R}^2}d\tilde{\xi} \, \int_{\mathbb{R}} dx_1   \int_0^{\pi} d\theta  \,L\left( x_1+  |\tilde \xi|\left(t \sin(2\theta)\right)^{-1}, \tau\right) 
L\left(x_1-|\tilde \xi|\left(t \sin(2\theta)\right)^{-1},\tau \right) \frac{  \sin\theta \, \vert \sin(2\theta)\vert^{a-1} }{  |\tilde \xi| ^{a+1}}\,.
\end{align}
We observe that 
$\int_0^{\pi} d\theta [\dots]=\int_0^{\frac \pi 2} d\theta [\dots] +\int_{\frac \pi 2}^{\pi} d\theta [\dots] = 2\int_0^{{\frac \pi 2}} d\theta [\dots]$ once we perform the change of variables $\zeta=\pi-\theta$ which implies $|\sin(2\theta)|=|\sin(2\zeta)|$, $\sin(2\theta)=-\sin(2\zeta)$. Then, renaming the variables $\zeta \to \theta$, we have that \eqref{eq:MassF_gain3} becomes
\begin{align}
\label{eq:MassF_gain4}
&\int_{\R}d\xi \,  Q^+[F,F](\xi,\tau)  \notag \\&
=  \frac {16}{2^a} t^{a-1} \int_{\mathbb{R}^2}d\tilde{\xi} \, \int_{\mathbb{R}} dx_1   \int_0^{\frac{\pi}{2}} d\theta  \,L\left( x_1+  |\tilde \xi|\left(t \sin(2\theta)\right)^{-1}, \tau\right)
L\left(x_1-|\tilde \xi|\left(t \sin(2\theta)\right)^{-1},\tau \right) \frac{  \sin\theta \, \vert \sin(2\theta)\vert^{a-1} }{  |\tilde \xi| ^{a+1}}\nonumber \\&
= \frac {16}{2^a} t^{a-1} \int_{\mathbb{R}^2}d\tilde{\xi} \, \int_{\mathbb{R}} dx_1   \int_0^{{\pi}} \frac 1 2 d\psi  \,L\left( x_1+  \frac{|\tilde \xi|}{\left(t \sin(\psi)\right)}, \tau\right)
L\left(x_1-\frac{|\tilde \xi|}{\left(t \sin(\psi)\right)},\tau \right) \frac{  \sin\left(\frac \psi 2\right) \, \vert \sin(\psi)\vert^{a-1} }{  |\tilde \xi| ^{a+1}} 
\nonumber\\& 
= \frac {8}{2^a} t^{a-1} \int_{\mathbb{R}^2}d\tilde{\xi} \, \int_{\mathbb{R}} dx_1   \int_{0}^{\frac{\pi}2} d\psi  \,L\left( x_1+  \frac{|\tilde \xi|}{\left(t \sin(\psi)\right)}, \tau\right)
L\left(x_1-\frac{|\tilde \xi|}{\left(t \sin(\psi)\right)},\tau \right) \frac{  \left(\sin\left(\frac \psi 2\right)+\cos\left(\frac \psi 2\right)\right)  }{  |\tilde \xi| ^{a+1}}  \, \vert \sin(\psi)\vert^{a-1}
\end{align}
where in the second identity we changed variables and set $2\theta=\psi$ with $d\theta=\frac 1 2  d\psi$. Moreover, in the last identity, we used $\int_0^{{\pi}} d\psi [\dots]=\int_0^{\frac{\pi}2} d\psi [\dots] + \int_{\frac{\pi}2}^{{\pi}} d\psi [\dots]$ and changed the variables in $\int_{\frac{\pi}2}^{{\pi}} d\psi$ setting $\psi=\pi-\zeta$ so that $\int_{\frac{\pi}2}^{{\pi}} d\psi [\dots]=\int_{0}^{\frac{\pi}2} d\zeta [\dots]$ 
and we then renamed $\zeta \to \psi $. We now set $z=\frac{|\tilde \xi|}{\left(t \sin(\psi)\right)}$, i.e. $\sin(\psi)=\frac{|\tilde \xi|}{ t z }$ , with $dz=-\frac{|\tilde \xi|}{\left(t \sin^2(\psi)\right)}\cos(\psi)d\psi$
and from \eqref{eq:MassF_gain4} we arrive at 
\begin{align}
\label{eq:MassF_gain5}
&\int_{\R}d\xi \,  Q^+[F,F](\xi,\tau)  \nonumber \\&  
= \frac {8}{2^a}\frac{1}{t}\int_{\mathbb{R}^2}d\tilde{\xi} \, \int_{\mathbb{R}} dx_1   \int_{\frac{|\tilde \xi|}{t }}^{\infty} dz  \, \frac {L\left( x_1+ z, \tau\right)
L\left(x_1-z,\tau \right) }{z ^{a+1}}  \frac{  \left(\sin\left(\frac \psi 2\right)+\cos\left(\frac \psi 2\right)\right)  }{ \cos \psi}  \frac 1  {|\tilde \xi| }  \nonumber \\&
= \frac {8}{2^a} \frac{2\pi}{t} \int_{0}^{\infty} |\tilde{\xi}|\, d(|\tilde{\xi}|) \, \int_{\mathbb{R}} dx_1   \int_{\frac{|\tilde \xi|}{t }}^{\infty} dz  \, \frac {L\left( x_1+ z, \tau\right)
L\left(x_1-z,\tau \right) }{z ^{a+1}} \frac{  \left(\sin\left(\frac \psi 2\right)+\cos\left(\frac \psi 2\right)\right)  }{ \cos \psi}  \frac 1  {|\tilde \xi| } \nonumber \\&
= \frac {8}{2^a} \frac{2\pi}{t}  \int_{\mathbb{R}} dx_1   \int_{0}^{\infty} dz  \, \frac {L\left( x_1+ z, \tau\right)
L\left(x_1-z,\tau \right) }{z ^{a+1}} \int_{0}^{ zt} d(|\tilde{\xi}|)\frac{  \left(\sin\left(\frac \psi 2\right)+\cos\left(\frac \psi 2\right)\right)  }{ \cos \psi}\nonumber\\&
=\frac {8}{2^a} \frac{2\pi}{t}  \int_{\mathbb{R}} dx_1   \int_{0}^{\infty} dz  \, \frac {L\left( x_1+ z, \tau\right)
L\left(x_1-z,\tau \right) }{z ^{a+1}} \int_{0}^{ 1} tz\,d\rho \frac{  \left(\sin\left(\frac \psi 2\right)+\cos\left(\frac \psi 2\right)\right)  }{ \cos \psi}\nonumber \\& 
=\frac {8}{2^a} {2\pi} \int_{\mathbb{R}} dx_1   \int_{0}^{\infty} dz  \, \frac {L\left( x_1+ z, \tau\right)
L\left(x_1-z,\tau \right) }{z ^{a}} \int_{0}^{ 1} \,d\rho \frac{  \left(\sin\left(\frac \psi 2\right)+\cos\left(\frac \psi 2\right)\right)  }{ \cos \psi}  
\end{align}
where we passed in polar coordinates in $\int_{\mathbb{R}^2} d\tilde{\xi}[\dots]$, used Fubini and later the change of variables $ |\tilde{\xi}|=tz\rho$ with $d(|\tilde{\xi}|)=tz\,d\rho$. 
Setting now $\rho= \cos\psi d\psi$
we finally obtain 
\begin{align}\label{eq:MassF_gain6}
& 
\int_{\R}d\xi \,  Q^+[F,F](\xi,\tau)  \notag \\&  
=\frac {8}{2^a} {2\pi} \int_{\mathbb{R}} dx_1   \int_{0}^{\infty} dz  \, \frac {L\left( x_1+ z, \tau\right)
L\left(x_1-z,\tau \right) }{z ^{a}} \int_{0}^{ \frac{\pi}2} \,d\psi   \left(\sin\left(\frac \psi 2\right)+\cos\left(\frac \psi 2\right)\right) \nonumber\\&= \frac {8}{2^a} {4\pi} \int_{\mathbb{R}} dx_1   \int_{0}^{\infty} dz  \, \frac {L\left( x_1+ z, \tau\right)
L\left(x_1-z,\tau \right) }{|z| ^{a}}  = \frac {8}{2^a} \frac{4\pi} {2}\int_{\mathbb{R}} dx_1   \int_{\mathbb{R}} dz  \, \frac {L\left( x_1+ z, \tau\right)
L\left(x_1-z,\tau \right) }{|z| ^{a}}  
\end{align}
since $\int_{0}^{ \frac{\pi}2} \,d\psi   \left(\sin\left(\frac \psi 2\right)+\cos\left(\frac \psi 2\right)\right)=2$ and we used that $\int_{0}^{\infty} dz [\dots]=\frac 1 2 \int_{\mathbb{R}} dz [\dots]$ . 
We further set $\alpha=x_1+z,\,\beta=x_1-z$ so that $z=\frac{\alpha-\beta}{2}$ and with Jacobian $dx_1dz=\frac 1 2 d\alpha d\beta$. We thus obtain
\begin{align}
\label{eq:MassF_gain7}
& 
\int_{\R}d\xi \,  Q^+[F,F](\xi,\tau)    
=
 \frac {8}{2^a} \frac{4\pi} {2} \frac 1 2 \int_{\mathbb{R}} d\alpha   \int_{\mathbb{R}} d\beta  \, 2^a\,\frac {L\left( \alpha, \tau\right)
L\left(\beta,\tau \right) }{|\alpha-\beta| ^{a}} \nonumber\\& 
=
8\pi \int_{\mathbb{R}} d\alpha   \int_{\mathbb{R}} d\beta  \,\frac {L\left( \alpha, \tau\right)
L\left(\beta,\tau \right) }{|\alpha-\beta| ^{a}}
= 8\pi \int_{\R} d\xi   \int_{\R} d\xi_{\ast}  \,\frac{ F\left( \xi, \tau\right)F\left( \xi_{\ast}, \tau\right)}{|\xi-\xi_{\ast}|^a}
\end{align}
where we used \eqref{eq:defL_app}. The r.h.s of  \eqref{eq:MassF_gain7} coincides with \eqref{eq:MassF_loss} and this implies the mass conservation property of the equation \eqref{eq:ss_strong2}  satisfied by the distribution function $F$. 
 \medskip

\subsection{Justification of the form of the gain and loss term in the reduced two dimensional problem \eqref{eq:G} (cf. \eqref{eq:G_gainbis}-\eqref{eq:G_loss}).}  \label{ss:justificationGeq}

We now justify  \eqref{eq:G_gainbis}-\eqref{eq:G_loss} 
starting from \eqref{eq:ss_strong2} with the full collision  operator given as in \eqref{eq:ss_strong2_gain}.  
We recall that the equation satisfied by $G$ is the integrated version of \eqref{eq:ss_strong2} with respect to the third component of the variable $\xi$, namely $\xi_3$. 
We start considering the loss term, which is  simpler to obtain. We have 
\begin{align}\label{eq:G_loss_comp}
  {K}^-[G,G](\xi)&= \int_{\Rl}  d\xi_3  \, {Q}^-[F,F](\xi_1,\xi_2,\xi_3,\tau) \notag \\& =2 \int_{\Rl}  d\xi_3 \, F(\xi_1,\xi_2,\xi_3,\tau) \int_{S^2} d\omega\int_{\R}d \xi_{1,\ast}d\xi_{2,\ast}d\xi_{3,\ast} \,|\xi_1-\xi_{1,\ast}|^{-a} F(\xi_{1,\ast},\xi_{2,\ast},\xi_{3,\ast},\tau) \notag\\&  =8\pi G(\xi_1,\xi_2,\tau) \int_{\mathbb{R}^2} d \xi_{1,\ast}d\xi_{2,\ast} 
  \,|\xi_1-\xi_{1,\ast}|^{-a} G(\xi_{1,\ast},\xi_{2,\ast},\tau)
  \notag \\& 
  =8\pi G(\xi_1,\xi_2,\tau)  \int_{\mathbb{R}^2} d\eta \,|\xi_1-\eta_{1}|^{-a} G(\eta,\tau)=  G(\xi,\tau)\Lambda(\xi_1,\tau)\, 
\end{align}
which gives exactly \eqref{eq:G_loss}.  
On the other hand, regarding the gain term, we have 
\begin{align}\label{eq:G_gain}
  {K}^+[G,G](\xi) &= \int_{\Rl}d\xi_3 \,   {Q}^+[F,F](\xi,\tau)   \notag\\&  =\frac {16}{2^a} t^{a-1} \delta(\xi_1) \int_{\mathbb{R}} d\xi_3 \int_{\mathbb{R}^2} dx_1dx_2 \int_{\mathbb{R}}d{y_2}\int_0^{\pi}d\theta \, G\left( x_1+   |\tilde \xi|\left(t \sin(2\theta)\right)^{-1}, x_2+y_2, \tau\right) \times \notag \\& \qquad \qquad \qquad \qquad \qquad \times G\left(x_1-|\tilde \xi|\left(t \sin(2\theta)\right)^{-1},x_2-y_2,\tau \right) \frac{ \sin(\theta)|\sin(2\theta)|^{a-1}}{ |\tilde\xi|^{a+1}} 
 \end{align}
 with $\tilde \xi=(\xi_2,\xi_3)\in \mathbb{R}^2$ and where we used the  change of variables  $x_3+y_3=\eta_1$ and $x_3-y_3=\eta_2$ and the fact that the jacobian $dx_1dx_2=\frac 1 2 d\eta_1d\eta_2$.  Our goal is to obtain \eqref{eq:G_gainbis} starting  
from \eqref{eq:G_gain}.   
We look at \eqref{eq:G_gain} and separate the contributions of the angular integration in 
$
\int_0^\pi [\dots ] d\theta =\int_0^{\frac \pi 2} [\dots ] d\theta +\int_{\frac \pi 2}^\pi  [\dots ] d\theta $.  We will then have $${K}^+[G,G](\xi) := I_1+I_2 .$$
Performing the change of variables 
$\theta=\pi-\zeta$ in $I_2$, so that $\sin(2\theta)=-\sin(2\zeta)$ and $\sin\theta=\sin\zeta$ we then arrive at
\begin{align*}
 I_2 =
  \frac {16}{2^a} \, t^{a-1}\, \delta(\xi_1) \int_{\mathbb{R}} d\xi_3 \int_{\mathbb{R}^2} dx_1dx_2 \int_{\mathbb{R}}d{y_2}  &
  \Bigg[
  \int_0^{\frac{\pi}{2}}d\zeta \, G\left( x_1-  \frac{|\tilde \xi|}{\left(t \sin(2\zeta)\right)}, x_2+y_2, \tau\right) \times \notag \\& \ \ \times G\left(x_1+\frac{|\tilde \xi|}{\left(t \sin(2\zeta)\right)},x_2-y_2,\tau \right) \frac{ \sin(\zeta)|\sin(2\zeta)|^{a-1}}{|\tilde\xi|^{a+1}}
  \Bigg] \notag  
 \end{align*}
 We observe that changing $y_2\to -y_2$ in $I_2$  we obtain $I_2= I_1$ and then
 \begin{align}\label{eq:G_gain2}
  {K}^+[G,G](\xi)&=\frac {32}{2^a}  \, t^{a-1}\,  \delta(\xi_1) \int_{\mathbb{R}} d\xi_3 \int_{\mathbb{R}^2} dx_1dx_2 \int_{\mathbb{R}}d{y_2}  
  \Bigg[
  \int_0^{\frac{\pi}{2}}d\theta \, G\left( x_1+  \frac{|\tilde \xi|}{\left(t \sin(2\theta)\right)}, x_2+y_2, \tau\right) \times \notag \\& \ \ \times G\left(x_1-\frac{|\tilde \xi|}{\left(t \sin(2\theta)\right)},x_2-y_2,\tau \right) \frac{ \sin(\theta)|\sin(2\theta)|^{a-1}}{ |\tilde\xi|^{a+1}}
  \Bigg] \notag \\&
  =\frac {16}{2^a}  \, t^{a-1}\, \delta(\xi_1) \int_{\mathbb{R}} d\xi_3 \int_{\mathbb{R}^2} dx_1dx_2 \int_{\mathbb{R}}d{y_2}  
  \Bigg[
  \int_0^{\pi} d\psi \, G\left( x_1+  \frac{|\tilde \xi|}{\left(t \sin(\psi)\right)}, x_2+y_2, \tau\right) \times \notag \\& \ \ \times G\left(x_1-\frac{|\tilde \xi|}{\left(t \sin(\psi)\right)},x_2-y_2,\tau \right) \frac{ \sin(\frac \psi 2)|\sin(\psi)|^{a-1}}{|\tilde\xi|^{a+1}}
  \Bigg] \notag 
 \end{align}
 where in the second equality we have set $\psi=2\theta$. Using the same strategy as above, namely, separating $\int_0^{\pi} [\dots]d\psi =\int_0^{\frac \pi 2}[\dots] d\psi +\int_{\frac \pi 2}^{\pi} [\dots] d\psi $ so that ${K}^+[G,G]=J_1+J_2$ and setting $\zeta=\pi-\psi$ in $\int_{\frac \pi 2}^{\pi} [\dots] d\psi$ appearing in $J_2$, we can then rewrite ${K}^+[G,G]$ as 
\begin{align*} 
  {K}^+[G,G](\xi)&=\frac {16}{2^a} \, t^{a-1}\,  \delta(\xi_1) \int_{\mathbb{R}} d\xi_3 \int_{\mathbb{R}^2} dx_1dx_2 \int_{\mathbb{R}}d{y_2}  
  \Bigg[
  \int_0^{\frac \pi 2 } d\psi \, G\left( x_1+  \frac{|\tilde \xi|}{\left(t \sin(\psi)\right)}, x_2+y_2, \tau\right) \times \notag \\& \ \ \times G\left(x_1-\frac{ |\tilde \xi|}{\left(t \sin(\psi)\right)},x_2-y_2,\tau \right) \frac{ \Bigg(\sin(\frac \psi 2)+\cos(\frac \psi 2)\Bigg)|\sin(\psi)|^{a-1}}{  |\tilde\xi|^{a+1}}
  \Bigg] \notag .
  \end{align*} 
  We now perform the change of variables 
  \begin{equation}
     \label{eq:ChangeVar_zpsi}z=\frac{|\tilde\xi|}{t\sin(\psi)} \quad \text{with } \quad dz=- \frac{|\tilde\xi|}{t\sin^2(\psi)}\cos(\psi) d\psi, \quad \psi\in \left[0,\frac{\pi}{2}\right)  
     \end{equation} 
     and obtain, using the symmetry of the integral on $\xi_3$, 
 \begin{align*} 
  {K}^+[G,G](\xi)&= 
  \frac {32}{2^a} t^{a} \delta(\xi_1) \int_{0}^{+\infty} d\xi_3 \int_{\mathbb{R}^2} dx_1dx_2 \int_{\mathbb{R}}d{y_2}  
  \Bigg[
  \int_{\frac{|\tilde{\xi}|}{t}}^{+\infty } dz \, G\left( x_1+ z, x_2+y_2, \tau\right) \times \notag \\& \ \ \times G\left(x_1- z,x_2-y_2,\tau \right) \frac{ \Bigg(\sin(\frac \psi 2)+\cos(\frac \psi 2)\Bigg)}{ |\tilde\xi|^{a+2} } \frac{|\sin(\psi)|^{a+1}}{\cos(\psi)}
  \Bigg] \notag \,.
  \end{align*}
Recalling that $\tilde{\xi}=(\xi_2,\xi_3)$ and then $|\tilde{\xi}|=\sqrt{\xi_2^2 + \xi_3^2} $ and using Fubini we can rewrite
  $$ 
\int_{0}^{+\infty} d\xi_3 \int_{\frac{\sqrt{\xi_2^2 + \xi_3^2}}{t}}^{+\infty} dz  [\dots] =\int_{z=\frac{ |\xi_2|}{t}}^{+\infty} dz \int_{\xi_3=0}^{\sqrt{\left({z t}\right) ^2 - \xi_2^2}} d\xi_3 [\dots]
$$
and arrive at 
\begin{align}\label{eq:G_gain_4}
  {K}^+[G,G](\xi)&= 
  \frac {32}{2^a} t^{a} \delta(\xi_1)  \int_{\mathbb{R}^2} dx_1dx_2 \int_{\mathbb{R}}d{y_2}  \int_{\frac{|\xi_2|}{t}}^{+\infty} dz \, G\left( x_1+ z, x_2+y_2, \tau\right) G\left(x_1- z,x_2-y_2,\tau \right) \notag \\& \ \ 
  \quad \times \int_{ 0}^{\sqrt{\left({z t}\right) ^2 - \xi_2^2}} d\xi_3  \frac{ \Big(\sin(\frac \psi 2)+\cos(\frac \psi 2)\Big)}{ |\tilde\xi|^{a+2} } \frac{|\sin(\psi)|^{a+1}}{\sqrt{1-\sin^2(\psi)}}\,
  \end{align}
  where the relation between $\psi$ and $\xi_3$ is given as in \eqref{eq:ChangeVar_zpsi}. 
  We now look at 
\begin{equation}\label{def:Jint}
      J:= \int_{ 0}^{\sqrt{\left( {z t} \right) ^2 - \xi_2^2}} d\xi_3  \frac{  \Big(\sin(\frac \psi 2)+\cos(\frac \psi 2)\Big)}{  |\tilde\xi|^{a+2} } \frac{|\sin(\psi)|^{a+1}}{\sqrt{1-\sin^2(\psi)}}\,
  \end{equation}
  that, using simple trigonometric arguments in terms of bisection formulas, becomes 
  \begin{align}\label{def:Jint}
      J &= \int_{ 0}^{\sqrt{\left( {z t}\right) ^2 - \xi_2^2}} d\xi_3  \frac{|\sin(\psi)|^{a+1}}{\sqrt{1-\sin^2(\psi)}} 
      \frac{   \sqrt{\frac{1 - \sqrt{1 - \sin^2 \psi}}{2}} + \sqrt{\frac{1 + \sqrt{1 - \sin^2 \psi}}{2}}}{  |\tilde\xi|^{a+2} }\notag \\&
      =\frac{1}{\sqrt{2}}\int_{ 0}^{\sqrt{\left({z t}\right)^2 - \xi_2^2}} d\xi_3  \frac{\left(\frac{ |\tilde\xi|}{tz}\right)^{a+1}}{\sqrt{1-\left(\frac{|\tilde\xi|}{tz}\right)^2}} 
      \frac{   \sqrt{1 - \sqrt{1 - \left(\frac{|\tilde\xi|}{tz}\right)^2} } + \sqrt{1 + \sqrt{1 - \left(\frac{|\tilde\xi|}{tz}\right)^2}}}{  |\tilde\xi|^{a+2} } 
\end{align}

We now perform another change of variables. We set $\rho= \frac{|\tilde\xi|}{tz}$. Using that $|\tilde\xi|=\sqrt{\xi_2^2+\xi_3^2}$ we can express $\xi_3$ in terms of $\rho$, i.e. $\xi_3= \sqrt{\left({tz\rho}\right)^2-\xi_2^2}$, and thus $d\xi_3= \frac{\left({tz}\right)^2 \rho \, }{\sqrt{\left({tz\rho}\right)^2-\xi_2^2}}d\rho $. Since $ |\tilde\xi|^{a+2}=  (tz\rho)^{a+2}$ it then follows that
\begin{align}\label{def:Jint2}
      J= & 
      \frac{1}{\sqrt{2}}
      \int_{\frac{|\xi_2|}{tz} }^{1}   \frac{\rho^{a+1}}{\sqrt{1-\rho^2}} 
      \frac{  \sqrt{1 - \sqrt{1 - \rho^2} } + \sqrt{1 + \sqrt{1 - \rho^2}}}{ (tz\rho)^{a+2}} \frac{\left({tz}\right)^2 \rho}{\left({tz}\right)\sqrt{\rho^2-\left(\frac{|\xi_2|}{tz}\right)^2}} d\rho \notag \\&
    = \frac{\sqrt{2}}{2}\int_{ \frac{|\xi_2|}{tz}}^{1}   \frac{\rho^{a+2}}{\sqrt{1-\rho^2}} 
      \frac{  \sqrt{1 - \sqrt{1 - \rho^2} } + \sqrt{1 + \sqrt{1 - \rho^2}}}{ (tz\rho)^{a+2}} \frac{\left(tz \right)}{ \sqrt{\rho^2-\left(\frac{|\xi_2|}{tz}\right)^2}} d\rho \notag \\&
  = \frac{\sqrt{2}}{2\, (tz)^{a+1}}\int_{ \frac{|\xi_2|}{tz}}^{1}   \frac{1}{\sqrt{1-\rho^2}} 
      \frac{  \sqrt{1 - \sqrt{1 - \rho^2} } + \sqrt{1 + \sqrt{1 - \rho^2}}} { \sqrt{\rho^2-\left(\frac{|\xi_2|}{tz}\right)^2}} d\rho   = \frac{\sqrt{2}}{2\,(tz)^{a+1}} \Phi\left(\frac{|\xi_2|}{tz}\right)
\end{align}
where we introduced the auxiliary function $\Phi$, with self-similar structure, defined as in \eqref{def:auxPhi}, i.e.
\begin{equation*}
     \Phi\left(s\right)= \int_{ s}^{1}   \frac{1}{\sqrt{1-\rho^2}} 
      \frac{  \sqrt{1 - \sqrt{1 - \rho^2} } + \sqrt{1 + \sqrt{1 - \rho^2}}} { \sqrt{\rho^2-s^2}} d\rho  , \quad 0<s<1  \,.
\end{equation*}
Using \eqref{def:Jint2} in \eqref{eq:G_gain_4} the gain term becomes
\begin{align}\label{eq:G_gain_5}
  {K}^+[G,G](\xi)&= 
  \frac {16 \sqrt{2}}{2^a} t^{a} \delta(\xi_1)  \int_{\mathbb{R}^2} dx_1dx_2 \int_{\mathbb{R}}d{y_2}  \int_{\frac{ |\xi_2|}{t}}^{+\infty} dz \, G\left( x_1+ z, x_2+y_2, \tau\right) G\left(x_1- z,x_2-y_2,\tau \right) 
  \frac{1}{(tz)^{a+1}}\Phi\left(\frac{|\xi_2|}{tz}\right) \notag \\&
  =\frac {16 \sqrt{2}}{2^a \, t}  \delta(\xi_1)  \int_{\mathbb{R}^2} dx_1dx_2 \int_{\mathbb{R}}d{y_2}  \int_{\frac{|\xi_2|}{t}}^{+\infty} dz \, G\left( x_1+ z, x_2+y_2, \tau\right) G\left(x_1- z,x_2-y_2,\tau \right) \frac{1}{z^{a+1}}\Phi\left(\frac{|\xi_2|}{tz}\right) \,. 
  \end{align}
  Setting now $\eta_2=x_2+y_2$ and $\zeta_2=x_2-y_2$, that implies $x_2=\frac 1 2 (\eta_2+\zeta_2)$ $y_2=\frac 1 2 (\eta_2-\zeta_2)$ and with Jacobian $dx_2\,dy_2=\frac 1 2 d\eta_2\,d\zeta_2$, we rewrite \eqref{eq:G_gain_5} as 
  \begin{align*}
  {K}^+[G,G](\xi)& 
 =\frac {8\sqrt{2} }{2^a \, t}\delta(\xi_1)  \int_{\mathbb{R}} dx_1 \int_{\mathbb{R}} d\eta_2 \int_{\mathbb{R}}d\zeta_2 \int_{\frac{|\xi_2|}{t}}^{+\infty} dz \, G\left( x_1+ z, \eta_2, \tau\right) G\left(x_1- z,\zeta_2,\tau \right) \frac{1}{z^{a+1}}\Phi\left(\frac{|\xi_2|}{tz}\right)  
  \end{align*}
which coincides with \eqref{eq:G_gainbis}. We thus proved the equivalence of \eqref{eq:G_gain} and \eqref{eq:G_gainbis}.

\subsection{Conservation of mass for the equation satisfied by the reduced distribution $G$}
\label{ssec:massG}

We consider the equation satisfied by $G$, namely
\begin{equation}\label{eq:G_2}
    \partial_{\tau}G-\frac{1}{a}\partial_{\xi_1}(\xi_1 G)-\left(\frac{1}{a}-1\right)\partial_{\xi_2}(\xi_2 G)
   +  \partial_{\xi_1}(\xi_2 G )=  {K}^+[G,G](\xi)- {K}^-[G,G](\xi), \quad G=G(\xi,\tau), \quad \xi\in \mathbb{R}^2
\end{equation}
with ${K}^+[G,G]$ given as in \eqref{eq:G_gainbis} and ${K}^-[G,G]$ given as in \eqref{eq:G_loss} and prove that this equation satisfies the conservation of mass property. 
Integrating the left hand side of \eqref{eq:G_2} with respect to $\xi\in \mathbb{R}^2$ we obtain zero, therefore it remains to show that 
$\int_{\mathbb{R}^2} {K}^+[G,G](\xi) \, d\xi=\int_{\mathbb{R}^2} {K}^-[G,G](\xi) \, d\xi$. 
Integrating the loss term \eqref{eq:G_loss} we have
\begin{equation}
\label{eq:massGloss}
\int_{\mathbb{R}^2} {K}^-[G,G](\xi) \, d\xi= 8\pi \int_{\mathbb{R}^2} d\xi  \int_{\mathbb{R}^2} d\eta \,|\xi_1-\xi_{1,\ast}|^{-a} \, G(\xi,\tau) G(\eta,\tau)\, . 
\end{equation}
The integral of the gain term \eqref{eq:G_gainbis} is 
 \begin{align}
\label{eq:massGgain} 
& \int_{\mathbb{R}^2} {K}^+[G,G](\xi) \, d\xi\notag \\& \quad = \frac {8  {\sqrt{2}} }{2^a \, t} \int_{\mathbb{R}} d\xi_1 \delta(\xi_1)  \int_{\mathbb{R}} d\xi_2 \int_{\mathbb{R}} dx_1 \int_{\mathbb{R}} d\eta_2 \int_{\mathbb{R}}d\zeta_2 \int_{\frac{ |\xi_2|}{t}}^{+\infty} dz \, G\left( x_1+ z, \eta_2, \tau\right) G\left(x_1- z,\zeta_2,\tau \right) \frac{1}{z^{a+1}}\Phi\left(\frac{|\xi_2|}{tz}\right)\notag \\&
\quad 
=\frac {16 \sqrt{2} }{2^a \, t}  \int_{0}^\infty d\xi_2 \int_{\mathbb{R}} dx_1 \int_{\mathbb{R}} d\eta_2 \int_{\mathbb{R}}d\zeta_2 \int_{\frac{ |\xi_2|}{t}}^{+\infty} dz \, G\left( x_1+ z, \eta_2, \tau\right) G\left(x_1- z,\zeta_2,\tau \right) \frac{1}{z^{a+1}}\Phi\left(\frac{|\xi_2|}{tz}\right)
 \end{align}
 where we used that $\int_{\mathbb{R}} d\xi_2=2 \int_{0}^\infty d\xi_2$ by symmetry. Applying Fubini, we can exchange the order of integration  $\int_{0}^{\infty} d\xi_2 \int_{\frac{|\xi_2|}{t}}^{+\infty} dz [\dots]=\int_{0}^\infty dz \int_{0}^{{tz}} d\xi_2 [\dots]$ so that \eqref{eq:massGgain} can be rewritten as 
\begin{align}
\label{eq:massGgain2} 
&  \int_{\mathbb{R}^2} {K}^+[G,G](\xi) \, d\xi = \\&
\quad =\frac {16 \sqrt{2} }{2^a \, t}  \int_{0}^\infty dz \int_{\mathbb{R}} dx_1 \int_{\mathbb{R}} d\eta_2 \int_{\mathbb{R}}d\zeta_2 \int_{0}^{ {tz} } d\xi_2  \, G\left( x_1+ z, \eta_2, \tau\right) G\left(x_1- z,\zeta_2,\tau \right) \frac{1}{z^{a+1}}\Phi\left(\frac{ \xi_2}{tz}\right)\,,
 \end{align}
 where we used the fact that $\xi_2>0$ in the region of integration. 
We now consider  
$$\int_{0}^{tz} d\xi_2 \Phi\left(\frac{\xi_2}{tz}\right)= {tz}\int_0^1 ds\, \Phi\left(s\right)$$
where we used the change of variable $s=\frac{\xi_2}{tz}$. 
Recalling the definition of the function $\Phi$ (cf. \eqref{def:auxPhi}) we have 
\begin{align*}
    \int_0^1 ds\, \Phi\left(s\right)&=\int_0^1 ds\, \int_{ s}^{1}   \frac{1}{\sqrt{1-\rho^2}} 
      \frac{  \sqrt{1 - \sqrt{1 - \rho^2} } + \sqrt{1 + \sqrt{1 - \rho^2}}} { \sqrt{\rho^2-s^2}} d\rho \notag\\&
      =\int_0^1 d\rho \, \frac{1}{\sqrt{1-\rho^2}} 
      {\left[  \sqrt{1 - \sqrt{1 - \rho^2} } + \sqrt{1 + \sqrt{1 - \rho^2}} \right]} \int_{ 0}^{\rho} ds\, \frac 1 { \sqrt{\rho^2-s^2}}  \notag \\& 
      = \frac \pi 2 \int_0^1 d\rho \, \frac{1}{\sqrt{1-\rho^2}} 
      {\left[  \sqrt{1 - \sqrt{1 - \rho^2} } + \sqrt{1 + \sqrt{1 - \rho^2}}\right]} \notag \\& 
       = \frac \pi 2 \int_0^{\frac \pi 2} d\vp \, \left[  
      {  \sqrt{1 - \cos(\vp) } + \sqrt{1 + \cos(\vp)}} \right] \notag \\& =
      \frac {\sqrt{2}\pi} 2 \int_0^{\frac \pi 2} d\vp \, \left[  
    \sin \left(\frac \vp 2\right) +  \cos\left(\frac \vp 2\right) \right] 
    = \sqrt{2} \pi
\end{align*}
where in the fourth identity we used the change of variables $\rho=\sin(\vp)$ and then applied the trigonometric identities. Therefore,
\begin{align}\label{eq:intPhi} 
\int_{0}^{ tz } d\xi_2\, \Phi\left(\frac{|\xi_2|}{tz}\right)
= {\sqrt{2} \pi} \, tz\, .
\end{align}
Using \eqref{eq:intPhi} in \eqref{eq:massGgain2} we finally arrive at 
\begin{align}
\label{eq:massGgain3} 
&  \int_{\mathbb{R}^2} {K}^+[G,G](\xi) \, d\xi = \notag \\&
 =\frac {16 \sqrt{2}} {2^a \, t}  \int_{0}^\infty dz \int_{\mathbb{R}} dx_1 \int_{\mathbb{R}} d\eta_2 \int_{\mathbb{R}}d\zeta_2  \, G\left( x_1+ z, \eta_2, \tau\right) G\left(x_1- z,\zeta_2,\tau \right) \frac{1}{z^{a+1}} {\sqrt{2} \pi} \, tz \,\notag \\&
\quad =
\frac {16 \cdot {2\pi} }{2^a}  \int_{0}^\infty dz \int_{\mathbb{R}} dx_1 \int_{\mathbb{R}} d\eta_2 \int_{\mathbb{R}}d\zeta_2  \, \frac{1}{z^{a}} \, G\left( x_1+ z, \eta_2, \tau\right) G\left(x_1- z,\zeta_2,\tau \right)  \, \notag \\& 
\quad =\frac {16 \cdot {2\pi} }{2^a} \frac 1 2 \int_{\mathbb{R}} dz \int_{\mathbb{R}} dx_1 \int_{\mathbb{R}} d\eta_2 \int_{\mathbb{R}}d\zeta_2  \, \frac{1}{|z|^{a}} \, G\left( x_1+ z, \eta_2, \tau\right) G\left(x_1- z,\zeta_2,\tau \right)\notag \\&
\quad =\frac {16 \cdot {2\pi} }{2^a} \frac 1 4 \int_{\mathbb{R}} d\eta_1 \int_{\mathbb{R}} d\eta_2 \int_{\mathbb{R}} d\zeta_1 \int_{\mathbb{R}}d\zeta_2  \, \frac{2^{a}}{| \eta_1-\zeta_1|^{a}} \, G\left( \eta_1, \eta_2, \tau\right) G\left(\zeta_1,\zeta_2,\tau \right)\notag \\& 
\quad = 8 \pi \int_{\mathbb{R}^2} d\eta \int_{\mathbb{R}} d\zeta \, | \eta_1-\zeta_1|^{-a}\,G(\eta, \tau)G(\zeta, \tau) =    \int_{\mathbb{R}^2} {K}^-[G,G](\zeta) \, d\zeta \, .
 \end{align}
 where in the forth identity we used the change of variables $x_1+z=\eta_1$, $x_1-z=\zeta_1$ so that $z=\frac {\eta_1-\zeta_1}{2}$ with $dx_1dz=\frac 1 2 d\eta_1d\zeta_1$. The computation above hence shows the mass conservation property of $G$, as expected.

\bigskip

\bigskip

\noindent\textbf{Acknowledgements.}  
J.~J.~L.~Vel\'azquez 
gratefully acknowledges the support by the Deutsche Forschungsgemeinschaft (DFG) through the collaborative research centre ``Analysis of criticality: from complex phenomena to models and estimates" (CRC 1720, Project-ID 539309657)  and Germany’s Excellence Strategy -EXC2047/2-390685813 funded by DFG.

\noindent The funders had no role in study design, analysis, decision to publish, or
preparation of the manuscript.

\bigskip

\adresse

\end{document}